\documentclass[hidelinks,onefignum,onetabnum]{siamart220329}

\usepackage{lipsum}
\usepackage{amsfonts}
\usepackage{graphicx}
\usepackage{epstopdf}
\usepackage{algorithmic}
\ifpdf
  \DeclareGraphicsExtensions{.eps,.pdf,.png,.jpg}
\else
  \DeclareGraphicsExtensions{.eps}
\fi

\newsiamremark{remark}{Remark}
\newsiamremark{hypothesis}{Hypothesis}
\crefname{hypothesis}{Hypothesis}{Hypotheses}
\newsiamthm{claim}{Claim}
\headers{Low Rank Convex Clustering}{Meixia Lin, and Yangjing Zhang}
\title{Low Rank Convex Clustering For Matrix-Valued Observations}

\author{Meixia Lin\thanks{Engineering Systems and Design, Singapore University of Technology and Design
  (\email{meixia\_lin@sutd.edu.sg}).}
\and Yangjing Zhang\thanks{Institute of Applied Mathematics, Academy of Mathematics and Systems Science, Chinese Academy of Sciences
  (\email{yangjing.zhang@amss.ac.cn}).}
}

\usepackage{amsopn}
\usepackage{amssymb}
\usepackage{amsmath}
\usepackage{graphicx}
\usepackage{subfigure}
\allowdisplaybreaks[4]
\usepackage{cite}
\usepackage{booktabs}
\usepackage{enumitem}
\usepackage{aligned-overset}

\makeatletter
\newcommand{\leqnomode}{\tagsleft@true\let\veqno\@@leqno}
\makeatother
\graphicspath{{figures/}}

\ifpdf
\hypersetup{
  pdftitle={Low Rank Convex Clustering For Matrix-Valued Observations},
  pdfauthor={Meixia Lin, and Yangjing Zhang}
}
\fi

\begin{document}

\maketitle

\begin{abstract}
Common clustering methods, such as $k$-means and convex clustering, group similar vector-valued observations into clusters. However, with the increasing prevalence of matrix-valued observations, which often exhibit low rank characteristics, there is a growing need for specialized clustering techniques for these data types. In this paper, we propose a low rank convex clustering model tailored for matrix-valued observations. Our approach extends the convex  clustering model originally designed for vector-valued data to classify matrix-valued observations. Additionally, it serves as  a convex relaxation of the low rank $k$-means method proposed by Z. Lyu, and D. Xia (arXiv:2207.04600). Theoretically, we establish exact cluster recovery for finite samples and asymptotic cluster recovery as the sample size approaches infinity. We also give a finite sample bound on prediction error in terms of centroid estimation,  and further establish the prediction consistency. To make the model practically useful, we develop an efficient double-loop algorithm for solving it. Extensive numerical experiments are conducted to show the effectiveness of our proposed model.
\end{abstract}



\section{Introduction}
Clustering is an essential task in the realm of unsupervised learning, with applications across many domains. One class of clustering methods formulates the task as a nonconvex optimization problem, with $k$-means clustering being a notable example. Although $k$-means is straightforward to implement (often implemented via Lloyd's algorithm \cite{lloyd1982least}), it is known for its sensitivity to initialization, which can affect the stability of clustering outcomes.  An alternative approach involves solving convex relaxations of nonconvex clustering problems. A prominent method in this category is convex clustering \cite{pelckmans2005convex,lindsten2011clustering,hocking2011clusterpath}, also known as sum-of-norms clustering or clusterpath clustering. Convex clustering methods have become popular for their stability and favorable theoretical properties. Numerous studies have explored exact recovery guarantees, statistical properties, and optimization algorithms for convex clustering and its variants \cite{zhu2014convex,chi2015splitting,tan2015statistical,panahi2017clustering,wang2018sparse,jiang2020recovery,sun2021convex,chakraborty2023biconvex}. However, these works primarily focus on vector-valued data, grouping similar vectors into clusters. Meanwhile, matrix-valued observations are increasingly common in fields such as image processing and spatial-temporal analysis, where data naturally takes the form of matrices. Consequently, developing methods for clustering matrix-valued observations has become essential in these applications.

A straightforward method for clustering matrix-valued observations is to first convert each matrix into a long vector and then apply clustering methods designed for vector-valued data. A typical example is clustering handwritten digit images by methods like $k$-means and convex clustering. However, vectorizing matrices overlooks structural information inherent in the matrix form, such as row-wise and column-wise dependencies and the potential for low rank structure. In fact, many matrix-valued data sets are either low rank or can be well approximated by low rank matrices, especially in high dimensional settings \cite{lyu2022optimal,lyu2023optimal}. Motivated by this, we propose a method that directly leverages the matrix structure, avoiding the need for vectorization, and incorporates a regularization term to promote low rank structure in the data.

To this end, we propose the low rank convex clustering (lrCC) model for clustering $n$ matrix-valued observations $A_1,\dots,A_n$ with $d_1 \times d_2$ features:
\begin{equation}\label{model-convex}
  \min_{X_1,\dots,X_n\in \mathbb{R}^{d_1\times d_2}}\ \frac{1}{2}\sum_{i=1}^{n} \|X_i-A_i\|_F^2 + \gamma_1 \sum_{l(i,j)\in \mathcal{E}} w_{ij} \|X_i-X_j\|_F + \gamma_2 \sum_{i=1}^{n}\|X_i\|_*.
\end{equation}
Here, the penalty parameters $\gamma_1 > 0$ and $\gamma_2 > 0$, along with the edge set $\mathcal{E}$ and the positive edge weights $\{w_{ij}\}_{l(i,j)\in \mathcal{E}}$, are given. The notation $\|\cdot\|_F$ and $\|\cdot\|_*$ denote the Frobenius and nuclear norms, respectively. We set $w_{ij}=0$ if $l(i,j)\notin \mathcal{E}$. Without loss of generality, we assume $d_1 \geq d_2$. Problem \eqref{model-convex} is strongly convex, with a unique optimal solution denoted as $(X_1^*, \ldots, X_n^*)$. We refer to $X_i^*$ as the centroid of observation $A_i$. Observations $A_i$ and $A_j$ are considered to belong to the same cluster if and only if they share the same centroid, i.e., $X_i^* = X_j^*$. The terms in \eqref{model-convex} serve distinct purposes: the first term keeps centroids close to their corresponding observations, the second term encourages identical centroids for linked observations (linked by the edges $l(i,j)\in \mathcal{E}$), and the third term promotes low rank structures in the centroids. We can interpret the lrCC model \eqref{model-convex} from two perspectives. First, it extends the convex clustering model (e.g., \cite[(2)]{lindsten2011clustering}, \cite[(4)]{hocking2011clusterpath}, \cite[(1.1)]{chi2015splitting}, and \cite[(2)]{sun2021convex}) designed for vector-valued data to classify matrix-valued observations, accommodating the possibility of low rank centroids. The additional regularization term $\|\cdot\|_*$ for promoting low-rankness on each centroid introduces non-trivial challenges in both algorithmic design and theoretical analysis. Notably, if no low rank structure is imposed and $\gamma_2 = 0$, \eqref{model-convex} reduces to the original vector based model in dimension $d:=d_1d_2$. Second, the recent works \cite{lyu2022optimal,lyu2023optimal} assume that matrix-valued observations within the same cluster share a common low rank expectation and proposes a low rank variant of the classic $k$-means method. Specifically, it adapts the classic $k$-means approach by adding a low rank approximation step during the cluster center updates \cite[Algorithm~1]{lyu2022optimal}. As mentioned earlier, the lrCC model similarly assumes that the centroids of matrix-valued observations are low rank, aligning in spirit with the low rank assumption in \cite{lyu2022optimal,lyu2023optimal}. This shared emphasis on low rank structures highlights the importance of leveraging matrix-specific properties for clustering.  Furthermore, the lrCC model can be viewed as a convex relaxation of their low rank $k$-means method \cite[Algorithm~1]{lyu2022optimal}. Notably, convex clustering has also been interpreted by some researchers as a convex relaxation of $k$-means \cite{lindsten2011just,tan2015statistical}.

The theoretical properties of convex clustering can be categorized into two main aspects: (exact or asymptotic) cluster recovery, and prediction error bounds for centroid estimation. Exact cluster recovery refers to the ability to perfectly identify the true clusters. The first result on exact cluster recovery was given in \cite{zhu2014convex}, showing that the convex clustering can perfectly recover two clusters, each of which is a cube, under specific conditions. This property was later generalized in \cite{panahi2017clustering} to an arbitrary number of clusters, though their analysis was limited to uniformly weighted fully connected graphs  (i.e., $w_{ij}=1$ for all $i<j$). Subsequently, the work \cite{sun2021convex} extended the exact recovery property to more general weighted graphs with non-negative weights ($w_{ij} \geq 0$), subject to certain connectivity conditions. Notably, these exact cluster recovery results apply only to finite sample sizes ($n < \infty$), as perfect classification becomes unattainable when the sample size grows infinitely large ($n\to \infty$). This limitation has been addressed by asymptotic cluster recovery in large sample settings. In particular, the paper \cite{jiang2020recovery} studied convex clustering for data generated from a Gaussian mixture model as $n\to \infty$. They showed that convex clustering can correctly classify samples lying within a fixed distance of their cluster means with high probability, a property referred to as asymptotic cluster recovery. On the other hand, the prediction error bounds for centroid estimation have been explored in \cite[Section~3.2]{tan2015statistical} and \cite[Section~3]{wang2018sparse}.  These works established the prediction consistency of convex clustering and its variants, demonstrating that under specific conditions, the estimated centroids converge to their true values as the sample size increases.

Our work represents an early effort to extend the vector based convex clustering model to the matrix setting, accommodating the possibility of low rank centroids. However, the theoretical properties mentioned in the previous paragraph are specific to the vector based convex clustering model and do not directly apply to the lrCC model \eqref{model-convex}. To bridge this gap, we aim to establish analogous theoretical properties for the lrCC model, overcoming the non-trivial challenges posed by the nuclear norm regularization. Specifically, we provide results on exact and asymptotic cluster recovery in Section~\ref{sec:the_gua} and derive prediction error bounds in Section~\ref{sec:bound_pre_err}. Additionally, we have not noticed efficient optimization methods specifically designed for solving \eqref{model-convex}, except for some possible general optimization solvers. Inspired by the optimization method for solving the vector based convex clustering model in \cite[Section~5]{sun2021convex}, we design an efficient double-loop algorithm for solving \eqref{model-convex} in Section~\ref{sec:opt_method}. We summarize the main contributions of the paper in three-fold: (i) We establish the exact cluster recovery for finite samples  and the asymptotic cluster recovery as the sample size approaches infinity. This bridges the gap that existing  theoretical properties of convex clustering are mainly for vector based models. (ii) We give a finite sample bound on prediction error in terms of centroid estimation, showing that the centroids estimated by the lrCC model will converge to the truth as the sample size goes to infinity. (iii) We design an efficient double-loop algorithm for solving the lrCC model, making the model practically useful for clustering tasks.

The rest of the paper is outlined below. Theoretical guarantees for cluster recovery of the lrCC model \eqref{model-convex} are established in Section~\ref{sec:the_gua}. The bounds on the prediction error of centroids estimated by \eqref{model-convex} are given in Section~\ref{sec:bound_pre_err}. In Section~\ref{sec:opt_method}, we design an efficient double-loop algorithm for solving \eqref{model-convex}. Numerical experiments are presented in Section~\ref{sec:num_exp}, and we conclude the paper in Section~\ref{sec:conclu}.

\vspace{0.2cm}
\noindent {\bf Notation.} For any matrix $Z\in \mathbb{R}^{m\times n}$, let $\sigma_{\min}(Z)$ denote the minimum nonzero singular value of $Z$, $\sigma_{\max}(Z)$ denote the maximum singular value of $Z$, and $Z^{\dagger}$ denote the Moore-Penrose pseudo-inverse of $Z$. Let $[n]=\{1,2,\cdots,n\}$. We say $\mathcal{I}_1,\dots,\mathcal{I}_G$ is a partition of $[n]$ if $\cup_{i=1}^G \mathcal{I}_i = [n]$ and $\mathcal{I}_i \cap \mathcal{I}_j = \emptyset$ for $i\neq j$. $\|\cdot\|$ denotes the $\ell_2$ norm of a vector; $\|\cdot\|_F$ denotes the Frobenius norm of a matrix; $\|\cdot\|_*$ denotes the nuclear norm of a matrix; $\|\cdot\|_2$ denotes the spectral norm of a matrix. $o(\cdot)$ and $O(\cdot)$ stand for the standard small O and big O notation, respectively. ${\rm vec} (Z)$ denotes the column vector obtained by stacking the columns of the matrix $Z$ on top of one another. For a vector $x\in\mathbb{R}^{mn}$, we use the notation ${\rm mat}(x,m,n)$ to denote the matrix obtained by reshaping $x$ into an $m \times n$ matrix such that ${\rm vec}({\rm mat}(x,m,n)) = x$. For two matrices $A$ and $B$, we use the notation $[A;B]$ to denote the matrix obtained by appending $B$ below the last row of $A$ when they have identical number of columns, and similarly we use the notation $[A,B]$ to denote the matrix obtained by appending $B$ to the last column of $A$ when they have identical number of rows. $\otimes$ denotes the tensor product. A random variable $x\in\mathbb{R}$ is said to be sub-Gaussian with mean zero and variance $\sigma^2$ if $\mathbb{E}[\exp(sx)] \leq \exp(\frac{\sigma^2s^2}{2})$ for any $s\in\mathbb{R}$. For such a sub-Gaussian variable $x$, it holds that $\mathbb{P}(|x| > t) \leq 2 \exp (-\frac{t^2}{2\sigma^2})$ for any $t>0$.

\section{Theoretical Guarantees for Cluster Recovery}
\label{sec:the_gua}
We provide theoretical guarantees for the recovery of cluster memberships of the lrCC model \eqref{model-convex}. First, we establish the exact recovery guarantee for general low rank matrix-variate clustering problems. Next, we focus on the asymptotic cluster recovery behavior of our model as the sample size tends to infinity. Specifically, we show that for a mixture of Gaussians, the lrCC model \eqref{model-convex}  effectively labels points located within the neighborhoods of the low rank means.

\subsection{Exact Cluster Recovery for Finite Samples}
\label{sec:exact_clu}
In this subsection, we establish the exact recovery guarantees for the lrCC model \eqref{model-convex} in the context of a finite sample setting. Specifically, we prove that under certain conditions, the lrCC model \eqref{model-convex} can accurately recover the true cluster memberships with properly chosen penalty parameters. Notably, by setting the parameter $\gamma_2$ before the nuclear norm to zero, our result can be refined to align with the weighted convex clustering model in \cite{sun2021convex}. To the best of our knowledge, our results provide the most general and comprehensive exact recovery guarantees available in the literature. Building on the weighted convex clustering model in \cite{sun2021convex}, we introduce an additional nuclear norm penalty term to promote low-rankness when the true underlying centroids are low rank. This enhancement makes our model and analysis more flexible for capturing complex data structures. The exact cluster recovery guarantee of the lrCC model \eqref{model-convex} is provided in the following theorem, with its proof provided in Section \ref{sec:proof_exact_recovery}.

\begin{theorem}\label{thm:exact_recovery}
Let $A_1, \ldots, A_n \in \mathbb{R}^{d_1 \times d_2}$ be the observations, partitioned into $K$ underlying true clusters. Denote $\mathcal{I}_{\alpha} = \{i \in [n] \mid A_i \mbox{ belongs to cluster } \alpha\}$, for each $\alpha \in [K]$. We apply the lrCC model \eqref{model-convex} for cluster assignment, and we  define
\begin{align*}
\begin{aligned}
& A^{(\alpha)} := \frac{1}{|\mathcal{I}_{\alpha}|}\sum_{i \in \mathcal{I}_{\alpha}} A_i, \quad
&&\Delta  := \min_{\alpha,\beta\in [K], \alpha\neq \beta} \|A^{(\alpha)}-A^{(\beta)}\|_F,\\
&    \eta_{ij}^{(\alpha)} := \frac{1}{w_{ij}} \sum\limits_{\beta\in [K]\backslash \{\alpha\}} \Big|\sum_{m\in \mathcal{I}_{\beta}}(w_{im} - w_{jm})\Big|, \quad
&& w_{\max} := \max_{\alpha \in[K]} \frac{2}{|\mathcal{I}_{\alpha}|} \sum_{i\in \mathcal{I}_{\alpha}, j\in [n]\backslash \mathcal{I}_{\alpha}} w_{ij}.
\end{aligned}
\end{align*}

\begin{itemize}[left=5pt, labelsep=3pt, itemsep=2pt]
\item[{\rm (i)}]
For any $\alpha\in [K]$, the observations in $\mathcal{I}_{\alpha}$ are assigned to the same cluster if conditions \eqref{cond-a} and \eqref{eq:gamma1min} hold:
\begin{center}
\begin{minipage}{0.95\textwidth}
\vspace{-0.2cm}
\begin{equation}\leqnomode
\hspace{-0.2cm} \text{The induced subgraph of } \mathcal{G} \text{ on } \mathcal{I}_{\alpha} \text{ is a clique, and }   \max_{i,j \in \mathcal{I}_{\alpha}}\eta_{ij}^{(\alpha)} < |\mathcal{I}_{\alpha}| ,  \alpha \in [K]; \tag{a} \label{cond-a}
\end{equation}

\vspace{-0.7cm}
\begin{equation}\leqnomode
\gamma_1 \geq \gamma_{1,\min} := \max_{\alpha\in [K]}\max_{  \substack{i,j\in \mathcal{I}_{\alpha},\\ i\neq j} } \frac{\|A_i-A_j\|_F}{ w_{ij}\left(|\mathcal{I}_{\alpha}| -\eta_{ij}^{(\alpha)}\right)}. \tag{b} \label{eq:gamma1min}
\end{equation}
\end{minipage}
\end{center}

\item[{\rm (ii)}]
For any distinct $\alpha, \beta \in [K]$,  observations in $\mathcal{I}_{\alpha}$ and $\mathcal{I}_{\beta}$ are assigned to two different clusters if condition   \eqref{eq:gamma12}, in addition to \eqref{cond-a} and \eqref{eq:gamma1min}, holds:
\begin{center}
\begin{minipage}{0.95\textwidth}
\vspace{-0.2cm}
\begin{equation}\leqnomode
\hspace{-0.2cm} A^{(\alpha)} \neq A^{(\beta)} \ \forall \ \alpha \neq \beta,  \text{ and } \ \gamma_1 w_{\max} + \gamma_2 \sqrt{d_2} < \Delta.
\tag{c} \label{eq:gamma12}
\end{equation}
\end{minipage}
\end{center}
\end{itemize}
\end{theorem}

\begin{remark}
Theorem~\ref{thm:exact_recovery}(i) asserts that the lrCC model \eqref{model-convex} is capable of {\it merging} points from the same cluster under conditions \eqref{cond-a} and \eqref{eq:gamma1min}, which we refer to as the {\it merging ability} of lrCC. To guarantee the {\it merging ability}, a lower bound for $\gamma_1$, as in \eqref{eq:gamma1min}, is necessary. This requirement is intuitive, as $\gamma_1$ penalizes differences in $\|X_i - X_j\|_F$, $(i,j)\in {\cal E}$. When $\gamma_1$ is sufficiently large, all $X_i$'s will merge into a single cluster. On the other hand, Theorem~\ref{thm:exact_recovery}(ii) guarantees that lrCC is capable of {\it distinguishing} points from different clusters, which we refer to as its {\it distinguishing ability}. This ability is guaranteed under condition \eqref{eq:gamma12}, which stipulates that the penalty parameters $\gamma_1$ and $\gamma_2$ must not exceed a specified upper bound.
\end{remark}

Next, we further verify and clarify Theorem~\ref{thm:exact_recovery} via simple numerical experiments.
For the well constructed data presented below,  we demonstrate the existence of a ``perfect'' region in the $(\gamma_1, \gamma_2)$ space, where lrCC exhibits both merging and distinguishing abilities. This enables the method to perfectly recover the true cluster memberships. This ``perfect'' region is constrained by condition \eqref{eq:gamma1min}:  $\gamma_1 \geq \gamma_{1,\min}$, and condition \eqref{eq:gamma12}: $\gamma_1 w_{\max} + \gamma_2 \sqrt{d_2} < \Delta $. Additionally, we explain the consequences when the parameters $\gamma_1$ and $\gamma_2$ fail to satisfy one or both of the two conditions --- specifically, when lrCC may lose its merging ability, its distinguishing ability, or both.

We construct synthetic data with four low rank centroids, which are given by $ M_i = U_i{\rm Diag}( C_{1i},  C_{2i})  V_i^{\intercal}$, $i\in [4]$, where $C_{i1},C_{i2} \sim N(0,1)$,
$U_i\in\mathbb{R}^{20\times 2}$ and $V_i\in\mathbb{R}^{10\times 2}$ are the top $2$ left and right singular vectors of a random matrix in $\mathbb{R}^{20\times 10}$ with entries drawn from $N(0,1)$.
The observations are then generated with noise as $A_i = M_{ s_i^*} + 0.1E_i,\ i\in[n]$, where $s_i^*$ is the true underlying cluster label, and each entry of $E_i$ is drawn from $N(0,1)$.
We generate $50$ samples for each cluster and visualize the data using principal component analysis (PCA), projecting it onto a two-dimensional space.  Figure~\ref{fig:verify1}(b) visualizes the data using PCA and the data points are colored according to their cluster labels.

\vspace{-0.35cm}
\begin{figure}[H]
\centering
\subfigure[Four regions in $(\gamma_1,\gamma_2)$ space]{\hspace{-25pt}
{\includegraphics[width=0.33\textwidth]{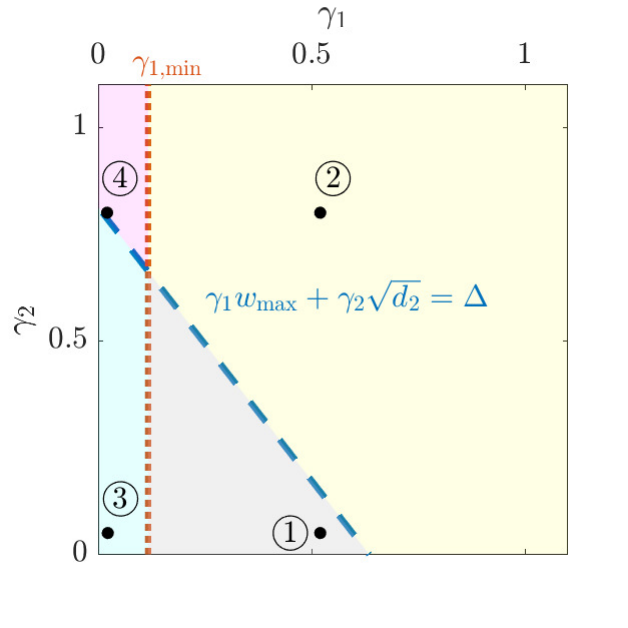}}
}
\hspace{0.02\textwidth}
\subfigure[PCA visualization of true labels]{
    \includegraphics[width=0.28\textwidth]{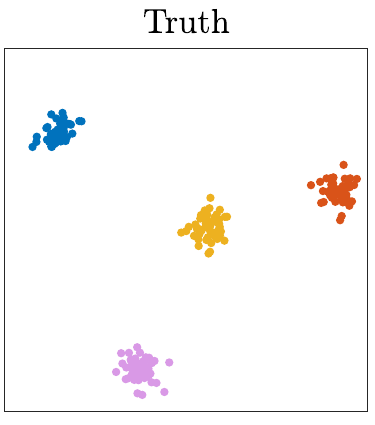}
}
\hspace{0.02\textwidth}
\subfigure[PCA visualization of lrCC with $(\gamma_1,\gamma_2)\!\!=\!\! (0.5,0.05)$ at $\textcircled{1}$]{
    \includegraphics[width=0.28\textwidth]{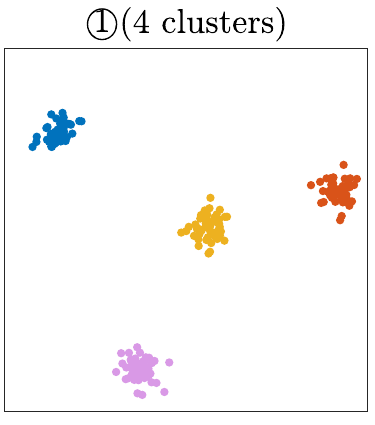}
}
\\[-2mm]
\subfigure[PCA visualization of lrCC with $(\gamma_1,\!\gamma_2)\!\!=\!\! (0.5,0.8)$ at $\textcircled{2}$]{
    \includegraphics[width=0.28\textwidth]{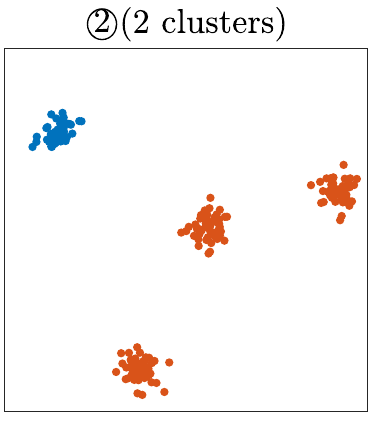}
}
\hspace{0.02\textwidth}
\subfigure[PCA visualization of lrCC with $\!(\!\gamma_1,\!\gamma_2\!)\!\!=\!\! (0.0026,\!0.05\!)\!$ at $\!\textcircled{3}$]{
    \includegraphics[width=0.28\textwidth]{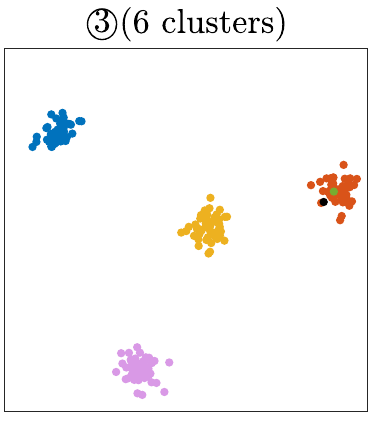}
}
\hspace{0.02\textwidth}
\subfigure[PCA visualization of lrCC with $(\!\gamma_1,\!\gamma_2\!)\!\!=\!\! (0.001,\!0.8)$ at $\textcircled{4}$]{
    \includegraphics[width=0.28\textwidth]{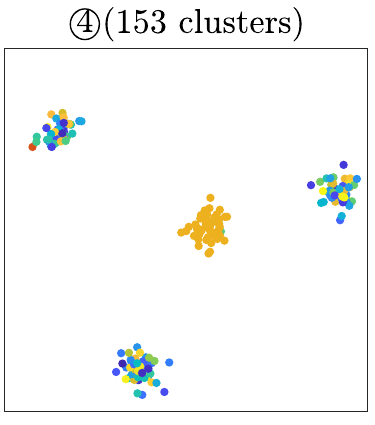}
}
\vspace{-0.3cm}
\caption{Visual explanation of Theorem~\ref{thm:exact_recovery} on a synthetic data set.}
\label{fig:verify1}
\end{figure}

We apply lrCC to learn the cluster memberships of the observations. The edge set is defined as $\mathcal{E} =\cup_{i=1}^n \{l(i,j)\mid A_i \mbox{ is among } A_j\mbox{'s 50-nearest neighbors, } 1\leq i < j \leq n\}$, and the weights $\{w_{ij}\mid l(i,j)\in\mathcal{E}\}$ are determined using a Gaussian kernel as $w_{ij} = \exp(-0.5\|A_i-A_j\|_F^2)$. With this setup, we can verify that condition \eqref{cond-a} and $ A^{(\alpha)} \neq A^{(\beta)}$ for $\alpha \neq \beta$ in condition \eqref{eq:gamma12} of Theorem~\ref{thm:exact_recovery}  are satisfied. In order to ensure the application of Theorem~\ref{thm:exact_recovery}, appropriate choices of $\gamma_1$ and $\gamma_2$ are required. Through computation, we obtain that $\gamma_{1,\min}=0.0078$,  $w_{\max}=4.05$, and $\Delta=2.57$.
To illustrate the requirements for $\gamma_1$ and $\gamma_2$, in Figure~\ref{fig:verify1}(a), we plot the lines corresponding to condition \eqref{eq:gamma1min}:  $\gamma_1 \geq \gamma_{1,\min}$, and the second part of condition \eqref{eq:gamma12}: $\gamma_1 w_{\max} + \gamma_2 \sqrt{d_2} < \Delta $.  For clarity, we slightly shift the red dotted line representing $\gamma_1 = \gamma_{1,\min}$ to the right, since $\gamma_{1,\min}$ is very close to zero. By Theorem~\ref{thm:exact_recovery}, values of $(\gamma_1,\gamma_2)$ that lie within the gray triangular region in Figure~\ref{fig:verify1}(a) theoretically guarantee exact cluster recovery by lrCC. To numerically verify this, we select the point marked as ``$\textcircled{1}$'' in the $(\gamma_1, \gamma_2)$ space within this triangular region. As expected, the resulting clustering result by lrCC, displayed in Figure~\ref{fig:verify1}(c), matches the true labels in Figure~\ref{fig:verify1}(b) perfectly.

We are also interested in examining the remaining three regions divided by the two lines $\gamma_1=\gamma_{1,\min}$ and $\gamma_1 w_{\max} +\gamma_2 \sqrt{d_2} =\Delta$. First, we select point ``$\textcircled{2}$'' in the yellow region and apply lrCC. In this region, lrCC retains its merging ability since $\gamma_1 \geq\gamma_{1,\min}$, but it may lose the distinguishing ability because $\gamma_1 w_{\max} +\gamma_2 \sqrt{d_2} >\Delta$. The clustering result, depicted in Figure~\ref{fig:verify1}(d), shows that one cluster is correctly identified, while the remaining three clusters are incorrectly merged into a single cluster. This outcome is consistent with our theoretical conclusions:  lrCC maintains its merging ability but loses its distinguishing ability, failing to differentiate between the red, yellow, and purple clusters. Secondly, we examine point ``$\textcircled{3}$'' in the cyan region. In this region, lrCC maintains its distinguishing ability because $\gamma_1 w_{\max} +\gamma_2 \sqrt{d_2} < \Delta$, but it may lose the merging ability as $\gamma_1 < \gamma_{1,\min}$. The clustering result, shown in Figure~\ref{fig:verify1}(e), indicates that the data points are classified into 6 clusters, with the blue, yellow, and purple clusters correctly identified. However, for the red cluster, one point is assigned to the green cluster and another to the black cluster. This result aligns with the expected theory:  lrCC loses its merging ability, indicated by the failure to merge the green and black points into the red cluster. Lastly, we examine point ``$\textcircled{4}$'' in the purple region and plot the result in Figure~\ref{fig:verify1}(f). It shows that lrCC may lose both merging and distinguishing abilities in this region.

\subsection{Asymptotic Cluster Recovery for Infinite Samples}
\label{sec:asy_clu}
Building on the finite sample exact cluster recovery results in Theorem~\ref{thm:exact_recovery}, we now turn our attention to the asymptotic cluster recovery behavior of lrCC \eqref{model-convex}.

We narrow our focus to a mixture of $K$ Gaussians in $\mathbb{R}^{d_1\times d_2}$, with mixture weights $\pi_1,\cdots,\pi_K$, and the corresponding means $M_1,\cdots,M_K$. Denote $d:=d_1d_2$. The covariance matrices for all components are isotropic, given by $\sigma^2 I_{d} $, where $I_d$ is the $d\times d$ identity matrix. We assume that the mean matrices $M_1,\dots,M_K$ are low rank, deterministic, and unknown. Denote $\pi_{\min} := \min_{\alpha \in [K]}\pi_k$. Given a set of i.i.d. samples drawn from this mixture distribution, which are denoted as $A_1,\cdots,A_n \in\mathbb{R}^{d_1\times d_2}$. Each $A_i$ has a latent label $s_i^*\in [K]$, and we view it as a fixed realization sampled from the mixture distribution $\mathbb{P}(s_i^*=k)=\pi_k,\,\,k\in[K]$. Then the data generation process can be written as
\begin{equation}
 A_i = M_{ s_i^*} + E_i,\quad i\in[n], \label{model-lrmm}
\end{equation}
where the entries of the noise matrix $E_i$ are i.i.d. Gaussian variables with mean zero and variance $\sigma^2$. Our goal is to learn the cluster membership vector $s^*=(s_1^*,\dots,s_n^*)$ given the observations  $\{A_i \mid i\in[n]\}$.  Model \eqref{model-lrmm} is referred to as the low rank mixture model in \cite{lyu2022optimal,lyu2023optimal}.

Denote $\mathcal{I}_{\alpha} := \{i \in [n] \mid s_i^*= \alpha\}$. If we apply lrCC with a uniformly weighted fully connected graph, where $w_{ij}=1$ for all $i\neq j$, Theorem~\ref{thm:exact_recovery} provides a sufficient condition for exact cluster recovery:
\begin{align}
\min_{\substack{\alpha,\beta\in [K],\\ \alpha\neq \beta}} \|A^{(\alpha)}-A^{(\beta)}\|_F
> 2\max_{\alpha\in [K]} \left( \frac{\max_{  i,j\in \mathcal{I}_{\alpha}} \|A_i-A_j\|_F}{ |\mathcal{I}_{\alpha}|} \right)\cdot\max_{\alpha \in[K]} \left( n-|\mathcal{I}_{\alpha}|\right), \label{eq:mG_condition}
\end{align}
under which the conditions of Theorem~\ref{thm:exact_recovery} hold, with properly chosen parameters $\gamma_1$ and $\gamma_2$. When the sample size $n = {\rm poly}(\frac{d}{\pi_{\min}})$, we show in Appendix \ref{sec: dis_G} that
\begin{align*}
\max_{  i,j\in \mathcal{I}_{\alpha}} \|A_i-A_j\|_F = O\left(\sigma \sqrt{K} \ {\rm polylog}\left(\frac{d}{\pi_{\min}}\right)\right)
\end{align*}
with high probability. Here, ${\rm poly}(x)$ denotes a polynomial in $x$, and  ${\rm polylog}(x)$ denotes a polynomial in $\log x$.
Consequently, condition \eqref{eq:mG_condition} is aligned in principle with:
\begin{align*}
\min_{\alpha,\beta\in [K], \alpha\neq \beta} \|M_{\alpha}-M_{\beta}\|_F > O\left( \frac{1-\pi_{\min}}{\pi_{\min}} \sigma\sqrt{K} \  {\rm polylog}\left(n\right)\right).
\end{align*}
As the sample size $n$ tends to infinity, the bound above implies that  $M_{\alpha}$'s must be more widely separated, making it increasingly challenging to distinguish between the clusters. To address this challenge, the next theorem focuses on  asymptotic cluster memberships recovery, exploring conditions under which lrCC guarantees correct labeling for data points  within a fixed number of standard deviations from the center means, defined as $\mathcal{I}_{\alpha,t}$ below. This extends the results in \cite{jiang2020recovery}.  
Specifically, for a uniformly weighted fully connected graph, where $w_{ij}=1$ for all $i\neq j$, condition \eqref{cond-a1} in Theorem~\ref{thm:merge} is automatically satisfied. Condition \eqref{cond-b1} in Theorem~\ref{thm:merge} then simplifies to  $\gamma_{1} \geq  \frac{2t\sigma}{(F(t;d)\pi_{\alpha} - \epsilon)n }$, which aligns with the result in \cite[(13)]{jiang2020recovery}.

\begin{theorem}\label{thm:merge}
Suppose $ A_1,\dots,A_n \in \mathbb{R}^{d_1\times d_2}$ are generated from \eqref{model-lrmm}.  We apply the lrCC model \eqref{model-convex} for cluster assignment. Given $t > 0$ and $\alpha \in [K]$. We define
$$
\mathcal{I}_{\alpha,t} := \{i \in [n] \mid \|A_i - M_{\alpha} \|_F \leq t \sigma \},\quad
\eta_{ij}^{(\alpha,t)} :=\frac{1}{w_{ij}} \sum_{m\in [n] \backslash  \mathcal{I}_{\alpha,t}} | w_{im} - w_{jm}|.
$$
Let $F(\cdot ;d)$ be the cumulative distribution function of the chi distribution with $d$ 
degrees of freedom. Let $\epsilon \in (0, F(t,d) \pi_{\alpha})$ be a small number.

\vspace{0.1cm}
\begin{itemize}[left=5pt, labelsep=3pt, itemsep=2pt]
\item[{\rm (i)}]
With probability at least $1 - \exp (-2\epsilon^2 n)$, the observations in $\mathcal{I}_{\alpha,t}$ are assigned to the same cluster if conditions \eqref{cond-a1} and \eqref{cond-b1} hold:

\begin{center}
\begin{minipage}{0.96\textwidth}
\vspace{-0.3cm}
\begin{equation}\leqnomode
\hspace{-7pt} \text{The induced subgraph of } \mathcal{G} \!\text{ on }  \mathcal{I}_{\alpha,t} \!\text{ is a clique,} \! \text{ and }
\! \!\! \!\max_{i,j \in \mathcal{I}_{\alpha,t}}\! \eta_{ij}^{(\alpha,t)}  \!\!\!<\!
\!(\!F(t;\!d)\pi_{\alpha} \!-  \epsilon)n; \tag{a1} \label{cond-a1}
\end{equation}

\vspace{-0.8cm}
\begin{equation}\leqnomode
\gamma_1 \geq \max_{  \substack{i,j\in \mathcal{I}_{\alpha,t},\\ i\neq j} } \frac{2t\sigma}{ w_{ij}\left((F(t;d)\pi_{\alpha} - \epsilon)n -\eta_{ij}^{(\alpha,t)}\right)}. \tag{b1} \label{cond-b1}
\end{equation}
\end{minipage}
\end{center}

\vspace{0.1cm}
\item[{\rm (ii)}]
Given $\beta \in [K]$ such that $\beta\neq \alpha$. Define $\mathcal{I}_{\beta,t}$ analogously to $\mathcal{I}_{\alpha,t}$.
Then with probability at least $
(1 -  \exp (-2\epsilon^2 n) - 2^{ -(F(t;d) \pi_{\alpha} -\epsilon)n } )(1 - \exp (-2\epsilon^2 n) - 2^{ -(F(t;d) \pi_{\beta} -\epsilon)n } )$,
there exist $ i\in \mathcal{I}_{\alpha,t}$ and
$ j\in \mathcal{I}_{\beta,t}$ such that $A_i$ and $A_j$ are assigned to different clusters if

\vspace{-0.7cm}
\begin{center}
\begin{minipage}{0.96\textwidth}
\vspace{0.5cm}
\begin{equation}\leqnomode
\gamma_1 (n-1) \widetilde{w}_{\max} + \gamma_2 \sqrt{ d_2} < 0.5\|M_{\alpha} - M_{\beta}\|_F, \quad \widetilde{w}_{\max} := \max_{l(i,j)\in \mathcal{E}} w_{ij}. \tag{c1} \label{cond-c1}
\end{equation}
\end{minipage}
\end{center}
\vspace{-0.1cm}
Namely, $\mathcal{I}_{\alpha,t}$ and $\mathcal{I}_{\beta,t}$ will not be in the same cluster.
\end{itemize}
\end{theorem}

\subsection{Proof of Theorems~\ref{thm:exact_recovery} and \ref{thm:merge}}
\label{sec:proof_exact_recovery}
In this subsection, we provide the full proofs of Theorems~\ref{thm:exact_recovery} and \ref{thm:merge}. To proceed, we first need to establish two propositions --- Propositions~\ref{prop: connection_two_problems} and \ref{prop: neq} --- which serve as essential intermediate steps. Specifically, Propositions~\ref{prop: connection_two_problems} and \ref{prop: neq}  characterizes the so called {\it merging ability} and {\it distinguishing ability}, respectively, as mentioned above in Section~\ref{sec:exact_clu}.

Proposition~\ref{prop: connection_two_problems} connects the optimal solution of the lrCC model \eqref{model-convex} to the solution of the following centroid learning optimization problem \eqref{eq: cluster_opt}. Specifically, given a partition $\Gamma=\{\mathcal{J}_1,\cdots,\mathcal{J}_G\}$ of $[n]$, we consider the corresponding centroid learning optimization problem
\begin{equation}
\min_{X^{(1)},\dots,X^{(G)}}
\left\{ \begin{aligned}
&\frac{1}{2}\sum_{\alpha\in [G]} |\mathcal{J}_{\alpha}| \cdot \|X^{(\alpha)}-A^{\Gamma,\alpha}\|_F^2  \\
&+ \gamma_1 \!\!\!\!
\sum_{\alpha,\beta\in[G], \alpha<\beta}
\! \!w^{(\alpha,\beta)} \|X^{(\alpha)}-X^{(\beta)}\|_F +\gamma_2\!\!\sum_{\alpha\in[G]}\! |\mathcal{J}_{\alpha}|\cdot\|X^{(\alpha)}\|_*
\end{aligned}
\right\},
\tag{\mbox{P-$\Gamma$}}
\label{eq: cluster_opt}
\end{equation}
where $A^{\Gamma,\alpha} := \frac{1}{|\mathcal{J}_{\alpha}|}\sum_{i\in \mathcal{J}_{\alpha}} A_i$ and $w^{(\alpha,\beta)} := \sum_{i\in \mathcal{J}_{\alpha},j\in \mathcal{J}_{\beta}} w_{ij}$. Proposition~\ref{prop: connection_two_problems} provides a sufficient condition under which the solution to the lrCC model \eqref{model-convex} can be derived from the solution to  problem \eqref{eq: cluster_opt}. This offers crucial insight into the clustering behavior of the lrCC model \eqref{model-convex} --- the points in $\mathcal{J}_{\alpha}$ will belong to the same cluster. The detailed proof is given in Appendix \ref{sec:proof_of_prop_merging}.

\begin{proposition}\label{prop: connection_two_problems}
Suppose for each $\alpha \in [G]$, the induced subgraph of $\mathcal{G}$ on $\mathcal{J}_{\alpha}$ is a clique, and for each $(i,j)\in \mathcal{J}_{\alpha}$, $\eta_{ij}^{\Gamma,\alpha} := \frac{1}{w_{ij}}\sum_{\beta\in [G]\backslash \{\alpha\}} \Big|\sum_{m\in \mathcal{J}_{\beta}}(w_{im}  - w_{jm}) \Big| < |\mathcal{J}_{\alpha}|$. In addition, the parameter $\gamma_1$ satisfies
\begin{align}
\gamma_1 \geq \max_{\alpha\in [G]}\max_{  \substack{i,j\in \mathcal{J}_{\alpha},\\ i\neq j} } \frac{\|A_i-A_j\|_F}{ w_{ij}\left(|\mathcal{J}_{\alpha}| -\eta_{ij}^{\Gamma,\alpha}\right)}.\label{eq: lb_gamma1_J}
\end{align}
Then the unique optimal solution $(X_1^*,\dots,X_n^*)$ to problem \eqref{model-convex} can be obtained as
\begin{align*}
X_i^* = X^{(\alpha),*},\  i \in \mathcal{J}_{\alpha},\ \alpha\in [G],
\end{align*}
where $(X^{(1),*},\dots,X^{(G),*})$ denotes the unique optimal solution to \eqref{eq: cluster_opt}.
\end{proposition}

The above proposition tells when all points in $\mathcal{J}_{\alpha}$ will belong to the same set. However, it does not address whether we can  distinguish between sets $\mathcal{J}_{\alpha}$ and $\mathcal{J}_{\beta}$. The following Proposition \ref{prop: neq} resolves this by providing conditions under which the learned centroids $X^{(\alpha),*}$ and $X^{(\beta),*}$ are indeed distinct, whose proof is provided in Appendix \ref{sec:proof_of_distinguish}.

\begin{proposition} \label{prop: neq}
Consider problem \eqref{eq: cluster_opt} along with its optimal solution, denoted as $(X^{(1),*}, \dots, X^{(G),*})$. For $\alpha \neq \beta$, if the penalty parameters $\gamma_1$ and $\gamma_2$ satisfy the inequality
\begin{align*}
     \gamma_1 {\left(w^{(\alpha)} + w^{(\beta)} \right)} + \gamma_2 \sqrt{d_2-{\rm rank}(X^{(\alpha),*})} < \|A^{\Gamma,\alpha}-A^{\Gamma,\beta}\|_F,
\end{align*}
where $w^{(\alpha)} := \frac{1}{|\mathcal{J}_{\alpha}|}\sum_{m\in[G]\backslash \{\alpha\} } w^{(\alpha,m)}$, 
then $X^{(\alpha),*}\neq X^{(\beta),*}$.
\end{proposition}

With Propositions~\ref{prop: connection_two_problems} and \ref{prop: neq} established, we are now ready to prove Theorems~\ref{thm:exact_recovery} and \ref{thm:merge}, as stated in Sections~\ref{sec:exact_clu} and \ref{sec:asy_clu}, respectively. The key technique involves selecting an appropriate partition of $[n]$, since the formulation of problem \eqref{eq: cluster_opt}, and thus the two propositions, depends on the choice of the partition $\Gamma$.

\begin{proof}[\textbf{Proof of Theorem~\ref{thm:exact_recovery}}]
We take the partition $\Gamma$ in \eqref{eq: cluster_opt} as $\{\mathcal{I}_1,\cdots,\mathcal{I}_K\}$. Theorem~\ref{thm:exact_recovery}(i) follows directly from Proposition~\ref{prop: connection_two_problems}, and Theorem~\ref{thm:exact_recovery}(ii) follows directly from Proposition~\ref{prop: neq}. The proof is completed.
\end{proof}

\begin{proof}[\textbf{Proof of Theorem~\ref{thm:merge}}]
(i) We establish the proof by applying Proposition~\ref{prop: connection_two_problems} under a suitable partition of $[n]$. Specifically, we partition the $n$ observed points into $n- |\mathcal{I}_{\alpha,t}| + 1$ subsets in the following manner: all points in $\mathcal{I}_{\alpha,t}$
form a single subset, while each of the remaining $n- |\mathcal{I}_{\alpha,t}|$ points is placed into its own individual subset. First, we derive a bound on $|\mathcal{I}_{\alpha,t}|$. To do so, for $i\in [n]$, we define the indicator random variable $v_i$ as: $v_i = 1$ if $i\in \mathcal{I}_{\alpha,t}$ and $v_i=0$ otherwise. Then $|\mathcal{I}_{\alpha,t}|$ can be expressed as the sum of $v_1,\dots,v_n$, and thus
\begin{align*}
\mathbb{E}[ |\mathcal{I}_{\alpha,t}| ] = \sum_{i=1}^n \mathbb{E}[v_i] &= \sum_{i=1}^n \mathbb{P}(v_i=1) \geq \sum_{i=1}^n \mathbb{P}(i\in \mathcal{I}_{\alpha,t} \mid s_i^* = \alpha) \mathbb{P}(s_i^* = \alpha) \\
&= \sum_{i=1}^n\mathbb{P}(\|E_i\|_F \leq t\sigma) \pi_{\alpha}=F(t;d)\pi_{\alpha}n,
\end{align*}
where the last equality holds as $\|E_i\|_F/\sigma$ follows the chi distribution with $d$ degrees of freedom. Moreover, by Hoeffding's inequality:
\begin{align}\label{eq:proof-size}
\mathbb{P}(|\mathcal{I}_{\alpha,t}|\geq F(t;d) \pi_{\alpha} n - \epsilon n) \geq
\mathbb{P}(|\mathcal{I}_{\alpha,t}|\geq\mathbb{E}[ |\mathcal{I}_{\alpha,t}| ] - \epsilon n)
\geq
1-\exp(-2\epsilon^2n).
\end{align}
This, together with conditions \eqref{cond-a1} and \eqref{cond-b1}, guarantees that the assumptions in Proposition~\ref{prop: connection_two_problems} hold under the constructed partition with probability at least $1 - \exp (-2\epsilon^2 n)$. As a result, the points in $\mathcal{I}_{\alpha,t}$ will be clustered into the same cluster.

(ii) The second part mainly follows from \cite[Theorem~3]{jiang2020recovery}. The optimality condition of \eqref{model-convex} is
Note that $ \|{\rm Prox}_{\gamma_2 \|\cdot\|_*}(X) - X\|_F \leq \gamma_2 \sqrt{d_2}$ for any $X\in\mathbb{R}^{d_1\times d_2}$; see Proposition~\ref{prop: nuc2} for its proof.
Then we have for each $i\in [n]$, $\|X_i^* - A_i\|_F \leq \gamma_1 \sum_{l(i,j)\in \mathcal{E}} w_{ij} \|U_{ij}^* \|_F + \gamma_2 \sqrt{d_2}\leq \gamma_1(n-1) \widetilde{w}_{\max} + \gamma_2 \sqrt{d_2}.$
By \eqref{eq:proof-size} and \cite[(16)]{jiang2020recovery}, the mirror-image symmetry property of Gaussian distributions gives
\begin{align*}
& \mathbb{P}\left( \min _{i\in \mathcal{I}_{\alpha,t}} \langle A_i - M_{\alpha}, M_{\beta} - M_{\alpha}\rangle >0 \right) \\
\leq &  \mathbb{P}\!\left( \!\min _{i\in \mathcal{I}_{\alpha,t}} \!\langle A_i \!-\! M_{\alpha}, M_{\beta} \!-\! M_{\alpha}\rangle \!>\!0  \middle\vert  |\mathcal{I}_{\alpha,t}|\!\geq \!F(t;\!d) \pi_{\alpha} n \!-\! \epsilon n\!\right) + \mathbb{P}\left( |\mathcal{I}_{\alpha,t}|\!<\! F(t;\!d) \pi_{\alpha} n - \epsilon \right) \\
\leq &  2^{-F(t;d) \pi_{\alpha} n + \epsilon n } + \exp (-2\epsilon^2 n).
\end{align*}
Together with \cite[(17)]{jiang2020recovery}, we can find some $i\in \mathcal{I}_{\alpha,t}$ and $j \in \mathcal{I}_{\beta,t}$, with probability at least $(1-2^{-F(t;d) \pi_{\alpha} n + \epsilon n } - \exp (-2\epsilon^2 n))(1-2^{-F(t;d) \pi_{\beta} n + \epsilon n } - \exp (-2\epsilon^2 n))$, such that
$
\|A_i-A_j\|_F \geq \|M_{\alpha} - M_{\beta}\|_F,
$
which further implies
$
\|X_i^* - X_j^*\|_F \geq \|A_i - A_j\|_F - \|X_i^*-A_i\|_F - \|X_j^* - A_j\|_F \geq \|M_{\alpha} - M_{\beta}\|_F - 2(\gamma_1(n-1) \max_{i,j\in [n]} w_{ij} + \gamma_2 \sqrt{d_2}) > 0.
$
This concludes the proof.
\end{proof}

\section{Bounds on Prediction Error} \label{sec:bound_pre_err}
In this section, we analyze the statistical properties of the proposed lrCC model \eqref{model-convex}. We rigorously derive the finite sample bound for the prediction error, demonstrating how the estimator performs in finite sample settings. Furthermore, we establish the prediction consistency of the estimator, showing that as the sample size grows, the prediction converges to the truth. This analysis underscores both the theoretical soundness and practical reliability of the lrCC model, providing important insights into its behavior in various data regimes.

\subsection{Finite Sample Bound on Prediction Error}
\label{sec: finite_prediction}
For ease of notation, we reformulate the lrCC model \eqref{model-convex} in the vector form as
\begin{align}
\min_{x\in \mathbb{R}^{dn}}\  \frac{1}{2}\|x-a\|^2 + \gamma_1 \sum_{l(i,j)\in \mathcal{E}} w_{ij} \| D^{l(i,j)}x\|+ \gamma_2 \sum_{i=1}^n \|{\cal M}^{i} x \|_*, \label{eq: new_vec_form}
\end{align}
where $d := d_1d_2$; $a := [{\rm vec}(A_1);\cdots; {\rm vec}(A_n)]\in \mathbb{R}^{dn}$; $D^{l(i,j)} := (e_i^{\intercal}-e_j^{\intercal})\otimes I_d\in \mathbb{R}^{d\times dn}$ with $e_i$ being the $i$-th standard basis vector in $\mathbb{R}^n$; ${\cal M}^{i}$ is a linear map from $\mathbb{R}^{dn}$ to $\mathbb{R}^{d_1\times d_2}$ such that ${\cal M}^{i} x = {\rm mat}((e_i^{\intercal}\otimes I_d)x,d_1,d_2)$ for $i\in [n]$. The relationship between \eqref{model-convex} and \eqref{eq: new_vec_form} is captured by the equation $X_i = {\cal M}^{i} x$, $i\in [n]$.

Suppose that $a=x_0+\epsilon$, where $\epsilon\in \mathbb{R}^{dn}$ is a vector of independent sub-Gaussian random variables with mean zero and variance $\sigma^2$. Without loss of generality, we assume $\min_{l(i,j)\in \mathcal{E}} w_{ij}\geq \frac{1}{2}$. In order to provide a finite sample bound for the prediction error of the lrCC estimator in \eqref{model-convex} (see Section \ref{sec:pred_consistency}), we analyze its equivalent form \eqref{eq: new_vec_form} in the following theorem.  The proof is postponed to Section \ref{sec: thm_finite_sample}.

\begin{theorem}\label{thm:main_con}
Let $\hat{x}$ be the solution to \eqref{eq: new_vec_form}. If $\gamma_1 \geq \frac{4\sigma}{\sigma_{\min}(B)}\sqrt{d\log(d|\mathcal{E}|)}$, then
\begin{align*}
\frac{1}{2dn} \|\hat{x}-x_0\|^2 &\leq \sigma^2 \left[ \frac{\kappa_0}{n} +  \sqrt{\frac{\kappa_0\log(dn)}{dn^2}}\right]
 +\frac{\gamma_1}{2dn} \sum_{l(i,j)\in \mathcal{E}} (1+2w_{ij}) \| D^{l(i,j)}x_0\|
 \\
 &+  \gamma_2 \left[ \frac{ \sigma}{n^{1/4}}+ \frac{1}{dn}
\sum_{i=1}^n \|{\cal M}^{i} x_0 \|_* \right]
\end{align*}
holds with probability at least
\begin{align}\label{eq:prob_prediction}
1 - \frac{2}{d|\mathcal{E}|} -  2 \exp \left( -\frac{ \sqrt{d^3 n}}{2 }\right) -\exp\left\{ -\min\left(
c_1 \log(dn),c_2 \sqrt{d\kappa_0\log(dn)}\right)\right\},
\end{align}
where $c_1$, $c_2$ are positive constants. Here $\kappa_0$ denotes the number of connected components and $B\in \mathbb{R}^{n\times |{\cal E}|}$ represents the incidence matrix for the graph ${\mathcal{G}}=([n],\mathcal{E})$.
\end{theorem}

Here are some remarks on the above theorem. Comparing with the finite sample prediction bound analysis for other variants of convex clustering models in the literature \cite{tan2015statistical,wang2018sparse,chakraborty2023biconvex}, we extend existing theoretical frameworks to general edge sets and general weights, providing a more comprehensive and versatile analysis that can accommodate a wider variety of applications. In particular, in contrast to \cite{tan2015statistical,wang2018sparse}, which are limited to fully connected graphs ${\cal E}$, our framework operates under more general conditions by considering arbitrary edge sets. This extension is particularly relevant in sparse graph settings, such as the commonly used convex clustering models with $k$-nearest neighbor graphs. Moreover, unlike  existing works \cite{tan2015statistical,wang2018sparse,chakraborty2023biconvex}, which assume  uniform weights, our proposed estimator is more flexible 
by incorporating general weights. This generalization is critical in applications where the underlying data are inherently non-uniform. 
As a side note, we extend our analysis of the prediction error bound for the lrCC model \eqref{model-convex} to the case where $\epsilon$ is $M$-concentrated, allowing for dependence among the noise. Details are provided in Appendix \ref{sec: Mnoise}.

\subsection{Prediction Consistency}\label{sec:pred_consistency}
Based on the above theorem, we can now analyze the prediction consistency of the lrCC estimator in \eqref{model-convex}. Suppose the data matrices $\{A_i\}_{i=1}^n$ are generated according to model \eqref{model-lrmm}. Then by Theorem \ref{thm:main_con}, when $\gamma_1 \geq \frac{4\sigma}{\sigma_{\min}(B)}\sqrt{d\log(d|\mathcal{E}|)}$, the solution $(X_1^*,\dots,X_n^*)$
of \eqref{model-convex} satisfies
\begin{align}
\begin{array}{l}
\displaystyle
\frac{1}{2dn} \sum_{i=1}^n \|X^*_i-M_{s_i^*}\|_F^2
\leq
\sigma^2 \left[ \frac{\kappa_0}{n} +  \sqrt{\frac{\kappa_0\log(dn)}{dn^2}}\right]   \\
\displaystyle
+\frac{\gamma_1}{2dn} \sum_{l(i,j)\in \mathcal{E}} (1+2w_{ij}) \| M_{s_i^*} - M_{s_j^*} \|_F
+  \gamma_2 \left[ \frac{ \sigma}{n^{1/4}} + \frac{1}{dn}
\sum_{i=1}^n \|M_{s_i^*} \|_* \right],
\end{array}
\label{eq: lrcc_consistency}
\end{align}
with high probability  (see \eqref{eq:prob_prediction}). Next, we show that appropriate penalty parameters $\gamma_1$ and $\gamma_2$ can be chosen such that the prediction error $\frac{1}{2dn} \sum_{i=1}^n \|X^*_i-M_{s_i^*}\|_F^2$ decays to zero as $n,d \to \infty$, provided the following condition holds:
\begin{align}\label{cond:pred_consistency}
    \frac{1}{\sigma_{\min}(B)}\sqrt{ \frac{|{\cal E}|^2 \log(d|\mathcal{E}|)}{dn^2}} = o(1),
\end{align}
along with additional mild bounded assumptions: $\kappa_0$, $\widetilde{w}_{\max} := \max_{l(i,j)\in \mathcal{E}} w_{ij}$, $\tau_{\max} := \max_{\alpha \in [K]} \|M_{\alpha}\|_*$, and $\Delta_{\max} := \max_{\alpha, \beta \in [K]} \|M_{\alpha} - M_{\beta}\|_F$ are all bounded, i.e., $O(1)$. Specifically, given any $c\geq 1$, we can choose
\begin{align}\label{eq:choose_gamma}
\gamma_1 = c \cdot \frac{4\sigma \sqrt{d\log(d|\mathcal{E}|)} }{\sigma_{\min}(B)}  \
\overset{\eqref{cond:pred_consistency}}{=} \
o\left(\frac{dn}{|{\cal E}|}\right),\quad
\gamma_2 = o\left(\min\{{n^{1/4}},d\}\right),
\end{align}
and deduce that the right hand side of \eqref{eq: lrcc_consistency} vanishes as $n,d\to \infty$.

The bounded assumptions on $\kappa_0 $, $ \widetilde{w}$, $ \tau_{\max}$, and $ \Delta_{\max}$ are natural and widely adopted in the literature. In particular, similar assumptions have been made in \cite{tan2015statistical,wang2018sparse,chakraborty2023biconvex}, requiring that the distances between clusters are bounded and that the nuclear norm of each true cluster center is also bounded. Moreover, we need the number of connected components in the graph ${\mathcal{G}}=([n],\mathcal{E})$ and the edge weights to be bounded.

It should be emphasized that our prediction consistency result applies to general graphs ${\mathcal{G}}=([n],\mathcal{E})$ as well as general weights, which generalizes existing results on prediction consistency analysis of variants convex clustering models with fully connected graphs \cite{tan2015statistical,wang2018sparse} and uniform weights \cite{tan2015statistical,wang2018sparse,chakraborty2023biconvex}. Up to our knowledge, we are the first to establish the consistency analysis of low rank (matrix) convex clustering type models with general graphs and general weights. To guarantee that the estimator in \eqref{model-convex} is prediction consistent for model \eqref{model-lrmm},  condition \eqref{cond:pred_consistency} is necessary, which depends on not only $n,d$ but also the choice of graphs. A clear characterization of condition \eqref{cond:pred_consistency} is provided in the following remark on two specific examples of graphs.

\begin{remark}
For a fully connected graph and a $k$-nearest neighbor graph, we can characterize the size of edge set $\mathcal{E}$ and the the minimum nonzero singular value of the incidence matrix $B$, simplifying condition  \eqref{cond:pred_consistency}.
\begin{itemize}[left=5pt, labelsep=3pt, itemsep=2pt]
\item  \textbf{Fully Connected Graph.} Consider ${\mathcal{G}}=([n],\mathcal{E})$, where $\mathcal{E}$ contains all possible edges. We have that $|{\cal E}| = n(n-1)/2$ and $\sigma_{\min}(B) =\sqrt{n}$ (see, e.g., \cite{de2007old}). Then, condition \eqref{cond:pred_consistency} is equivalent to
$
\sqrt{ \frac{n \log(dn(n-1)/2)}{d}}  = o(1)
$,
which is true if $n=o(d/\log (d))$.
This means that  the sample size $n$ must grow slower than  the dimension $d$, up to a logarithmic factor. This requirement is typically true in the high dimensional settings, and therefore the prediction consistency of \eqref{model-convex} with a fully connected graph can be guaranteed with properly selected parameters $\gamma_1 = o(\frac{d}{n})$ and $\gamma_2 = o({n^{1/4}})$. If we take $\gamma_2=0$ and set the edge weights to be uniform, our prediction consistency analysis coincides with the existing result for the standard convex clustering model in \cite{tan2015statistical}, and it also matches the result for the sparse convex clustering model in \cite{wang2018sparse}.

\item  \textbf{$k$-Nearest neighbor Graph.} Consider 
${\mathcal{G}}=([n],\mathcal{E})$, where each node is connected to its $k$-nearest neighbors based on Euclidean distances, with $k$ being a fixed number. Typically, $k$ is much smaller than $n$, and we assume $k=O(1)$ as $n,d\rightarrow \infty$. Then we have that $kn/2\leq |{\cal E}|\leq kn$.
By Lemma \ref{lemma:lb_knear} in Appendix~\ref{sec:prooflemma6}, we have that $\sigma_{\min}(B)\geq \frac{2}{n}\sqrt{\frac{k+1}{3}}$. Then, condition \eqref{cond:pred_consistency} holds provided that
$
\sqrt{ \frac{n^2 \log(dn)}{d}} = o(1)
$. This is true if $n=o(\sqrt{d/\log (d)})$, whereby the prediction consistency of \eqref{model-convex} with a $k$-nearest neighbor graph is ensured with appropriately chosen parameters $\gamma_1 = o(d)$ and $\gamma_2 = o({n^{1/4}})$.
\end{itemize}
\end{remark}

\subsection{Proof of Theorem~\ref{thm:main_con}}
\label{sec: thm_finite_sample}
Let $\mathcal{D} := B^{\intercal}\otimes I_d \in \mathbb{R}^{d|\mathcal{E}|\times dn} $. According to Lemma \ref{lemma: D} in Appendix~\ref{sec:lemma_subgaussian}, $ {\rm rank}(\mathcal{D}) = d(n - \kappa_0) $, where $\kappa_0$  is the number of connected components in ${\mathcal{G}}=([n],\mathcal{E})$. We can have its singular value decomposition $\mathcal{D} = U S V_1^{\intercal}$. Here $S\in \mathbb{R}^{d(n-\kappa_0)\times d(n-\kappa_0)}$ is a diagonal matrix, $U\in \mathbb{R}^{d|\mathcal{E}|\times d(n-\kappa_0)}$ and $V_1\in \mathbb{R}^{dn\times d(n-\kappa_0)}$ satisfy $U^{\intercal}U = I_{d(n-\kappa_0)}$, $V_1^{\intercal}V_1 = I_{d(n-\kappa_0)}$. And there exists $V_2\in \mathbb{R}^{dn\times d\kappa_0}$ such that $V=[V_1,V_2]\in \mathbb{R}^{dn\times dn}$ is an orthogonal matrix. It can be easily seen that $V_1 V_1^{\intercal} + V_2 V_2^{\intercal} = I_{dn}$, $V_2^{\intercal} V_2 = I_{d\kappa_0} $, and $V_1^{\intercal} V_2 =0$.

\begin{proof}[\textbf{Proof of Theorem \ref{thm:main_con}}]
By denoting $y = V_1^{\intercal} x\in \mathbb{R}^{d(n-\kappa_0)}$ and $z = V_2^{\intercal} x \in \mathbb{R}^{d\kappa_0}$, we have $x=V_1 y + V_2 z$. Then problem \eqref{eq: new_vec_form} is equivalent to
\begin{align}
\min_{y\in \mathbb{R}^{d(n-\kappa_0)},z\in \mathbb{R}^{d\kappa_0}}\
&\frac{1}{2dn}\|V_1 y + V_2 z-a\|^2 + \frac{\gamma_1}{dn} \sum_{l(i,j)\in \mathcal{E}} w_{ij} \left\| D^{l(i,j)}(V_1 y + V_2 z)\right \| \nonumber\\
&+ {\frac{\gamma_2}{dn}} \sum_{i=1}^n \left\|{\cal M}^{i} (V_1 y + V_2 z) \right\|_*. \label{eq: re_vec_form0}
\end{align}
Note that $\mathcal{D}$ is the matrix formed by stacking the matrices $\{D^{l(i,j)}\}_{l(i,j)\in {\cal E}}$ vertically, and $\mathcal{D}$ takes the singular value decomposition $\mathcal{D} = U S V_1^{\intercal}$. It can be seen that
\begin{align*}
D^{l(i,j)}(V_1 y + V_2 z)  = W^{l(i,j)} y, \quad l(i,j)\in {\cal E},
\end{align*}
where $W = US\in \mathbb{R}^{d|\mathcal{E}|\times d(n-\kappa_0 )}$ consisting of the matrices $\{W^{l(i,j)}\}_{l(i,j)\in {\cal E}}$ stacked vertically. Then problem \eqref{eq: re_vec_form0} can be written as
\begin{align}
\min_{\substack{z\in\mathbb{R}^{d\kappa_0},\\y\in \mathbb{R}^{d(n\!-\!\kappa_0)}} }
&\frac{1}{2dn}\!\|V_1 y \!+\! V_2 z\!-\!a\|^2 \!\!+\! \frac{\gamma_1}{dn} \!\!\!\sum_{l(i,\!j)\in \mathcal{E}}\!\!\! w_{ij} \!\left\| W^{l(i,\!j)}\! y\right \|
\!+\! {\frac{\gamma_2}{dn}}   \!\sum_{i=1}^n \!\left\|{\cal M}^{i} (\!V_1 y \!+\! V_2 z\!) \right\|_*. \label{eq: re_vec_form}
\end{align}
It can be seen that the pseudo-inverse of $W$ is $W^{\dagger}=S^{-1}U^{\intercal}$, and $W^{\dagger}W = I_{d(n-\kappa_0 )}$. In addition, we have that
$
\sigma_{\max}(W^{\dagger}) = 1 / \sigma_{\min}(W) = 1 / \sigma_{\min}(US) = 1 / \sigma_{\min}(D) = 1 / \sigma_{\min}(B)
$, where the last equality comes from Lemma \ref{lemma: D}(ii).

Denote the optimal solution to \eqref{eq: re_vec_form} as $(\hat{y},\hat{z})$. Due to the equivalence of problems \eqref{eq: new_vec_form} and \eqref{eq: re_vec_form}, we have $\hat{x} = V_1 \hat{y}+V_2\hat{z}$. Let $y_0 = V_1^{\intercal} x_0$ and $z_0 = V_2^{\intercal} x_0$. The optimality of $(\hat{y},\hat{z})$ to \eqref{eq: re_vec_form} indicates that
\begin{align*}
&\frac{1}{2dn}\|V_1 \hat{y} + V_2 \hat{z}-a\|^2 + \frac{\gamma_1}{dn} \sum_{l(i,j)\in \mathcal{E}} w_{ij} \left\| W^{l(i,j)} \hat{y}\right \|
+ {\frac{\gamma_2}{dn}} \sum_{i=1}^n \left\|{\cal M}^{i} (V_1 \hat{y} + V_2 \hat{z}) \right\|_*  \\
\leq & \frac{1}{2dn}\|V_1 y_0 + V_2 z_0-a\|^2 + \frac{\gamma_1}{dn} \!\!\sum_{l(i,j)\in \mathcal{E}} w_{ij} \!\left\| W^{l(i,j)} y_0\right \|
+ {\frac{\gamma_2}{dn}}  \sum_{i=1}^n \left\|{\cal M}^{i} (V_1 y_0 + V_2 z_0) \right\|_*.
\end{align*}
By plugging $a = x_0+\epsilon = V_1 y_0+ V_2 z_0+\epsilon$ into the above inequality, we have
\begin{align}
&\frac{1}{2dn}\!\|V_1 \!(\hat{y}\!-\!y_0) \!\!+ \!V_2 (\hat{z}\!-\!z_0) \|^2\!+\! \frac{\gamma_1}{dn} \!\!\sum_{l(i,\!j)\in \mathcal{E}}\!\! w_{ij} \!\left\| W^{l(i,\!j)} \hat{y}\right \|
\!+\! {\frac{\gamma_2}{dn}}\! \sum_{i=1}^n \!\left\|{\cal M}^{i} (V_1 \hat{y} \!+\! V_2 \hat{z}) \right\|_* \label{eq: hat_bar} \\
&\! \leq \!\!\frac{1}{dn}\langle \!V_1 \!(\hat{y}\!-\!\!y_0) \!\!+ \!V_2 (\hat{z}\!-\!\!z_0),\!\epsilon\!\rangle \!+ \!\frac{\gamma_1}{dn} \!\!\!\sum_{l(i,j)\in \mathcal{E}} \!\!\!\!w_{ij} \!\left\|\! W^{l(i,j)} \!y_0\!\right \|\!\!+\!{\frac{\gamma_2}{dn}} \!\!\sum_{i=1}^n \!\left\|\!{\cal M}^{i}\! (V_1 y_0 \!+\! V_2 z_0\!) \right\|_*.\nonumber
\end{align}

Next, we establish a bound for the term $\frac{1}{dn}\langle V_1 (\hat{y}-y_0) + V_2 (\hat{z}-z_0),\epsilon\rangle$ in \eqref{eq: hat_bar}. As $(\hat{y},\hat{z})$ is optimal to  \eqref{eq: re_vec_form}, there exists $\hat{\theta}\in  \sum_{i=1}^n  ({\cal M}^{i})^*  \partial  \left\| {\cal M}^{i} (V_1 \hat{y} + V_2 \hat{z}) \right\|_* $ such that
\[
0 = \frac{1}{dn} V_2^{\intercal} (V_1 \hat{y} + V_2 \hat{z} - a) + \frac{\gamma_2}{dn}V_2^{\intercal} \hat{\theta},
\]
which implies that $\hat{z} = V_2^{\intercal}a - \gamma_2 V_2^{\intercal}\hat{\theta}  = V_2^{\intercal}(V_1 y_0+ V_2 z_0+\epsilon) - \gamma_2 V_2^{\intercal}\hat{\theta} = z_0 + V_2^{\intercal}\epsilon- \gamma_2 V_2^{\intercal}\hat{\theta}$. Here $({\cal M}^{i})^*$ denotes the adjoint of the linear map ${\cal M}^{i}$. Therefore, we have
\begin{align}
&\frac{1}{dn}\left| \langle V_1 (\hat{y}-y_0) + V_2 (\hat{z}-z_0),\epsilon\rangle\right|   = \frac{1}{dn}\left| \epsilon^{\intercal} V_1 (\hat{y}-y_0) + \epsilon^{\intercal} V_2 (V_2^{\intercal}\epsilon- \gamma_2 V_2^{\intercal}\hat{\theta})\right|\notag\\
&\leq \frac{1}{dn} \left|\epsilon^{\intercal} V_1 (\hat{y}-y_0)\right| + \frac{1}{dn}\left| \epsilon^{\intercal} V_2 V_2^{\intercal}\epsilon \right| + \frac{\gamma_2}{dn} \left| \epsilon^{\intercal} V_2V_2^{\intercal}\hat{\theta}\right|.\label{eq: G}
\end{align}
Next, we establish bounds for the three terms on the right hand side of \eqref{eq: G}.

\textbf{Bound for $\frac{1}{dn} \left|\epsilon^{\intercal} V_1 (\hat{y}-y_0)\right|$:}
By noting the fact that $W^{\dagger}W = I_{d(n-k)}$, we have
\begin{align}
\left|\epsilon^{\intercal} V_1 (\hat{y}-y_0)\right|
= \left|\epsilon^{\intercal} V_1 W^{\dagger}W (\hat{y}-y_0)\right|
&\leq \left\| \epsilon^{\intercal} V_1 W^{\dagger}\right\|_{\infty}\|W (\hat{y}-y_0)\|_1\notag \\
&\leq  \left\| \epsilon^{\intercal} V_1 W^{\dagger}\right\|_{\infty} \sum_{l(i,j)\in {\cal E}} \sqrt{d}  \|W^{l(i,j)}(\hat{y}-y_0)\|. \label{eq: epsilony}
\end{align}

\vspace{-0.3cm}
\noindent Let $z \!=\! (\epsilon^{\intercal} V_1 W^{\dagger})^{\intercal} $ and we aim to bound its infinity norm. Its $i$-th entry is $z_i \!=\! \epsilon^{\intercal} V_1 W^{\dagger} e_i$, where $e_i$ is the $i$-th standard basis vector in $\mathbb{R}^{d|\mathcal{E}|}$.
For $i\in [d|\mathcal{E}|]$, we have
\begin{align*}
&\mathbb{E}[z_i]  = \mathbb{E}[\epsilon^{\intercal} V_1 W^{\dagger} e_i] = (\mathbb{E}[\epsilon])^{\intercal} V_1 W^{\dagger} e_i = 0,\\
&\mathbb{E}[z_i^2] = \mathbb{E}[e_i^{\intercal} (W^{\dagger})^{\intercal} V_1^{\intercal }\epsilon
\epsilon^{\intercal} V_1  W^{\dagger} e_i]  \!= \!\sigma^2 e_i^{\intercal} (W^{\dagger})^{\intercal}  W^{\dagger} e_i \!\leq\! \sigma^2 \sigma_{\max}( (W^{\dagger})^{\intercal}  W^{\dagger}) \!=\!\frac{\sigma^2}{\sigma_{\min}^2(B)},\\
&{\rm Var}[z_i]  = \mathbb{E}[z_i^2]  - \left(\mathbb{E}[z_i] \right)^2  = \mathbb{E}[z_i^2] \leq \frac{\sigma^2}{\sigma_{\min}^2(B)},
\end{align*}
which means that $z_i$ is a sub-Gaussian random variable with mean zero and variance at most $\frac{\sigma^2}{\sigma_{\min}^2(B)}$. According to Boole's inequality, we have that for any $t>0$,
\begin{align*}
\mathbb{P}(\|z\|_{\infty} \geq t)
\leq \sum_{i\in [d|\mathcal{E}|]} 2\exp\left(-\frac{t^2\sigma_{\min}^2(B)}{2\sigma^2}\right) = 2d|\mathcal{E}|\exp\left(-\frac{t^2\sigma_{\min}^2(B)}{2\sigma^2}\right).
\end{align*}
By \eqref{eq: epsilony} and the above inequality with $t = \frac{2\sigma}{\sigma_{\min}(B)} \sqrt{\log (d|\mathcal{E}|)}$, we  have
\begin{align}
& \mathbb{P}\left(
\frac{1}{dn} \left|\epsilon^{\intercal} V_1 (\hat{y}-y_0)\right| > \frac{2\sigma}{\sigma_{\min}(B)dn} \sqrt{d\log (d|\mathcal{E}|)}\sum_{l(i,j)\in {\cal E}}   \|W^{l(i,j)}(\hat{y}-y_0)\|  \right) \notag \\
& \qquad \leq \mathbb{P}\left(\left\| \epsilon^{\intercal} V_1 W^{\dagger}\right\|_{\infty}   \geq \frac{2\sigma}{\sigma_{\min}(B)} \sqrt{\log (d|\mathcal{E}|)} \right)
\leq \frac{2}{d|\mathcal{E}|}. \label{eq: control_2norm}
\end{align}

\textbf{Bound for $\frac{1}{dn}\left| \epsilon^{\intercal} V_2 V_2^{\intercal}\epsilon \right|$:}
According to Lemma \ref{lemma: D2}, there exist constants $c_1>0$ and $c_2>0$ such that
\begin{align*}
\mathbb{P}\left( \! \frac{1}{dn}\!\left| \epsilon^{\intercal} V_2 V_2^{\intercal}\epsilon \right| \!>\!  \sigma^2 \!\!\left[ \frac{\kappa_0}{n} \!\!+\!\!\sqrt{\frac{\kappa_0\log(dn)}{dn^2}}\right]\!\right)
\!\leq\! \exp\!\left\{\! -\!\min \!\left( c_1 \!\log(dn),c_2 \sqrt{ d\kappa_0\log(dn)}\right)\!\right\}.
\end{align*}

\textbf{Bound for $\frac{\gamma_2}{dn} \left| \epsilon^{\intercal} V_2V_2^{\intercal}\hat{\theta}\right|$:}
From Lemma \ref{lemma: subgaussian}(ii), we can see that
\begin{align}
\mathbb{P}\left(  \frac{1}{dn} \left| \epsilon^{\intercal} V_2V_2^{\intercal}\hat{\theta}\right|  >  t \right) \leq 2 \exp \left( -\frac{t^2 d^2 n^2 }{2\sigma^2\|V_2V_2^{\intercal}\hat{\theta}\|^2}\right) \leq 2 \exp \left( -\frac{t^2 d^2 n^2 }{2\sigma^2\|\hat{\theta}\|^2}\right),\label{eq: ep_theta_bound}
\end{align}
where the fact that $V_2V_2^{\intercal}$ is a projection matrix is used.
By the definition of the linear maps ${\cal M}^{i}$ and that  $\hat{\theta}\in  \sum_{i=1}^n  ({\cal M}^{i})^*  \partial  \left\| {\cal M}^{i} (V_1 \hat{y} + V_2 \hat{z}) \right\|_* $, 
we have that
\begin{align*}
{\cal M}^{i}\hat{\theta} \in \partial  \left\| {\cal M}^{i} (V_1 \hat{y} + V_2 \hat{z}) \right\|_*,\quad i \in [n].
\end{align*}
By the formulae for $\partial \|\cdot\|_*$, see, e.g., \cite{watson1992characterization}, we have
\begin{align}
\|\hat{\theta}\|^2  = \sum_{i=1}^n \|\hat{\theta}_{(i-1)d+1:id}\|^2= \sum_{i=1}^n \|{\cal M}^{i}\hat{\theta}\|_F^2\leq \sum_{i=1}^n  d_2
\leq n\sqrt{d}.\label{eq: 2normtheta}
\end{align}
By taking $t = \sigma/n^{1/4}$  
in \eqref{eq: ep_theta_bound}, we have
\begin{align*}
\mathbb{P}\left(  \frac{1}{dn} \left| \epsilon^{\intercal} V_2V_2^{\intercal}\hat{\theta}\right|
>
\frac{\sigma}{ n^{1/4} }\right)
\leq
2 \exp \left( -\frac{ d^2 n^{3/2} }{2\|\hat{\theta}\|^2}\right)
= 2 \exp \left( -\frac{ \sqrt{d^3 n}}{2 }\right).
\end{align*}
Therefore, when $\gamma_1 \geq \frac{4\sigma}{\sigma_{\min}(B)}\sqrt{d\log(d|\mathcal{E}|)}$, the inequality \eqref{eq: G} indicates that
\begin{align}
& \frac{1}{dn}\left| \langle V_1 (\hat{y}-y_0) + V_2 (\hat{z}-z_0),\epsilon\rangle\right|    \leq  \frac{\gamma_1}{2dn} \sum_{l(i,j)\in \mathcal{E}}  \left\| W^{l(i,j)} (\hat{y}-y_0)\right\| \nonumber\\
&+ \sigma^2 \left[ \frac{\kappa_0}{n} +\sqrt{\frac{\kappa_0\log(dn)}{dn^2}}\right] + \frac{\gamma_2 \sigma}{n^{1/4}} ,\label{eq: boundVepsilon}
\end{align}
with probability at least 
\begin{align*}
p_0:=1 - \frac{2}{d|\mathcal{E}|} -  2 \exp \left( -\frac{ \sqrt{d^3 n}}{2 }\right) -\exp\left\{ -\min\left(
c_1 \log(dn),c_2 \sqrt{d\kappa_0\log(dn)}\right)\right\}.
\end{align*}
Finally, by plugging the inequality \eqref{eq: boundVepsilon} to the inequality \eqref{eq: hat_bar}, we have
\begin{align*}
&\frac{1}{2dn} \|\hat{x}-x_0\|^2  \leq \frac{\gamma_1}{2dn} \sum_{l(i,j)\in \mathcal{E}}  \left\| W^{l(i,j)} (\hat{y}-y_0)\right\|
+ \sigma^2 \left[ \frac{\kappa_0}{n} +\sqrt{\frac{\kappa_0\log(dn)}{dn^2}}\right] + \frac{\gamma_2 \sigma}{n^{1/4}} \\
&+ \frac{\gamma_1}{dn} \!\!\sum_{l(i,j)\in \mathcal{E}}\!\! w_{ij}\! \left\| W^{l(i,j)} y_0\right \|  + {\frac{\gamma_2}{dn}} \sum_{i=1}^n \left\|{\cal M}^{i} (V_1 y_0 + V_2 z_0) \right\|_* - \frac{\gamma_1}{dn} \!\!\sum_{l(i,j)\in \mathcal{E}}\!\! w_{ij}\! \left\| W^{l(i,j)} \hat{y}\right \|,
\end{align*}
with probability at least $p_0$. Moreover, it can be seen that
\begin{align*}
& \frac{\gamma_1}{2dn} \sum_{l(i,j)\in \mathcal{E}}  \left\| W^{l(i,j)} (\hat{y}-y_0)\right\|  + \frac{\gamma_1}{dn} \sum_{l(i,j)\in \mathcal{E}} w_{ij} \left\| W^{l(i,j)} y_0\right \| - \frac{\gamma_1}{dn} \sum_{l(i,j)\in \mathcal{E}} w_{ij} \left\| W^{l(i,j)} \hat{y}\right \|\\
&\leq \ \frac{\gamma_1}{2dn}
\sum_{l(i,j)\in \mathcal{E}}\left(
(1+2 w_{ij}) \left\| W^{l(i,j)} y_0 \right\|  +(1- 2 w_{ij}) \left\| W^{l(i,j)} \hat{y}\right\| \right)\\
&\leq \ \frac{\gamma_1}{2dn}
\sum_{l(i,j)\in \mathcal{E}}
(1+2 w_{ij}) \left\| W^{l(i,j)} y_0\right\|
\leq \ \frac{\gamma_1}{2dn}
\sum_{l(i,j)\in \mathcal{E}}
(1+2 w_{ij}) \left\| D^{l(i,j)} x_0\right\|,
\end{align*}
where the assumption  $\min_{l(i,j)\in \mathcal{E}} w_{ij}\geq \frac{1}{2}$ is used in the second inequality. This completes the proof.
\end{proof}

\section{Optimization Method for Solving \eqref{model-convex}}\label{sec:opt_method}
In this section, we propose a highly efficient algorithm for solving the lrCC model \eqref{model-convex}. Specifically, we design a double-loop algorithm  that employs 
a superlinearly convergent augmented Lagrangian method (ALM) for solving 
\eqref{model-convex}, with subproblems addressed by
a quadratically convergent semismooth Newton method. Our approach leverages the second-order sparsity and low rank structure of the solution to markedly enhance computational efficiency.

\subsection{Outer Loop: An Augmented Lagrangian Method for \eqref{model-convex}}
To simplify notation in the design of algorithms, we express \eqref{model-convex} in vector form. We denote
\begin{align*}
g(y) := \gamma_1 \sum_{l(i,j)\in \mathcal{E}} w_{ij} \| \mathcal{P}_{l(i,j)} y\|,\quad y\in \mathbb{R}^{d|\mathcal{E}|}, \quad\quad
h(x) := \gamma_2 \sum_{i=1}^n \|{\cal M}^{i} x \|_*,\quad x\in \mathbb{R}^{dn},
\end{align*}
where for each $l(i,j)\in \mathcal{E}$, the linear map $\mathcal{P}_{l(i,j)}: \mathbb{R}^{d|\mathcal{E}|}\rightarrow \mathbb{R}^d$ is defined as $\mathcal{P}_{l(i,j)}y=(\zeta_{l(i,j)}^{\intercal} \otimes I_d )y$ with $\zeta_{l(i,j)}$ being the $l(i,j)$-th standard basis vector in $\mathbb{R}^{|\mathcal{E}|}$,
and the linear map ${\cal M}^{i}$ is defined as in \eqref{eq: new_vec_form}. With $\mathcal{D}$ specified in Section \ref{sec: thm_finite_sample}, the lrCC model \eqref{model-convex} can then be written into the vector form
\begin{equation}
\label{Eq: Convex-composite-program}
\min_{x\in \mathbb{R}^{dn}} \ \frac{1}{2}\|x-a\|^2 + g(\mathcal{D} x) + h(x).
\end{equation}
The term $g(\cdot)$ imposes {\it group lasso regularization} \cite{yuan2006model} on the pairwise differences of the centroids, aiming to limit the number of unique centroids. And the term $h(\cdot)$, which we refer to as the {\it segmented nuclear norm regularization}, enforces the low-rankness of the centroids.

To facilitate our algorithmic design, we express the problem \eqref{Eq: Convex-composite-program} in the following constrained form:
\begin{equation}
\min_{x,z\in \mathbb{R}^{dn},\ y\in \mathbb{R}^{d|\mathcal{E}|}} \Big\{\frac{1}{2}\|x-a\|^2 + g(y) + h(z) \ \Big| \ \mathcal{D} x - y = 0,\ x - z = 0 \Big\}.\tag{P}\label{primal_problem}
\end{equation}
Then the Lagrangian function of \eqref{primal_problem} takes the form of
\begin{align*}
l(x, y, z; v, w) = \frac{1}{2}\|x-a\|^2 + g(y) + h(z) + \langle v, \mathcal{D}x - y \rangle + \langle w, x - z\rangle,
\end{align*}
for $(x,y,z,v,w)\in \mathbb{R}^{dn}\times \mathbb{R}^{d|\mathcal{E}|} \times \mathbb{R}^{dn}\times \mathbb{R}^{d|\mathcal{E}|} \times \mathbb{R}^{dn}$. By finding $\inf_{x,y,z} l(x, y, z; v, w)$, we can derive the dual problem of \eqref{primal_problem} as follows
\begin{equation}
\max_{v\in \mathbb{R}^{d|\mathcal{E}|},\ w\in \mathbb{R}^{dn}} \ \Big\{-\frac{1}{2}\|\mathcal{D}^* v + w - a\|^2 - g^{*}(v) - h^{*}(w) + \frac{1}{2}\|a\|^2 \Big\}, \tag{D}\label{dual_problem}
\end{equation}
where $g^*(\cdot)$ and $h^*(\cdot)$ are the Fenchel conjugate functions of $g(\cdot)$ and $h(\cdot)$, respectively. The Karush-Kuhn-Tucker (KKT) conditions for \eqref{primal_problem} and \eqref{dual_problem} are as follows:
\begin{align}
\begin{array}{l}
x-a + {\cal D}^*v + w = 0, \
y - {\rm Prox}_{g}(v + y) = 0, \
z - {\rm Prox}_{h}(w + z) = 0, \\
{\cal D} x - y = 0, \
x - z = 0.
\end{array}
\label{eq: kkt}
\end{align}

We are now prepared to develop the ALM for solving \eqref{primal_problem}, which will also  yield a solution to the dual problem \eqref{dual_problem} as a byproduct. For a given parameter $\sigma > 0$, the augmented Lagrangian function associated with \eqref{primal_problem} is given by
\begin{align*}
&{\cal L}_{\sigma}(x, y, z; v, w) = l(x, y, z; v, w) + \frac{\sigma}{2}\|{\cal D}x - y\|^2 + \frac{\sigma}{2} \|x - z\|^2\\
=&\frac{1}{2}\|x\!-\!a\|^2 + g(y) + h(z) + \frac{\sigma}{2}\left\|{\cal D}x \!-\! y\!+\! \frac{1}{\sigma} v\right\|^2\!\! + \frac{\sigma}{2} \left\|x \!-\! z \!+ \!\frac{1}{\sigma} w\right\|^2\!\! - \frac{1}{2\sigma}\|v\|^2 \!- \frac{1}{2\sigma}\|w\|^2.
\end{align*}
We describe the ALM for solving \eqref{primal_problem} in Algorithm \ref{alg:alm}. And we provide its global convergence  in  Theorem \ref{thm: global_alm}.

\begin{algorithm}[H]
	\caption{Augmented Lagrangian method for \eqref{primal_problem}}
	\label{alg:alm}
	\begin{algorithmic}
		\STATE {\bfseries Initialization:} Choose $ (v^0, w^0) \in  \mathbb{R}^{d  |\mathcal{E}|} \times \mathbb{R}^{d  n}$, and $\sigma_0 > 0$. Choose $\varepsilon_k\geq 0,\ \sum_{k=0}^{\infty} \varepsilon_k <\infty$. Set $k=0$.
		\REPEAT
		\STATE {\bfseries Step 1}. Solve the ALM subproblem
		\begin{equation}\label{eq: primal_update}
		(x^{k+1}, y^{k+1}, z^{k+1}) \approx \underset{x,z\in \mathbb{R}^{d n},\ y\in \mathbb{R}^{d  |\mathcal{E}|} }{\arg\min} \ \Big\{\Phi_{k}(x, y, z) := \mathcal{L}_{\sigma_k}(x, y, z; v^k, w^k)\Big\},
		\end{equation}
        such that
        \begin{align}
    \Phi_k(x^{k+1},y^{k+1},z^{k+1})-\inf\Phi_{k}\leq \varepsilon_k^2/(2\sigma_k).
    \label{eq: stopA}
        \end{align}
        \vspace{-0.5cm}
		\STATE {\bfseries Step 2}. Update the multipliers
        \begin{align*}
        v^{k+1} = v^k + \sigma_k({\cal D}x^{k+1} - y^{k+1}),\quad
        w^{k+1} = w^k + \sigma_k(x^{k+1} - z^{k+1}).
        \end{align*}
        \vspace{-0.5cm}
		\STATE {\bfseries Step 3}. Update $\sigma_{k+1} \uparrow \sigma_{\infty} \leq \infty$.
		\UNTIL{
        A prespecified stopping criterion
        is satisfied.}
	\end{algorithmic}
\end{algorithm}

\begin{theorem}\label{thm: global_alm}
Let $\{(x^k,y^k,z^k,v^k,w^k)\}$ be the infinite sequence generated by Algorithm \ref{alg:alm}. Then, the sequence $\{(x^k,y^k,z^k)\}$ is bounded and converges to the unique optimal solution $(x^*,y^*,z^*)$ of \eqref{primal_problem}. In addition, $\{(v^k,w^k)\}$ is also bounded and converges to an optimal solution of \eqref{dual_problem}, denoted as $(v^*,w^*)$.
\end{theorem}
\begin{proof}
The existence and uniqueness of the optimal solution to \eqref{primal_problem} can be easily seen from the fact that its equivalent problem \eqref{Eq: Convex-composite-program} is coercive and strongly convex. According to \cite[Corollary 31.2.1]{rockafellar1997convex}, we have that the problem \eqref{dual_problem} has a nonempty solution set. Moreover, the optimal value of \eqref{dual_problem} matches the optimal value of \eqref{primal_problem}, which further implies that the solution set to the KKT system \eqref{eq: kkt} is nonempty. Together with \cite[Theorem 4]{rockafellar1976augmented}, we can see that $\{(v^k,w^k)\}$ is bounded and converges to an optimal solution of \eqref{dual_problem}. The coerciveness further indicates that the sequence $\{(x^k,y^k,z^k)\}$ is bounded and converges to the unique optimal solution of \eqref{primal_problem}. This completes the proof.
\end{proof}

In order to obtain the local convergence results of Algorithm \ref{alg:alm}, we need the following stopping criteria for the subproblems of ALM, in addition to \eqref{eq: stopA}:
\begin{align}
&\Phi_k(x^{k+1},y^{k+1},z^{k+1}) - \inf \Phi_k \leq   (\delta_k^2/(2\sigma_k) )
\|(v^{k+1},\!w^{k+1})\!-\!(v^k,\!w^k)\|^2,\label{eq: stopB1}\\
&{\rm dist}(0,\partial \Phi_k(x^{k+1},y^{k+1},z^{k+1}) ) \leq
(\delta'_k/\sigma_k) \|(v^{k+1},w^{k+1}) - (v^k,w^k)\|, \label{eq: stopB2}
\end{align}
where $\delta_k\geq 0$, $\sum_{k=0}^{\infty} \delta_k <\infty$ and $0\leq \delta'_k\rightarrow 0$. The local Q-superlinear convergence result of the dual sequence $\{(v^k,w^k)\}$ generated by Algorithm \ref{alg:alm} is provided in the following theorem. For notational simplicity, we denote the dual objective function $\varphi(v,w):=\frac{1}{2}\|\mathcal{D}^* v + w - a\|^2 + g^{*}(v) + h^{*}(w)$. Then the solution set to the dual problem \eqref{dual_problem} is $(\partial \varphi)^{-1}(0)$,
which is nonempty according to Theorem \ref{thm: global_alm}.

\begin{theorem}\label{thm: local_dual}
Let $\{(x^k,y^k,z^k,v^k,w^k)\}$ be the infinite sequence generated by Algorithm \ref{alg:alm}, where the subproblem \eqref{eq: primal_update} is stopped when both \eqref{eq: stopA} and \eqref{eq: stopB1} are satisfied. Suppose there exist positive constants $\epsilon$ and $\kappa_{\varphi}$ such that for any $(v,w)\in \mathbb{R}^{d|\mathcal{E}|}\times \mathbb{R}^{dn}$ with ${\rm dist}(0,\partial \varphi(v,w))\leq \epsilon$, it holds that
\begin{align}
{\rm dist}((v,w),(\partial \varphi)^{-1}(0)) \leq \kappa_{\varphi}  {\rm dist}(0, \partial \varphi(v,w)). \label{eq: error_bound}
\end{align}
Then for $k$ sufficiently large, we have
\begin{align*}
{\rm dist}((v^{k+1},w^{k+1}),(\partial \varphi)^{-1}(0)) \leq \theta_k {\rm dist}((v^k,w^k),(\partial \varphi)^{-1}(0)) ,
\end{align*}
where $\theta_k = (\kappa_{\varphi} (\kappa_{\varphi}^2+\sigma_k^2)^{-1/2} + 2\delta_k)(1-\delta_k)^{-1}\rightarrow \theta_{\infty}=\kappa_{\varphi}(\kappa_{\varphi}^2+\sigma_{\infty}^2)^{-1/2}<1$ as $k\rightarrow \infty$. This, together with Theorem \ref{thm: global_alm}, means that the sequence $\{(v^k,w^k)\}$ converges to $(v^*,w^*)$ at a Q-linear rate. In particular, when $\sigma_{\infty}=\infty$, we have $\theta_{\infty}=0$, indicating that the convergence is Q-superlinear.
\end{theorem}
\begin{proof}
First, from \cite[Proposition 7]{rockafellar1976augmented} and the condition \eqref{eq: error_bound}, we can see that, for any
$u\in \mathbb{R}^{d|\mathcal{E}|}\times \mathbb{R}^{dn}$ such that $\|u\|\leq \epsilon$, and any $(v,w)\in (\partial \varphi)^{-1}(u)$, we have
\begin{align*}
{\rm dist}((v,w),(\partial \varphi)^{-1}(0)) \leq \kappa_{\varphi}  {\rm dist}(0, \partial \varphi(v,w))\leq \kappa_{\varphi} \|u\|.
\end{align*}
In addition, we can see from \eqref{eq: stopB1} and \cite[Proposition 6]{rockafellar1976augmented} that
\begin{align*}
\|(v^{k+1},w^{k+1})-(I+\sigma_k \partial \varphi)^{-1} (v^{k},w^{k})\|^2
&\leq 2\sigma_k\left(\Phi_k(x^{k+1},y^{k+1},z^{k+1}) - \inf \Phi_k\right)\\
&\leq \delta_k^2 \|(v^{k+1},w^{k+1})-(v^k,w^k)\|^2.
\end{align*}
Then according to \cite[Theorem 2.1]{luque1984asymptotic}, we have the desired conclusion.
\end{proof}

In order to study the convergence rate of the primal sequence $\{(x^k,y^k,z^k)\}$, which is of primary interest, we require one more condition on the $\mathcal{T}_l$, which is defined as $\mathcal{T}_l(x,y,z;v,w):=\{(x',y',z',v',w')\mid (x',y',z',-v',-w')\in \partial l(x,y,z;v,w) \}$. Details are provided in the following theorem.
\begin{theorem}
Assume the assumptions in Theorem \ref{thm: local_dual} hold and the stopping criterion \eqref{eq: stopB2} is  satisfied. Furthermore, assume that $\mathcal{T}_l$ is metrically subregular at $(x^*,y^*,z^*,v^*,w^*)$ for the origin with modulus $\kappa_l$, that is, there exist neighborhoods $U$ of $(x^*,y^*,z^*,v^*,w^*)$ and $V$ of the origin such that
\begin{align*}
{\rm dist}((x,y,z,v,w), \mathcal{T}_l^{-1}(0)) \leq \kappa_l {\rm dist}(0,\mathcal{T}_l(x,y,z,v,w) \cap V),\quad \forall \ (x,y,z,v,w)\in U.
\end{align*}
Then for $k$ sufficiently large, we have
\begin{align*}
\|(x^{k+1},y^{k+1},z^{k+1})-(x^*,y^*,z^*)\| \leq \xi_k {\rm dist}((v^k,w^k),(\partial \varphi)^{-1}(0)) ,
\end{align*}
where $\xi_k = \sqrt{1+(\delta'_k)^2}\kappa_l  \sigma_k^{-1}(1-\delta_k)^{-1} \rightarrow \xi_{\infty}=\kappa_l \sigma_{\infty}^{-1}$ as $k\rightarrow \infty$. That is, the sequence $\{(x^k,y^k,z^k)\}$ converges to $(x^*,y^*,z^*)$ at a R-linear rate provided that $\sigma_{\infty}>\kappa_l$. In particular, when $\sigma_{\infty}=\infty$, we have $\xi_{\infty}=0$,  indicating that the convergence is R-superlinear.
\end{theorem}
\begin{proof}
As noted in \cite[Equation (2.11)]{rockafellar1976augmented}, $\mathcal{T}_l^{-1}(0)$ consists of all $(x,y,z,v,w)$ satisfying the KKT condition \eqref{eq: kkt}, which, together with the proof of Theorem \ref{thm: global_alm}, indicates that $\mathcal{T}_l^{-1}(0) = \{(x^*,y^*,z^*,v,w) \mid (v,w)\in (\partial \varphi)^{-1}(0)\}$. The global convergence 
of $\{(x^k,y^k,z^k,v^k,w^k)\}$ 
in Theorem \ref{thm: global_alm} implies that for $k$ sufficiently large,
\begin{align*}
&\quad \|(x^{k+1},y^{k+1},z^{k+1})-(x^*,y^*,z^*)\|^2 + {\rm dist}^2((v^{k+1},w^{k+1}),(\partial \varphi)^{-1}(0))  \\
&={\rm dist}^2((x^{k+1},y^{k+1},z^{k+1},v^{k+1},w^{k+1}), \mathcal{T}_l^{-1}(0))\leq \kappa_l^2 {\rm dist}^2(0,\mathcal{T}_l(x,y,z,v,w)),
\end{align*}
where the last inequality follows from the metric subregularity of $\mathcal{T}_l$. According to \cite[Equation (4.21)]{rockafellar1976augmented}, we have that
\begin{align*}
{\rm dist}^2(0,\!\mathcal{T}_l(x,\!y,\!z,\!v,\!w)) \!\leq\!   {\rm dist}^2(0,\!\partial \Phi_k(x^{k+1},\!y^{k+1},\!z^{k+1}) ) \!+\! \sigma_k^{-2}\|(v^{k+1},\!w^{k+1})\!-\!(v^{k},\!w^{k})\|^2.
\end{align*}
As the stopping criterion \eqref{eq: stopB2} is satisfied, we further have that for $k$ sufficiently large,
\begin{align*}
& \ \|(x^{k+1},y^{k+1},z^{k+1})-(x^*,y^*,z^*)\|^2 \leq   \kappa_l^2 {\rm dist}^2(0,\mathcal{T}_l(x,y,z,v,w))\\
\leq & \ \kappa_l^2 {\rm dist}^2(0,\partial \Phi_k(x^{k+1},y^{k+1},z^{k+1}) ) + (\kappa_l^2/\sigma_k^{2})\|(v^{k+1},w^{k+1})-(v^{k},w^{k})\|^2\\
\leq & \ (\kappa_l\delta'_k/\sigma_k)^2 \|(v^{k+1},w^{k+1}) - (v^k,w^k)\|^2+ (\kappa_l^2/\sigma_k^{2})\|(v^{k+1},w^{k+1})-(v^{k},w^{k})\|^2\\
= & \  \left( (\delta_k')^2 + 1 \right) (\kappa_l\sigma_k^{-1})^2 \|(v^{k+1},w^{k+1}) - (v^k,w^k)\|^2.
\end{align*}
Then, from \cite[Lemma 3]{cui2019r} and the stopping criterion \eqref{eq: stopB1}, we have for all $k\geq 0$
\begin{align*}
\|(v^{k+1},w^{k+1}) - (v^k,w^k)\| \leq (1-\delta_k)^{-1}{\rm dist}((v^k,w^k), (\partial \varphi)^{-1}(0)).
\end{align*}
Therefore, by combining the above results, we can see that for $k$ sufficiently large,
\begin{align*}
\|(x^{k+1},\!y^{k+1},\!z^{k+1})\!-\!(x^*,\!y^*,\!z^*)\|\! \leq\! \sqrt{1\!+\!(\delta'_k)^2}\kappa_l  \sigma_k^{-1}(1\!-\!\delta_k)^{-1}  {\rm dist}((v^k,\!w^k),\! (\partial \varphi)^{-1}(0)).
\end{align*}
This completes the proof.
\end{proof}

Note that the main difficulty in implementing Algorithm \ref{alg:alm} lies in solving the ALM subproblem \eqref{eq: primal_update}, which is discussed in detail in the subsequent subsection.

\subsection{Inner Loop: A Semismooth Newton-CG Method for Subproblem \eqref{eq: primal_update}}
In this subsection, we will design a semismooth Newton-CG algorithm to solve the ALM subproblem \eqref{eq: primal_update}. Given $\sigma > 0$ and $ (\tilde{v}, \tilde{w}) \in\mathbb{R}^{d |\mathcal{E}|} \times \mathbb{R}^{d n}$,  we essentially need to solve the following problem
\begin{equation}
\label{eq: Phi-func}
\min_{x,z\in \mathbb{R}^{d n}, \ y\in \mathbb{R}^{d |\mathcal{E}|}} \ \Big\{\Phi(x, y, z) := \mathcal{L}_{\sigma}(x, y, z; \tilde{v}, \tilde{w})\Big\}.
\end{equation}
Since $\Phi(\cdot, \cdot, \cdot)$ is strongly convex, the above minimization problem admits a unique optimal solution, denoted as $(\bar{x}, \bar{y}, \bar{z})$. For any given $x$, we observe that minimizing \eqref{eq: Phi-func} with respect to $y$ and $z$ is attained at
$
{y} = {\rm Prox}_{g/\sigma}\left({\cal D}{x} + \frac{1}{\sigma}\tilde{v}\right)$, and $
{z} = {\rm Prox}_{h/\sigma}\left({x} + \frac{1}{\sigma} \tilde{w}\right)
$.
Substituting these expressions back into \eqref{eq: Phi-func}, the objective function becomes
$\phi:\mathbb{R}^{d n} \rightarrow \mathbb{R} $:
\begin{align}
& \phi(x)
= \frac{1}{2}\|x\!-\!a\|^2  + g\!\left(\!{\rm Prox}_{g/\sigma}\!(\!{\cal D}x \!+ \!\frac{1}{\sigma}\tilde{v})\!\right)\!\!+\!\frac{\sigma}{2}\left\|({\cal D}x \!+\! \frac{1}{\sigma}\tilde{v})\!-\! {\rm Prox}_{g/\sigma}\!(\!{\cal D}x \!+\! \frac{1}{\sigma} \tilde{v})\!\right\|^2\label{eq: phi-func}
\\
& \!\!+ h\!\left(\!{\rm Prox}_{h/\sigma}\!(x \!+\! \frac{1}{\sigma} \tilde{w})\!\right)\!\!+\!\frac{\sigma}{2}\left\|(x \!+\! \frac{1}{\sigma} \tilde{w})\!-\! {\rm Prox}_{h/\sigma}(x \!+\! \frac{1}{\sigma} \tilde{w})\!\right\|^2\!-\! \frac{1}{2\sigma}(\|\tilde{v}\|^2 \!+\! \|\tilde{w}\|^2).\notag
\end{align}
Then, we can compute the solution $(\bar{x}, \bar{y}, \bar{z})$ of the ALM subproblem \eqref{eq: Phi-func}
as
\begin{align}
\bar{x} = \underset{x\in \mathbb{R}^{d n}}{\arg\min}\ \phi(x), \quad \bar{y} = {\rm Prox}_{g/\sigma}\left({\cal D}\bar{x} + \frac{1}{\sigma}\tilde{v}\right), \quad
\bar{z} = {\rm Prox}_{h/\sigma}\left(\bar{x} + \frac{1}{\sigma} \tilde{w}\right). \label{eq: updates_SSN}
\end{align}

We can see that $\phi(\cdot)$ is strongly convex. Moreover, it is continuously differentiable with the gradient
\begin{align*}
\nabla \!\phi(x) \!\!= \!x\!-\!a \!+\! \sigma\!{\cal D}^{*}\!\!\left(\!{\cal D}x \!+\! \frac{1}{\sigma} \tilde{v}\!\right)\!\!-\!\sigma{\cal D}^{*}\!{\rm Prox}_{g\!/\!\sigma}\!\!\!\left(\!\!{\cal D}x \!+\! \frac{1}{\sigma}\tilde{ v}\!\!\right) \!\!+\! \sigma\!\left(\!x \!+\! \frac{1}{\sigma} \tilde{w}\!\right)\!\!-\!\sigma {\rm Prox}_{h\!/\!\sigma}\!\left(\!x \!+\! \frac{1}{\sigma} \tilde{w}\!\right)\!,
\end{align*}
which further implies that $\nabla \phi(\cdot)$ is nondifferentiable but Lipschitz continuous. The strong convexity and continuous differentiability of $\phi(\cdot)$ indicate that $\bar{x}$ can be calculated by solving the following nonlinear and nonsmooth equation
\begin{equation}
\label{eq: newton-system}
\nabla \phi(x) = 0.
\end{equation}
Define the multifunction $\hat{\partial}^2 \phi:\mathbb{R}^{d n}\rightrightarrows \mathbb{R}^{d  n\times dn}$ as
\begin{equation}\label{eq: nabla-phi}
\hat{\partial}^2 \phi (x)  \!:=\!
\left\{ \!I_{dn} \!+\! \sigma{\cal D}^*(I_{d|\mathcal{E}|} \!-\! {\cal W}){\cal D} + \sigma(I_{dn} \! - \! {\cal Q}) \left |
\begin{array}{l}
{\cal W} \in \partial {\rm Prox}_{g/\sigma}({\cal D}x \!+\! \tilde{v}/\sigma) \\
{\cal Q} \in \partial {\rm Prox}_{h/\sigma}(x \!+ \!\tilde{w}/\sigma)
\end{array}\right.\!\!\!\!
\right\},
\end{equation}
where $\partial {\rm Prox}_{g/\sigma}({\cal D}x + \tilde{v}/\sigma)$ and $\partial {\rm Prox}_{h/\sigma}(x + \tilde{w}/\sigma)$ are the generalized Jacobians of the Lipschitz continuous mappings ${\rm Prox}_{g/\sigma}(\cdot)$ and ${\rm Prox}_{h/\sigma}(\cdot)$ at ${\cal D}x + \tilde{v}/\sigma$ and $x + \tilde{w}/\sigma$, respectively.

\begin{algorithm}[!ht]
	\caption{Semismooth Newton-CG method for \eqref{eq: newton-system}}
	\label{alg:ssncg}
	\begin{algorithmic}
		\STATE {\bfseries Initialization:} Given $\bar{x}^{0} \in \mathbb{R}^{d  n}$, $\bar{\mu} \in (0, 1/2)$, $\bar{\tau} \in (0, 1]$, and $\bar{\gamma}, \bar{\delta} \in (0, 1)$. For $j = 0, 1, 2, \dots$
		\REPEAT
		\STATE {\bfseries Step 1}. Select an element ${\cal H}_j \in \hat{\partial}^{2} \phi(\bar{x}^j)$. Apply the conjugate gradient (CG) method to find an approximate solution $d^j \in \mathbb{R}^{d n}$ to
		\begin{equation}\label{eq: cg-system}
		{\cal H}_j(d) \approx - \nabla \phi(\bar{x}^j)
		\end{equation}
		such that $\|\nabla\phi(\bar{x}^j)+{\cal H}_j(d^j) \| \leq \min(\bar{\gamma}, \|\nabla\phi(\bar{x}^j)\|^{1+\bar{\tau}})$.
		\STATE {\bfseries Step 2}. (Line search) Set $\alpha_j = \bar{\delta}^{m_j}$, where $m_j$ is the first non-negative integer $m$ for which
		$\phi(\bar{x}^j + \bar{\delta}^m d^j) \leq \phi(\bar{x}^j) + \bar{\mu}\bar{\delta}^m \langle \nabla\phi(\bar{x}^j), d^j \rangle$.
		\STATE{\bfseries Step 3}. Set $\bar{x}^{j+1} = \bar{x}^j + \alpha_j d^j$.
		\UNTIL{A stopping criterion  based on $\|\nabla\phi(\bar{x}^{j+1})\|$ is satisfied.}
	\end{algorithmic}
\end{algorithm}

Next in Algorithm \ref{alg:ssncg}, we present the semismooth Newton-CG method for solving \eqref{eq: newton-system}, which can be regarded as a generalized Newton method for solving nonsmooth
equations. The convergence result is given in Theorem \ref{thm:ssn_convergence}.

\begin{theorem}\label{thm:ssn_convergence}
Let $\{\bar{x}^j\}$ be the infinite sequence generated by Algorithm \ref{alg:ssncg}. Then $\{\bar{x}^j\}$ converges to $\bar{x}$, and
\begin{align*}
\|\bar{x}^{j+1}-\bar{x}\| = O(\|\bar{x}^{j}-\bar{x}\|^{1+\bar{\tau}}).
\end{align*}
\end{theorem}
\begin{proof}
For any $x\in \mathbb{R}^{dn}$, it can be seen that all elements in $\hat{\partial}^2 \phi (x)$ are symmetric and positive definite, which, together with \cite[Proposition 3.3]{zhao2010newton}, indicates that Algorithm \ref{alg:ssncg} is well defined. By mimicking the proof of \cite[Theorem 3.4]{zhao2010newton}, we can see that $\{\bar{x}^j\}$ converges to $\bar{x}$.

As it can be seen in Appendix~\ref{sec:proximal_mapping}, ${\rm Prox}_{g/\sigma}$ and ${\rm Prox}_{h/\sigma}$ are strongly semismooth with respect to $\partial {\rm Prox}_{g/\sigma}$ and $\partial {\rm Prox}_{h/\sigma}$, respectively. This implies that $\nabla \phi$ is strongly semismooth with respect to $\hat{\partial}^2\phi$ in \eqref{eq: nabla-phi}. The convergence rate can then be established as in \cite[Theorem 3]{li2018efficiently}. This completes the proof.
\end{proof}

While the proposed double-loop algorithm demonstrates strong theoretical promise, its practical implementation posed several challenges that required careful consideration. These include the design of implementable stopping criteria, as well as the need for efficient procedures to compute the proximal mappings of $g(\cdot)$ and $h(\cdot)$, together with their generalized Jacobian. A detailed discussion can be found in Appendix~\ref{sec:implementation}.

\section{Numerical Experiments}\label{sec:num_exp}
In this section, we compare the clustering performances of three algorithms: (i) our low rank convex clustering algorithm, denoted as ``lrCC'', (ii) low rank Lloyd's algorithm \cite[Algorithm~1]{lyu2022optimal} with tensor based spectral initialization \cite[Algorithm~2]{lyu2022optimal}, denoted as ``lr-Lloyd*'', (iii) low rank Lloyd's algorithm \cite[Algorithm~1]{lyu2022optimal} with random initialization, denoted as ``lr-Lloyd''. We evaluate clustering performance using two standard metrics: the adjusted rand index (ARI) \cite[Equation~(5)]{hubert1985comparing} and normalized mutual information NMI \cite[Equations~(2) and (4)]{kvalseth1987entropy}. Both metrics range from $0$ and $1$, with higher values indicating better clustering performance. All experiments were performed in Matlab (version 9.11) on a Windows workstation (32-core, Intel Xeon Gold 6226R @ 2.90GHz, 128 GB of RAM).

\subsection{Two Quarter-Sphere}
We conduct experiments on two quarter-sphere data, where lrCC can successfully clusters all data points (see Figures~\ref{fig:sphere1} and \ref{fig:sphere_3methods}).
We generate low rank matrices with three nonzero singular values that form two distinct clusters. The singular values of the two clusters are interleaved in three-dimensional space, as illustrated in the left panel of Figure~\ref{fig:sphere1}. Each cluster forms a quarter-sphere, giving rise to the name  ``two quarter-sphere'' data. The data generation process is inspired by the well known benchmark  problem ``two half-moon'' classification, by increasing the dimensionality of nonzero singular values from two to three.
\begin{figure}[!ht]
    \centering
\includegraphics[height=0.2\textheight]{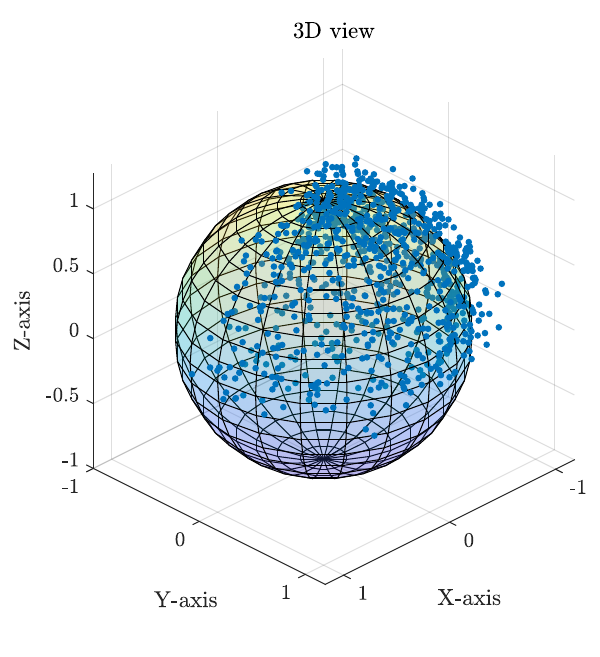}\,
\includegraphics[height=0.2\textheight]{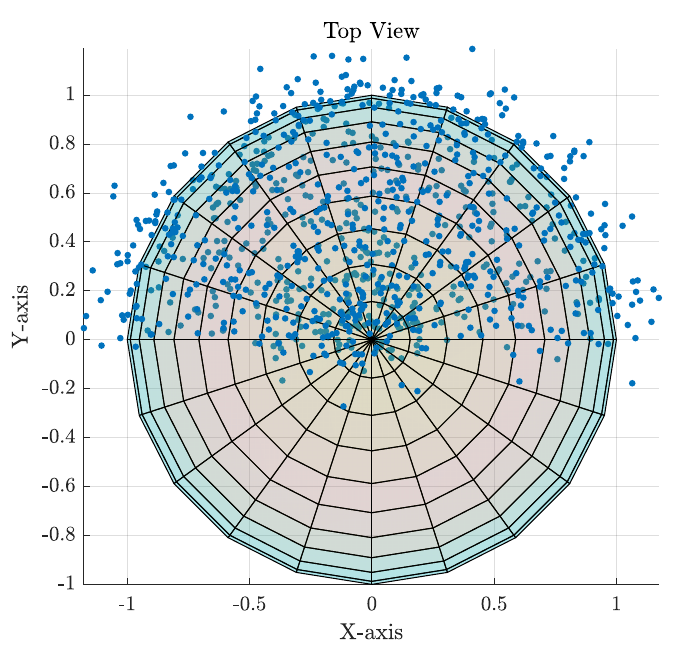}\,
\includegraphics[height=0.2\textheight]{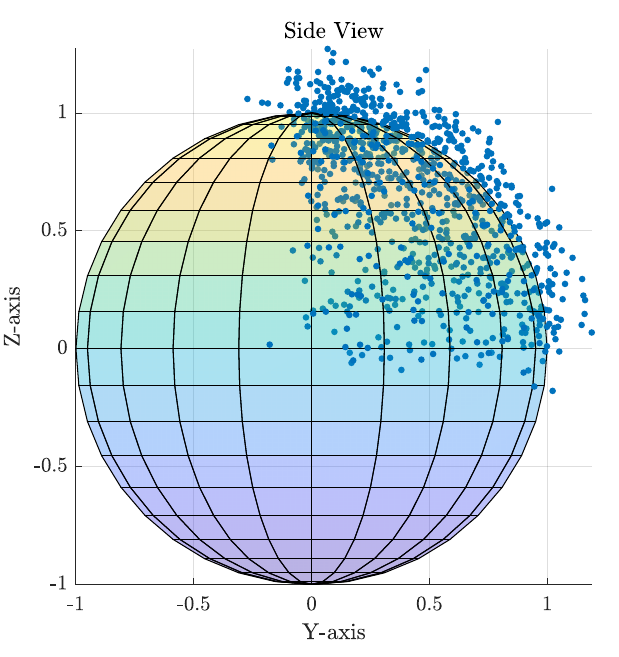}
\caption{Blue points are the singular value vectors $s_i\in\mathbb{R}^3, \ i=1,\dots,1000$ for  observations $A_i, \ i=1,\dots,1000$ in the first cluster \eqref{eq:sphere1},  lying around a quarter-sphere.} \label{fig:viewsphere}
\end{figure}

First, observations from the first cluster are generated independently as follows:
\begin{align}
\begin{array}{l}
s_i=    [ \sin{\vartheta_i}\cos{\theta_i}; \ \sin{\vartheta_i}\sin{\theta_i}; \ \cos{\vartheta_i} ] + 0.1\varepsilon_{i}, \\
A_i = U {\rm Diag}( s_i ) V^{\intercal},  \mbox{where }    \theta_i \sim {\rm Unif}[0,\pi], \, \vartheta_i \sim {\rm Unif}[0,\pi/2], \, \varepsilon_{i} \sim N(0,I_3).
\end{array}
\label{eq:sphere1}
\end{align}
Here, $U\in\mathbb{R}^{d_1\times 3}$ and $V\in\mathbb{R}^{d_2\times 3}$ are the top $3$ left and right singular vectors of a random $d_1\times d_2$ matrix with entries drawn from $N(0,1)$. The notation ${\rm Unif}[a,b]$ denotes the uniform distribution on the interval $[a,b]$. The construction of $s_i$  primarily relies the polar coordinates in three-dimensional spaces.
Notably, the vectors $s_i$ of the singular values of the observations $A_i$ are noisy samples from a quarter of a sphere, as illustrated in Figure~\ref{fig:viewsphere}. From the three-dimensional, top, and side views , we can see clearly from Figure~\ref{fig:viewsphere} that $s_i$'s are arranged around a quarter-sphere.
\begin{figure}[!ht]
    \centering
    \includegraphics[width=0.28\textwidth]{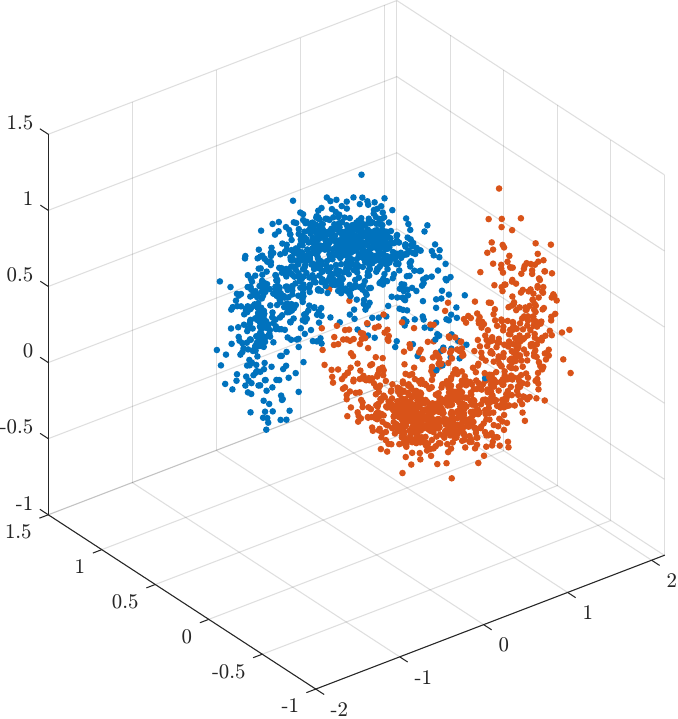}\,\,
    \includegraphics[width=0.28\textwidth]{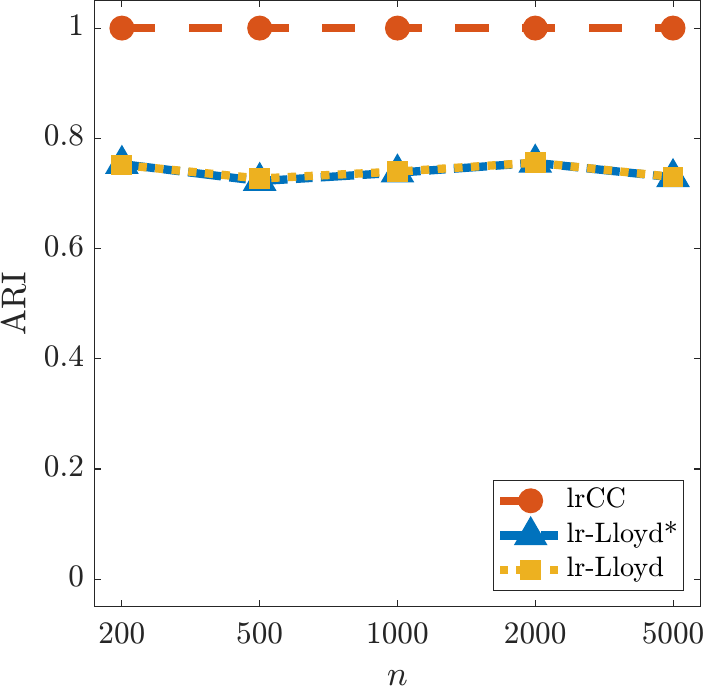}\,\,
    \includegraphics[width=0.28\textwidth]{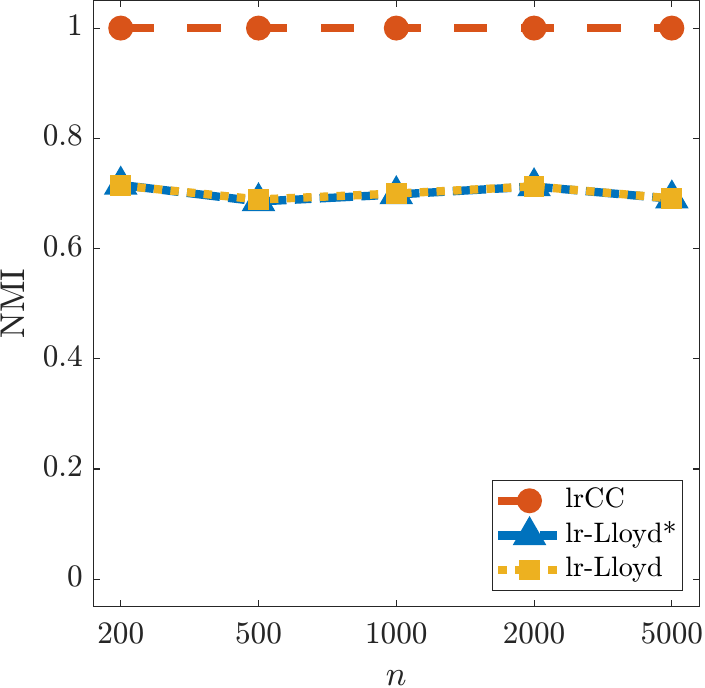}
    \caption{Left:  singular value vectors $s_i$ and $\hat{s}_i$ of $n=2000$ observations from two clusters \eqref{eq:sphere1} and \eqref{eq:sphere2}, represented by blue and red points, respectively. Middle and Right: clustering performance measured by ARI and NMI based on $10$ replications. The dimensions are $d_1=20$ and $d_2=10$. The penalty parameters for lrCC are $(\gamma_1,\gamma_2) = (8,8/15)$.}
    \label{fig:sphere1}
\end{figure}

\vspace{-0.4cm}
\begin{figure}[!ht]
    \centering
    \includegraphics[width=1\textwidth]{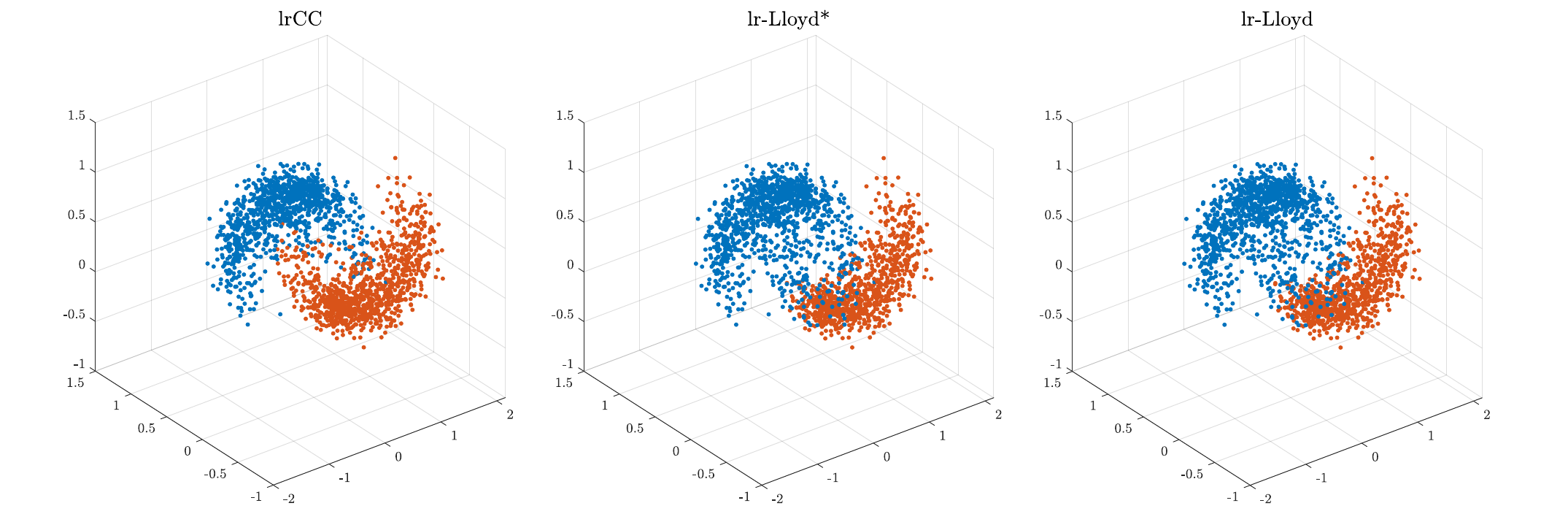}
    \vspace{-0.6cm}
    \caption{Clustering labels of three compared methods. The sample size is $n=2000$.}
    \label{fig:sphere_3methods}
\end{figure}

For the second cluster, we generate samples by reflecting and translating the singular values from the first cluster. Specifically, the samples in the second cluster are constructed as follows:
\begin{align}
\begin{array}{l}
\hat{s}_i=    [ 1 + \sin{\hat{\vartheta}_i}\cos{\hat{\theta}_i}; \ 0.5 - \sin{\hat{\vartheta}_i}\sin{\hat{\theta}_i}; \ 0.5 - \cos{\hat{\vartheta}_i} ] + 0.1 \hat{\varepsilon}_{i},
\\
  A_i = \widehat{U} {\rm Diag}( \hat{s}_i ) \widehat{V}^{\intercal},
\mbox{ where }
\hat{\theta}_i \sim {\rm Unif}[0,\pi], \,
\hat{\vartheta}_i \sim {\rm Unif}[0,\pi/2],  \,
\hat{\varepsilon}_{i} \sim N(0,I_3).
\end{array}\label{eq:sphere2}
\end{align}
Here, $\widehat{U}\in\mathbb{R}^{d_1\times 3}$ and $\widehat{V}\in\mathbb{R}^{d_2\times 3}$ are the top 3 left and right singular vectors of a random $d_1\times d_2$ matrix with entries drawn from $N(0,1)$. We plot the singular value vectors $\hat{s}_i$ in the second cluster \eqref{eq:sphere2} in red, alongside ${s}_i$ in the first cluster \eqref{eq:sphere1} in blue, in the left panel of Figure~\ref{fig:sphere1}.
The figure illustrates noisy observations from two intertwined quarter spheres.

We compare different methods by evaluating clustering performance across various sample sizes $n\in \{200 ,500,1000,2000,5000 \}$, with each cluster containing half of the samples. Clustering performance is evaluated using the standard ARI  and  NMI. The results for the three compared methods are shown in the middle and right panels of Figure~\ref{fig:sphere1}.
Figure~\ref{fig:sphere1} indicates that  our lrCC method successfully identifies the true labels, achieving both ARI and NMI values of one. In contrast,  lr-Lloyd* and  lr-Lloyd achieve ARI and NMI scores below $0.8$. The corresponding clustering labels are displayed in Figure~\ref{fig:sphere_3methods}. The left panel of Figure~\ref{fig:sphere_3methods} aligns with the true labels shown in the left panel of Figure~\ref{fig:sphere1},  confirming that the lrCC method accurately recovers the ground truth. In contrast, the middle and right panels of Figure~\ref{fig:sphere_3methods} show that the lr-Lloyd* and  lr-Lloyd methods misclassify some red points near the blue cluster, consistent with their lower ARI and NMI scores, both below 0.8.

\subsection{Unbalanced Gaussian}
In this section, we construct unbalanced Gaussian data based on the approach in \cite[Section~7.8]{rezaei2016set}.  Specifically, we extend their vector data in $\mathbb{R}^2$ to low rank matrices of rank $2$. For clarity, we give the details of the data generation process:  $K=8$, and the centroid $M_i$ of the $i$-th cluster, $i\in [8]$, is
\begin{align*}
    M_i = U_i
    \begin{pmatrix}
        C_{1i} & \\
         &  C_{2i}
    \end{pmatrix} V_i^{\intercal},\
    C = \begin{pmatrix}
    0.02  & 0.09  &  0.16  &  0.70  &  0.70  &  0.80  &  0.90  &  0.90 \\
    0.48   & 0.55  &  0.48  &  0.36  &  0.60  &  0.48  &  0.36  &  0.60
        \end{pmatrix},
\end{align*}
$U_i\in\mathbb{R}^{d_1\times 2}$ and $V_i\in\mathbb{R}^{d_2\times 2}$ are the top $2$ left and right singular vectors of a random $d_1\times d_2$ matrix with entries drawn from $N(0,1)$.
The observations $A_i$ are then generated with noise $E_i$, where each entry of $E_i$ is drawn independently from $N(0,1)$:
$A_i = M_{s_i^*} + 0.1E_i$, $i\in[n]$. Here $s_i^*\in [8]$ is the true cluster label of $A_i$.
\begin{figure}[!ht]
    \centering
   \includegraphics[width=0.8\textwidth]{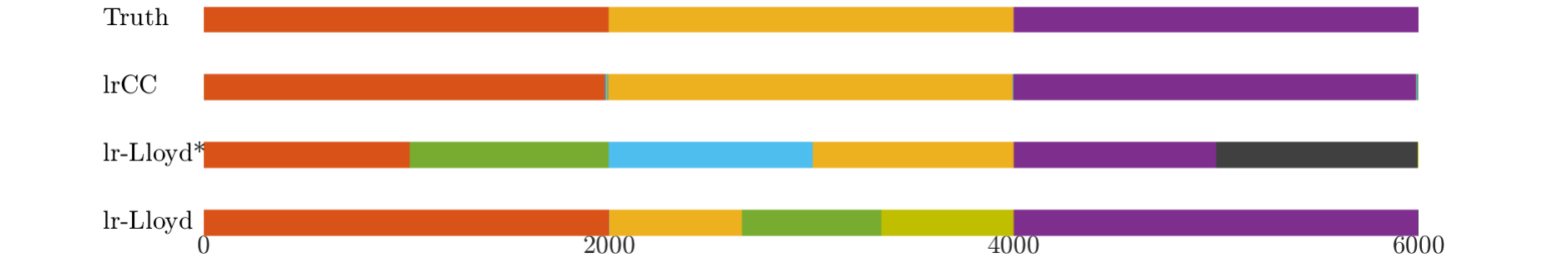}\\[4mm]
    \includegraphics[width=0.8\textwidth]{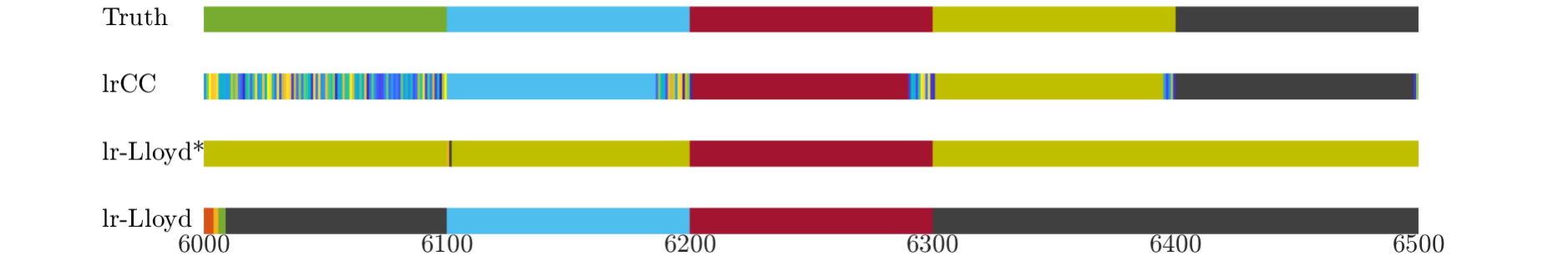}
    \caption{Top: three clusters with $2000$ points. Bottom: five clusters with $100$ points. Points of the same color are assigned into the same cluster. The penalty parameters for lrCC are  $(\gamma_1,\gamma_2) = (0.08,0.04)$.}
    \label{fig:unGauss}
\end{figure}

Following the setup in \cite[Section~7.8]{rezaei2016set}, we assign unbalanced cluster sizes: three clusters contain $2000$ points each, while the remaining five clusters contain $100$ points each.  The clustering results for the three methods, together with the ground truth, are visualized in Figure~\ref{fig:unGauss}. The top figure illustrates the partitioning of the three large clusters, each with $2000$ points, while the bottom figure shows the partitioning of the five small clusters, each containing $100$ points. And the colors in Figure~\ref{fig:unGauss} represent the assigned labels. As shown in Figure~\ref{fig:unGauss},  the lr-Lloyd and lr-Lloyd* methods tend to split the large clusters with $2000$ points into multiple parts. In contrast, our lrCC method can effectively identify these three large clusters but tends to split one of the small clusters with $100$ points into several parts. The ARI and NMI performances of three methods based on $5$ replications are presented in Table~\ref{tab:comparison}(a), where we set the dimensions $d_1=20$, $d_2=10$, and the sample size $n= 6500$. This shows that lrCC achieves a higher ARI score compared to the other two methods, while lr-Lloyd attains a slightly higher NMI score than lrCC.

\begin{table}[!ht]
  \centering
\caption{Comparison of three methods under criteria ARI and NMI.}\label{tab:comparison}
\begin{minipage}{0.42\textwidth}
    \centering
    \begin{tabular}{lcc}
\toprule
Method & ARI & NMI \\
\midrule
lrCC & 0.9849 & 0.9124  \\
lr-Lloyd* & 0.6941  & 0.8206  \\
lr-Lloyd & 0.9161  & 0.9415  \\
\bottomrule
\end{tabular}
    \\[1ex]  
\footnotesize{(a) Results on the unbalanced Gaussian data. The penalty parameters for lrCC are $(\gamma_1,\gamma_2) = (0.08,0.04)$.}
  \end{minipage} \hspace{1cm}
  \begin{minipage}{0.42\textwidth}
    \centering
    \begin{tabular}{lcc}
\toprule
Method & ARI & NMI \\
\midrule
lrCC & 0.8475 & 0.9007  \\
lr-Lloyd* & 0.6579  & 0.7406  \\
lr-Lloyd & 0.6267  & 0.7185  \\
\bottomrule
\end{tabular}
\\[1ex]
\footnotesize{(b) Results on the digits data. The penalty parameters for lrCC are $(\gamma_1,\gamma_2) = (10^{-0.5},10^{-1.5})$.}
  \end{minipage}
\end{table}


\subsection{Handwritten Digits}
In this section, we use the handwritten digit samples from the UCI Machine Learning Repository\footnote{\url{https://archive.ics.uci.edu/dataset/81/pen+based+recognition+of+handwritten+digits}}. The downloaded data\footnote{\url{https://scikit-learn.org/stable/auto_examples/datasets/plot_digits_last_image.html}} includes  $8\times 8$ pixel grayscale images of digits (0 through 9), captured from pen-based input devices. The dimensions are $d_1=8$, $d_2=8$, and the sample size is $n = 1797$, where each digit appears approximately $170$ to $180$ times.
To illustrate the appearance of the digits, we display one representative digit from each cluster in Figure~\ref{fig:3}.
\begin{figure}[!ht]
    \centering   \includegraphics[width=0.8\textwidth]{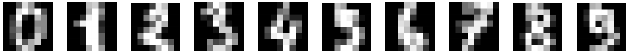}
    \caption{Illustration of one digit from each cluster in the handwritten digits data.}
    \label{fig:3}
\end{figure}


We apply the three methods lrCC, lr-Lloyd*, and lr-Lloyd to identify $K=10$ clusters in the handwritten digits data. The clustering performance, measured by ARI and NMI, is presented in Table~\ref{tab:comparison}(b). Table~\ref{tab:comparison}(b) shows that our lrCC achieves higher scores of ARI and NMI than both lr-Lloyd* and lr-Lloyd, indicating that lrCC outperforms the other two methods. To understand the clustering behaviors of the three methods, we embed each digit image into a two-dimensional space and visualize the results in Figure~\ref{fig:digits}, with each color representing a different digit. For this low-dimensional embedding, we use t-distributed stochastic neighbor embedding (t-SNE) \cite{hinton2002stochastic,van2008visualizing}, a popular nonlinear dimensionality reduction technique that is particularly effective at revealing the underlying structure of complex data sets and provides enhanced partitioning of the digits.
\begin{figure}[!ht]
    \centering
     \includegraphics[width=0.095\textwidth]{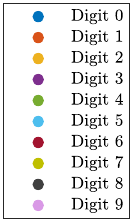}
     \includegraphics[width=0.88\textwidth]{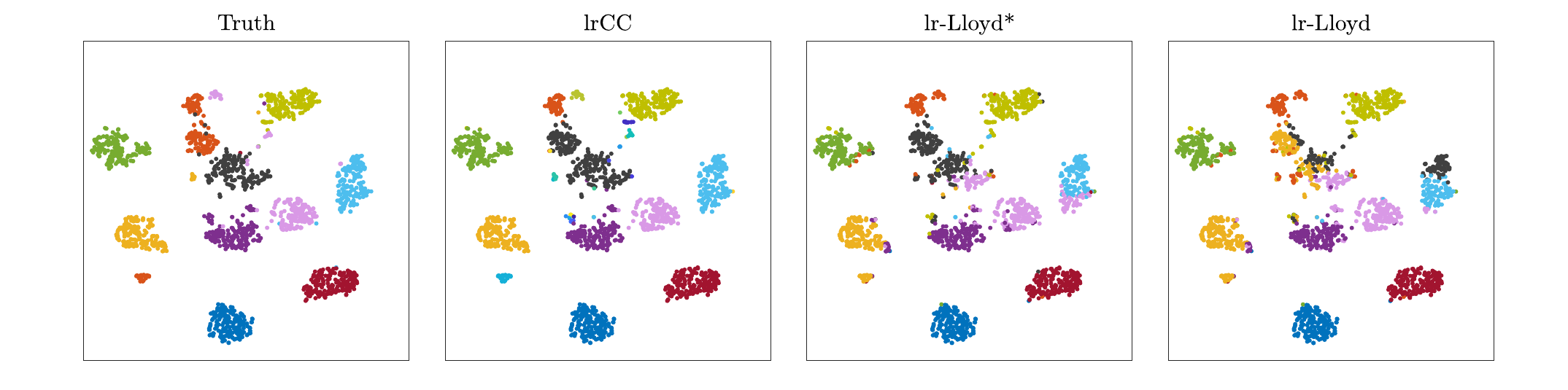}
   \caption{Visualization of the handwritten digits data using t-SNE, with each color representing a different digit.}
    \label{fig:digits}
\end{figure}

From the ``Truth'' and ``lrCC'' panels in Figure~\ref{fig:digits}, we observe that lrCC correctly identifies most digits. However, some instances of Digit $1$ (orange-red points in the ``Truth'' panel) are misclassified as Digit $8$ (dark-gray points in the ``lrCC'' panel). This challenge in distinguishing between Digits $1$ and $8$ is also evident in the results of the lr-Lloyd* and lr-Lloyd methods. The lr-Lloyd* method, similar to lrCC, misclassifies some instances of Digit $1$  as Digit $8$, and it additionally misclassifies several instances of Digit $8$  (dark-gray in the ``Truth'' panel) as Digit $9$ (light-purple in the ``lr-Lloyd*'' panel). The lr-Lloyd method shows even poorer performance with Digits $1$ and $8$, misclassifying some instances of the two digits as Digit $2$, and further splitting the Digit $8$ cluster into several groups mixed with Digits $1$, $2$, $7$, and $9$. Furthermore, for Digit $5$ (light-blue points in the ``Truth'' panel), lr-Lloyd* misclassifies some instances as Digit $9$, while lr-Lloyd incorrectly labels some as Digits $8$ and $9$. Overall,  lrCC has a substantially  better clustering result compared to lr-Lloyd* and lr-Lloyd as shown in Figure~\ref{fig:digits}, which aligns with the metrics in Table~\ref{tab:comparison}(b).

\subsection{Fashion MNIST}
In this section, we use the fashion MNIST data, which serves as a more challenging alternative to the original MNIST dataset. Fashion MNIST  consists of $70000$ grayscale images  categorized into $10$ different classes of clothing and accessories, including items like shirts, shoes, sneakers, and bags. Each image has dimensions of $d_1=28$ by $d_2=28$ pixels. The dataset is publicly available, and
more information and access to the data can be found on GitHub\footnote{\url{https://github.com/zalandoresearch/fashion-mnist?tab=readme-ov-file##get-the-data}}.

\begin{figure}[!ht]
    \centering
     \includegraphics[width=0.95\textwidth]{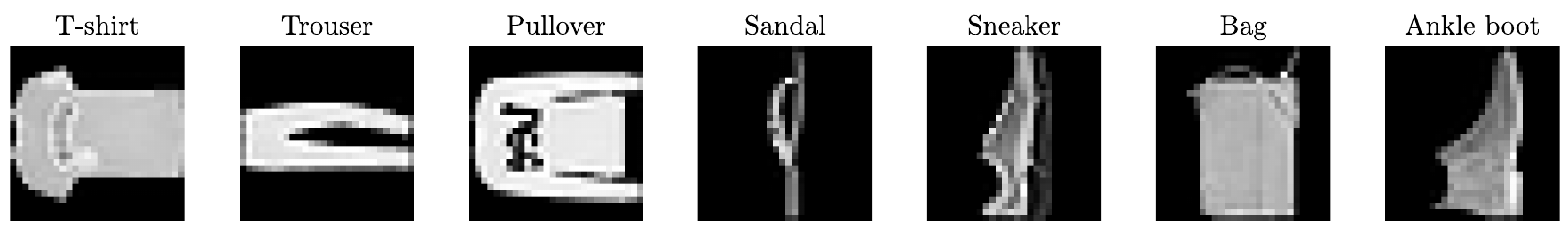}
     \vspace{-0.2cm}
    \caption{Illustration of one image from each cluster in the fashion MNIST data.}
    \label{fig:fmnist}
\end{figure}

\vspace{-0.2cm}
\begin{figure}[!ht]
    \centering
\includegraphics[width=1\textwidth]{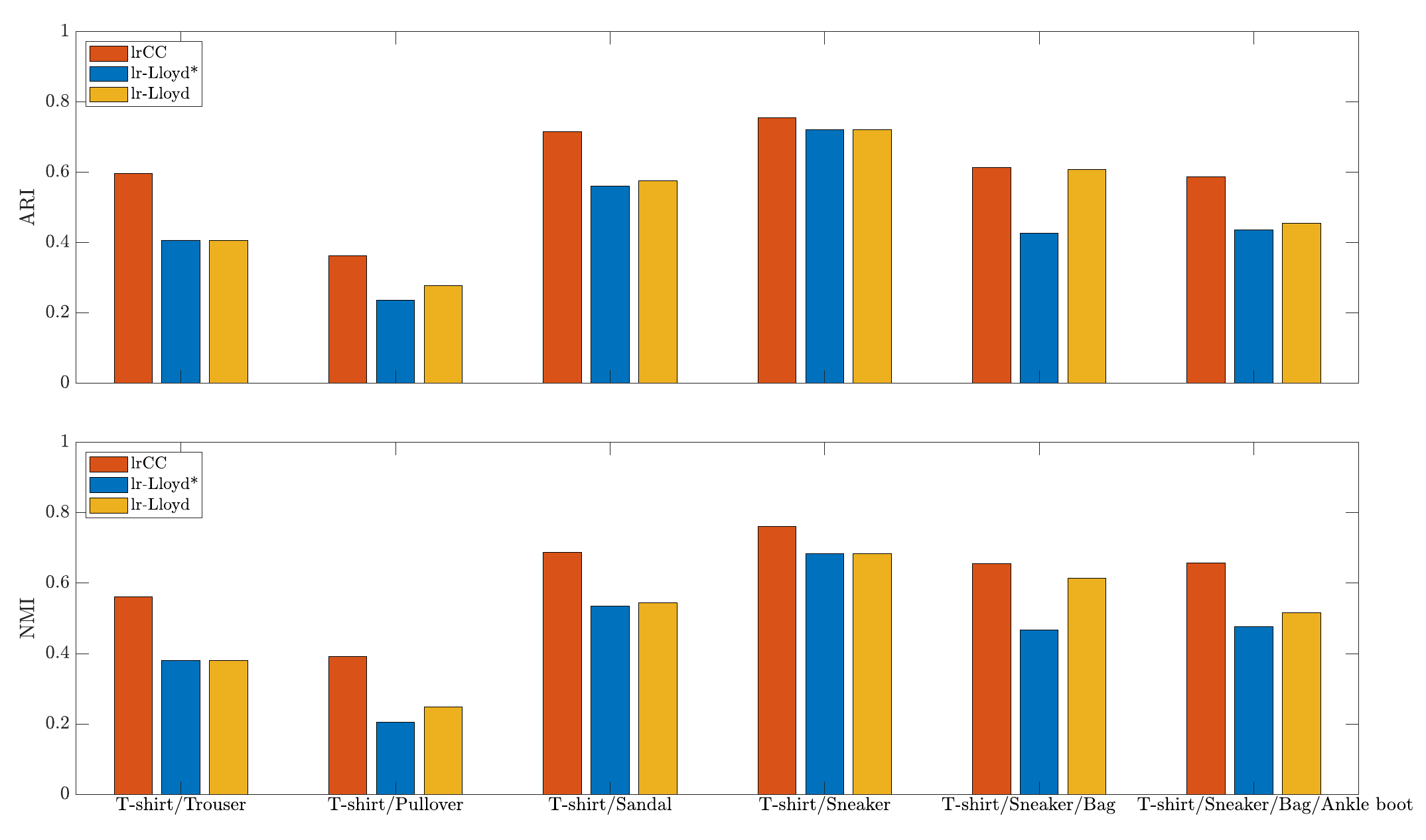}
\vspace{-0.7cm}
    \caption{Clustering performance on the fashion MNIST data. The penalty parameters for lrCC are $(\gamma_1,\gamma_2) = (1,0.1)$. The $x$-axis shows the cluster pairs or groups being distinguished, corresponding to different tasks.}
    \label{fig:fmnist_ARI}
\end{figure}

In our experiment, we randomly select $100$ samples from each of the $10$ clusters.  Given the difficulty of accurately classifying all $10$ categories simultaneously, we focus on specific binary, three-class, and four-class classification tasks. In Figure~\ref{fig:fmnist}, we provide a sample image from the clusters used in these classification tasks.
To evaluate the clustering performance of the three methods, we calculate ARI and NMI scores across a series of tasks, which include: two-cluster  classification between ``T-shirt'' and ``Trouser'' (as well as ``T-shirt'' vs. ``Pullover'', ``T-shirt'' vs. ``Sandal'', ``T-shirt'' vs. ``Sneaker''), three-cluster  classification among ``T-shirt'', ``Sneaker'', and ``Bag'', and four-cluster  classification among ``T-shirt'', ``Sneaker'', ``Bag'', and ``Ankle boot''. These tasks reflect a progressively challenging classification setting as additional clusters are included. The ARI and NMI scores for each method across these tasks are displayed in Figure~\ref{fig:fmnist}. It shows that the lrCC method consistently achieves higher ARI and NMI scores than lr-Lloyd* and lr-Lloyd, indicating its superior clustering accuracy across various tasks.

\section{Conclusion}\label{sec:conclu}
In this work, we introduced a novel low rank convex clustering model designed for matrix-valued observations, extending existing convex clustering methods for vector-valued data. We established rigorous cluster recovery results and conducted a comprehensive analysis of prediction error, further enhancing the model's practical applicability. To efficiently solve the proposed model, we developed a double-loop algorithm. Extensive numerical experiments were performed to validate the model, demonstrating its effectiveness in clustering matrix-valued data. Our results suggest that the proposed model is a valuable tool for a wide range of applications in data science and machine learning. Future work could focus on extending the model to handle more complex data structures or exploring its scalability for large-scale data sets, including the application of the adaptive sieving strategy \cite{yuan2023adaptive,yuan2022dimension} to reduce the dimensionality and accelerate computational speed.


\appendix

\section{Auxiliary Results}
\subsection{Pairwise Distance of Gaussian Samples}
\label{sec: dis_G}

Suppose that $n$ samples $\{A_i\}_{i=1}^n$  are drawn from \eqref{model-lrmm}. We form a matrix of samples $A\in \mathbb{R}^{d\times n}$, whose $i$-th column is the vectorization of the sample $A_i$. We also form the corresponding matrix of means $C\in \mathbb{R}^{d\times n}$, whose $i$-th column is the vectorization of the sample mean $M_{s_i^*}$.

From \cite{dasgupta2007spectral} and \cite[Lemma 6.3]{kumar2010clustering}, we know that
\begin{align*}
\|A-C\|_2 = O\left(\sigma\sqrt{n}\ {\rm polylog}\left( \frac{d}{\pi_{\min}}
\right)\right),
\end{align*}
with high probability. Since the above norm $\|\cdot\|_2$ is the spectral norm and the matrix $C$ is of rank $K$, we further have that
\begin{align*}
\|A-C\|_F = O\left(\sigma\sqrt{nK}\ {\rm polylog}\left( \frac{d}{\pi_{\min}}
\right)\right),
\end{align*}
with high probability, which further indicates that
\begin{align*}
\max_{  i\in [n]} \|A_i-M_{s_i^*}\|_F = O\left(\sigma\sqrt{K} {\rm polylog}\left(\frac{d}{\pi_{\min}}\right)\right).
\end{align*}
Therefore,
by denoting $\mathcal{I}_{\alpha} = \{i \in [n] \mid s_i^*= \alpha\}, \ \alpha \in [K]$,
we have
\begin{align*}
& \max_{  i,j\in {\cal I}_{\alpha}} \|A_i-A_j\|_F
=\max_{  i,j\in {\cal I}_{\alpha}} \|A_i- M_{\alpha} + M_{\alpha} - A_j\|_F \\
\leq & \max_{  i,j\in {\cal I}_{\alpha}} \left( \|A_i- M_{s_i^*}\|_F + \|A_j- M_{s_j^*}\|_F \right) = O\left(\sigma\sqrt{K} {\rm polylog}\left(\frac{d}{\pi_{\min}}\right)\right),
\end{align*}
with high probability.

\subsection{Lower Bound of $\sigma_{\min}(B)$ for $k$-nearest Neighbor Graph}\label{sec:prooflemma6}
Consider a $k$-nearest neighbor graph ${\mathcal{G}}=([n],\mathcal{E})$, where each node in $[n]$ is connected to its $k$-nearest neighbors based on the Euclidean distances, and $k$ is a given number. Then the following lemma provides a lower bound for the smallest nonzero singular value
of its incidence matrix.
\begin{lemma}\label{lemma:lb_knear}
Given $k\geq 2$. For the $k$-nearest neighbor graph ${\mathcal{G}}=([n],\mathcal{E})$, we have
\begin{align*}
\sigma_{\min}(B) \geq \frac{2}{n}\sqrt{\frac{k+1}{3}},
\end{align*}
where $B\in \mathbb{R}^{n\times |{\cal E}|}$ represents the incidence matrix for the graph ${\mathcal{G}}$.
\end{lemma}
\begin{proof}
We first consider the case when ${\cal G}$ is a connected $k$-nearest neighbor graph. According to \cite{mohar1991laplacian}, we have that
\begin{align*}
\sigma_{\min}(B)\geq \frac{2}{\sqrt{{\rm diag}({\cal G}) n}}.
\end{align*}
In addition, from \cite[Theorem 1]{erdHos1989radius}, it holds that
\begin{align*}
{\rm diag}({\cal G}) \leq \frac{3n}{\delta+1},
\end{align*}
where $\delta$ is the minimum degree of the graph ${\cal G}$, which is obviously lower bounded by $k$. Therefore, we have that
\begin{align*}
\sigma_{\min}(B)\geq \frac{2}{\sqrt{{\rm diag}({\cal G}) n}}\geq \frac{2}{n}\sqrt{\frac{\delta+1}{3}} \geq \frac{2}{n}\sqrt{\frac{k+1}{3}} .
\end{align*}

Next, we proceed to examine the case when ${\cal G}$ has $\kappa_0>1$ connected components. Let {$\mathcal{G}_1,\dots,\mathcal{G}_{\kappa_0}$} be the connected components in ${\mathcal{G}}$, whose incidence matrices are $B_1,\dots,B_{\kappa_0}$, respectively. Denote the number of vertices in the graph ${\cal G}_i$ as $n_i$ for $i=1,\cdots,\kappa_0$. Then we have that
\begin{align*}
\sigma_{\min}(B) = \min_{i\in [\kappa_0]} \{ \sigma_{\min}(B_i) \}.
\end{align*}
For each $i=1,\cdots,\kappa_0$, since ${\cal G}_i$ is a connected $k$-nearest neighbor graph with $n_i$ nodes. Then we have that
\begin{align*}
\sigma_{\min}(B_i) \geq \frac{2}{n_i}\sqrt{\frac{k+1}{3}} \geq \frac{2}{n}\sqrt{\frac{k+1}{3}},
\end{align*}
which further implies that
\begin{align*}
\sigma_{\min }(B)\geq \frac{2}{n}\sqrt{\frac{k+1}{3}}.
\end{align*}
This completes the proof.
\end{proof}

\subsection{Lemmas}
\label{sec:lemma_subgaussian}
In order to prove Theorem \ref{thm:main_con}, we need to present three lemmas. Lemma~\ref{lemma: subgaussian}(i) provides the well known Hanson-Wright inequality \cite{hanson1971bound} regarding tail probabilities for the quadratic form of independent centered sub-Gaussian random variables; see also \cite[Theorem~1.1]{zajkowski2020bounds}. Lemma~\ref{lemma: subgaussian}(ii)  is a direct consequence of sub-Gaussian tail bounds, since the  linear combination $b^{\intercal} z$ is sub-Gaussian with mean zero and variance $\sigma^2\|b\|^2$ if all $z_i$'s are independent sub-Gaussian random variables with mean zero and variance $\sigma^2$.

\begin{lemma}\label{lemma: subgaussian}
Let $z\in \mathbb{R}^p$ be a random vector of independent sub-Gaussian random variables with mean zero and variance $\sigma^2$. Then we have the following conclusions.
\begin{itemize}[left=5pt, labelsep=3pt, itemsep=2pt]
\item[{\rm (i)}] Let $H\in \mathbb{R}^{p\times p}$ be a symmetric deterministic matrix. Then, there exist constants $c_{H,1}>0$ and $c_{H,2}>0$ such that for any $t>0$,
$\displaystyle
\mathbb{P}(z^{\intercal} H z\geq t+\sigma^2 {\rm tr}(H)) \leq \exp \left\{ -\min \left( \frac{c_{H,1} t^2 }{\sigma^4 \|H\|_F^2},\frac{c_{H,2} t}{ \sigma^2 \|H\|_2}\right)\right\}.
$
\item[{\rm (ii)}] Let $b\in \mathbb{R}^p$ be a deterministic vector. Then, for any $t>0$,
$ \mathbb{P}( |b^{\intercal} z|\geq t ) \leq 2 \exp \left( -\frac{t^2}{2\sigma^2\|b\|^2}
\right) $.
\end{itemize}
\end{lemma}

The following lemma characterizes the rank and minimum nonzero singular value of $\mathcal{D}=B^{\intercal}\otimes I_d$.

\begin{lemma} \label{lemma: D}
For the matrix $\mathcal{D}$, we have
$ {\rm rank}(\mathcal{D}) = d(n - \kappa_0) $, where $\kappa_0$  is the number of connected components in ${\mathcal{G}}=([n],\mathcal{E})$, and $\sigma_{\min}(\mathcal{D}) = \sigma_{\min}(B)$.
\end{lemma}
\begin{proof}
(i) According to \cite[Theorem~2.3]{bapat2010graphs}, we can see that ${\rm rank}(B)= n-k$. Thus, we have that ${\rm rank}(\mathcal{D}) = {\rm rank}(B) \times {\rm rank}(I_d) = d(n-k)$.

(ii) Suppose that $B$ has $r$ nonzero singular values: $\sigma_1,\cdots, \sigma_r$. Then by \cite[Chapter~5]{merris1997multilinear}, we know that $\mathcal{D} = B^{\intercal}\otimes I_d $ has $d r$ nonzero singular values:
\begin{align*}
\underbrace{\sigma_{1},\cdots,\sigma_{1}}_{d}, \cdots, \underbrace{\sigma_{r},\cdots,\sigma_{r}}_{d}.
\end{align*}
Therefore, we can see that $\sigma_{\min}(\mathcal{D}) = \sigma_{\min}(B)$. This completes the proof.
\end{proof}

Lemma \ref{lemma: D} provides the rank of $\mathcal{D}$, and therefore we can have its singular value decomposition $\mathcal{D} = U S V_1^{\intercal}$, where $S\in \mathbb{R}^{d(n-\kappa_0)\times d(n-\kappa_0)}$ is a diagonal matrix, $U\in \mathbb{R}^{d|\mathcal{E}|\times d(n-\kappa_0)}$ and $V_1\in \mathbb{R}^{dn\times d(n-\kappa_0)}$ satisfy $U^{\intercal}U = I_{d(n-\kappa_0)}$, $V_1^{\intercal}V_1 = I_{d(n-\kappa_0)}$. Then there exists $V_2\in \mathbb{R}^{dn\times d\kappa_0}$ such that $V=[V_1,V_2]\in \mathbb{R}^{dn\times dn}$ is an orthogonal matrix. In the following lemma, we establish the bound for the quadratic form consisting of the matrix $V_2 V_2^{\intercal}$ and independent sub-Gaussian random variables.
\begin{lemma} \label{lemma: D2}
Let $\epsilon\in \mathbb{R}^{dn}$ be a random vector of independent sub-Gaussian random variables with mean zero and variance $\sigma^2$. Then there exist constants $c_1>0$ and $c_2>0$ such that the following inequality holds:
\begin{align*}
   \mathbb{P}\left(\! \frac{1}{dn}\epsilon^{\intercal} V_2 V_2^{\intercal}\epsilon \!\geq \!\sigma^2 \left[ \frac{\kappa_0}{n} \!\!+\!\! \sqrt{\frac{\kappa_0\log(dn)}{dn^2}}\right]\right)
     \!\leq \! \exp\!\left\{\! -\!\min \left( c_1 \!\log(dn),c_2\! \sqrt{ d\kappa_0\log(dn)}
    \right)\!\right\}.
\end{align*}
\end{lemma}
\begin{proof}
It can be easily seen that
$$
{\rm tr}(V_2 V_2^{\intercal}) = dk, \quad
\|V_2 V_2^{\intercal}\|_F = \sqrt{{\rm tr}(V_2 V_2^{\intercal}V_2 V_2^{\intercal}) } = \sqrt{{\rm tr}(V_2 V_2^{\intercal}) } = \sqrt{{\rm tr}(V_2^{\intercal} V_2) } = \sqrt{dk}.
$$
Moreover, since $V_2 V_2 ^{\intercal}V_2 V_2 ^{\intercal} = V_2 V_2 ^{\intercal}$, we have that $V_2 V_2 ^{\intercal}$ is a projection matrix with $\|V_2 V_2 ^{\intercal}\|_2 = 1$. According to Lemma \ref{lemma: subgaussian}, there exist constants $c_1>0$ and $c_2>0$ such that for any $t>0$,
\begin{align*}
\mathbb{P}\left(\frac{1}{dn}\epsilon^{\intercal} V_2 V_2 ^{\intercal} \epsilon \geq \frac{t+\sigma^2 dk}{dn}\right) \leq \exp \left\{ -\min \left( \frac{c_{1} t^2 }{\sigma^4 dk},\frac{c_{2} t}{ \sigma^2}\right)\right\}.
\end{align*}
By taking $t = \sigma^2 \sqrt{dk\log(dn)}$, we have
\begin{align*}
   \mathbb{P}\left(\! \frac{1}{dn}\epsilon^{\intercal} V_2 V_2^{\intercal}\epsilon \!\geq \!\sigma^2 \left[ \frac{\kappa_0}{n} \!\!+\!\! \sqrt{\frac{\kappa_0\log(dn)}{dn^2}}\right]\right)
     \!\leq \! \exp\!\left\{\! -\!\min \left( c_1 \!\log(dn),c_2\! \sqrt{ d\kappa_0\log(dn)}
    \right)\!\right\}.
\end{align*}
This completes the proof.
\end{proof}

\section{Proofs of Propositions}
\subsection{Proof of Proposition \ref{prop: connection_two_problems}}\label{sec:proof_of_prop_merging}
Since $(X^{(1),*},\dots,X^{(G),*})$ is the optimal solution to \eqref{eq: cluster_opt}, there must exist $Z^{(\alpha),*}\in \partial \|X^{(\alpha),*}\|_*$ and $U^{(\alpha,\beta),*}\in \partial \|X^{(\alpha),*}-X^{(\beta),*}\|_F$ for all $\alpha,\beta\in [G]$, satisfying $U^{(\alpha,\beta),*} = -U^{(\beta,\alpha),*}$ when $\alpha \neq \beta$, such that
\begin{align}
X^{(\alpha),*} - A^{\Gamma,\alpha} + \frac{\gamma_1}{|\mathcal{J}_{\alpha}|}\sum_{\beta\in[G]\backslash \{\alpha\} } w^{(\alpha,\beta)} U^{(\alpha,\beta),*} + \gamma_2 Z^{(\alpha),*}=0,\ \forall \ \alpha \in [G]. \label{eq: opt_cluster}
\end{align}
For each $i\in [n]$ and $\alpha\in[G]$, we define
\begin{align*}
T_{i}^{(\alpha)} &= \sum_{\beta\in [G]\backslash \{\alpha\}} \left( \sum_{j\in \mathcal{J}_{\beta}}w_{ij}  -\frac{w^{(\alpha,\beta)}}{|\mathcal{J}_{\alpha}|}
\right)U^{(\alpha,\beta),*},
\end{align*}
and for $l(i,j)\in \mathcal{E}$, we define
\begin{align*}
U_{ij}^{*} \!=\! \left\{\begin{aligned}
&U^{(\alpha,\beta),*}, && \mbox{if } i\in \mathcal{J}_{\alpha}, j\in \mathcal{J}_{\beta},  \mbox{ with } \alpha,\beta\!\in\! [G], \alpha\!\neq\! \beta,\\
&\frac{1}{|\mathcal{J}_{\alpha}| w_{ij} }  \left[ \frac{A_i\!-\!A_j }{\gamma_1} - (T_{i}^{(\alpha)}\!-\! T_{j}^{(\alpha)})\right]
, && \mbox{if } i,j\in \mathcal{J}_{\alpha} \mbox{ with } i\neq j, \alpha\in [G].
\end{aligned}\right.
\end{align*}
It can be seen that for any $l(i,j)\in \mathcal{E}$, $U_{ij}^{*} = -U_{ji}^{*}$. Moreover, for $i,j\in \mathcal{J}_{\alpha}$ with $i\neq j$, $\alpha \in [G]$, it holds that
\begin{align*}
\|U_{ij}^{*}\|_F
&\leq
\frac{1}{|\mathcal{J}_{\alpha}|w_{ij}\gamma_1 } \|A_i-A_j\|_F + \frac{1}{|\mathcal{J}_{\alpha}|w_{ij} }   \sum_{\beta\in [G]\backslash \{\alpha\}} \Big|\sum_{m\in \mathcal{J}_{\beta}}(w_{im} -w_{jm})
\Big| \| U^{(\alpha,\beta),*} \|_F\\
&\leq \frac{1}{|\mathcal{J}_{\alpha}| w_{ij}\gamma_1 } \|A_i-A_j\|_F + \frac{\eta_{ij}^{\Gamma,\alpha}}{|\mathcal{J}_{\alpha}| }
\leq 1,
\end{align*}
where the second inequality holds as $U^{(\alpha,\beta),*}\in \partial \|X^{(\alpha),*}-X^{(\beta),*}\|_F$, implying its Frobenius norm is at most $1$, and the last inequality follows from \eqref{eq: lb_gamma1_J}. This implies that $U_{ij}^{*}\in \partial \|X_i^*-X_j^*\|_F $ for all $l(i,j)\in \mathcal{E}$. Furthermore, by setting $Z_i^* = Z^{(\alpha),*}$ for all $i\in \mathcal{J}_{\alpha}$, $\alpha\in [G]$, we see that $Z_i^*\in \partial \|X_i^*\|_*$ for all $i\in[n]$.

Finally, we verify that the optimality condition of \eqref{model-convex} holds at $(X_1^*,\dots,X_n^*)$, which means $(X_1^*,\dots,X_n^*)$ is optimal to \eqref{model-convex}. Note that we set by default that $w_{ij}=0$ when $l(i,j)\notin \mathcal{E}$. Given $\alpha\in [G]$, for any $i\in \mathcal{J}_{\alpha}$, we have that
\begin{align*}
&\qquad X_i^* - A_i + \gamma_1 \sum_{l(i,j)\in \mathcal{E}} w_{ij}U_{ij}^* +\gamma_2 Z_i^* \\
&=X^{(\alpha),*}- A_i +\gamma_1 \sum_{\beta \in [G]\backslash \{\alpha \}} \sum_{j\in \mathcal{J}_{\beta}}w_{ij}U^{(\alpha,\beta),*} \\
&\quad+\frac{1}{|\mathcal{J}_{\alpha}|} \sum_{j\in \mathcal{J}_{\alpha}} \left[ A_i-A_j - \gamma_1(T_{i}^{(\alpha)}-T_{j}^{(\alpha)})
\right] +\gamma_2 Z^{(\alpha),*}\\
&=X^{(\alpha),*} +\gamma_1 \sum_{\beta \in [G]\backslash \{\alpha \}} \sum_{j\in \mathcal{J}_{\beta}}w_{ij} U^{(\alpha,\beta),*} -A^{\Gamma,\alpha}-\gamma_1 T_i^{(\alpha)} + \frac{\gamma_1}{|\mathcal{J}_{\alpha}|}\sum_{j\in \mathcal{J}_{\alpha}} T_j^{(\alpha)}+\gamma_2 Z^{(\alpha),*}\\
&= \gamma_1 \sum_{\beta \in [G]\backslash \{\alpha \}} \left( \sum_{j\in \mathcal{J}_{\beta}}w_{ij} -\frac{w^{(\alpha,\beta)}}{|\mathcal{J}_{\alpha}|} \right) U^{(\alpha,\beta),*}-\gamma_1 T_i^{(\alpha)} + \frac{\gamma_1}{|\mathcal{J}_{\alpha}|}\sum_{j\in \mathcal{J}_{\alpha}} T_j^{(\alpha)}\\
&= \frac{\gamma_1}{|\mathcal{J}_{\alpha}|}\sum_{j\in \mathcal{J}_{\alpha}} \sum_{\beta \in [G]\backslash \{\alpha \}} \left( \sum_{m\in \mathcal{J}_{\beta}}w_{jm}
-\frac{w^{(\alpha,\beta)}}{|\mathcal{J}_{\alpha}|} \right) U^{(\alpha,\beta),*} \\
&=\frac{\gamma_1}{|\mathcal{J}_{\alpha}|} \sum_{\beta \in [G]\backslash \{\alpha \}} \left(   \sum_{\substack{j\in \mathcal{J}_{\alpha},\\ m\in \mathcal{J}_{\beta}}}w_{jm} -w^{(\alpha,\beta)} \right) U^{(\alpha,\beta),*}=0,
\end{align*}
where the third equality follows from the optimality condition \eqref{eq: opt_cluster}, the fourth equality follows from the definition of $T_i^{(\alpha)}$, and the last equality follows from the definition of $w^{(\alpha,\beta)}$. This completes the proof.

\subsection{Proof of Proposition \ref{prop: neq}}\label{sec:proof_of_distinguish}
As $(X^{(1),*},\dots,X^{(G),*})$ is optimal to problem \eqref{eq: cluster_opt}, there exist $U^{(\alpha,\beta),*}\in \partial \|X^{(\alpha),*}-X^{(\beta),*}\|_F$ for all $\alpha,\beta\in [G]$, satisfying $U^{(\alpha,\beta),*} = -U^{(\beta,\alpha),*}$ when $\alpha \neq \beta$, such that
$$
X^{(\alpha),*} = {\rm Prox}_{\gamma_2 \|\cdot\|_*}\left( A^{\Gamma,\alpha}- B^{\Gamma,\alpha}\right), \mbox{ where }  B^{\Gamma,\alpha} :=\frac{\gamma_1}{|\mathcal{J}_{\alpha}|}\sum_{m\in[G]\backslash \{\alpha\} } w^{(\alpha,m)} U^{(\alpha,m),*}.$$
We prove by contradiction. Suppose there exist $\alpha,\beta\in [G]$ such that $A^{\Gamma,\alpha}\neq A^{\Gamma,\beta}$ and $X^{(\alpha),*}= X^{(\beta),*}$. Note that
$$
\|X - Y\|_F \leq \gamma_2 \sqrt{d_2 - r},
$$
if $ {\rm Prox}_{\gamma_2 \|\cdot\|_*}(X) = {\rm Prox}_{\gamma_2 \|\cdot\|_*}(Y) $ and the rank is $r$, see Proposition~\ref{prop: nuc2} for the proof. Then
we have
$$
\| (A^{\Gamma,\alpha} - B^{\Gamma,\alpha}) - (A^{\Gamma,\beta} - B^{\Gamma,\beta}) \|_F \leq \gamma_2 \sqrt{d_2 - r}$$ with $r={\rm rank}(X^{(\alpha),*})$. This implies
$$
\|A^{\Gamma,\alpha} - A^{\Gamma,\beta}\|_F \leq \|B^{\Gamma,\alpha} \|_F + \|B^{\Gamma,\beta} \|_F + \gamma_2 \sqrt{d_2 - r} \leq  \gamma_1 \left(w^{(\alpha)} + w^{(\beta)} \right) + \gamma_2 \sqrt{d_2-r},
$$
which contradicts the inequality. The proof is completed.

\section{Extension of Prediction Error Bounds to $M$-concentrated Noise}
\label{sec: Mnoise}
In this section, we are going to extend the analysis on the prediction error of the lrCC model \eqref{model-convex} to the case when $a = [{\rm vec}(A_1);\cdots; {\rm vec}(A_n)]\in \mathbb{R}^{dn}$ satisfies $a=x_0+\epsilon$, with $\epsilon$ being $M$-concentrated, wherein the dependence is allowed. Below is the definition of the $M$-concentrated random vector \cite{vu2015random}.

\begin{definition}\label{def: M_concentrated}
A random vector $z \in \mathbb{R}^p$ is said to be $M$-concentrated (where $M > 0$ is a scalar that may depend on $p$) if there exist constants $c>0$ and $c' > 0$ such that for any convex, $1$-Lipschitz function $\psi: \mathbb{R}^p \rightarrow \mathbb{R}$ and any $t > 0$,
\begin{align*}
    \mathbb{P}\big( \left| \psi(z) - \mathbb{E}[\psi(z)] \right| \geq t \big) \leq c \exp\left(-c'\frac{t^2}{M^2}\right).
\end{align*}
\end{definition}
Note that the concept of an $M$-concentrated random variable is more general than that of a sub-Gaussian random variable, which was used in Section \ref{sec:bound_pre_err}. This generality arises because $M$-concentrated random variables can accommodate dependencies among their coordinates. The following are examples of
$M$-concentrated random variables as provided in \cite{vu2015random}:

\begin{itemize}[left=5pt, labelsep=3pt, itemsep=2pt]
    \item If $z$ is a vector of independent sub-Gaussian random variables with mean zero and variance $\sigma^2$, then $z$ is $\sigma$-concentrated, as shown in \cite[Theorem 5.3]{ledoux2001concentration} and \cite{bizeul2023log}.

    \item If $z_i$ are independent, and each $z_i$ is $M$-bounded for all $i$, then $z$ is $M$-concentrated, as shown  in \cite[Chapter 4]{ledoux2001concentration} and \cite[Theorem F.5]{tao2010random}.

    \item If the coordinates $z_i$ come from a random walk satisfying certain mixing properties, then $z$ is $M$-concentrated, as shown  in \cite[Corollary 4]{samson2000concentration}. Here $z_i$'s are not necessarily independent.
\end{itemize}

We provide a variant of Theorem \ref{thm:main_con} to analyze the finite sample bound for the prediction error of the lrCC model when applied to data with $M$-concentrated noises.

\begin{theorem}\label{thm: Mconcentrated}
Suppose that $a=x_0+\epsilon$, where $\epsilon\in \mathbb{R}^{dn}$ is an $M$-concentrated random vector with mean zero. Let $\hat{x}$ be the solution of problem \eqref{eq: new_vec_form} with the assumption that $ \min_{l(i,j)\in \mathcal{E}}w_{ij}\geq \frac{1}{2}$. If $\gamma_1 \geq\frac{2M}{c_0\sigma_{\min}(B)}\sqrt{2d\log(d|\mathcal{E}|)}$
with some constant $c_0>0$, then
\begin{align*}
\begin{array}{l}
\displaystyle
\frac{1}{2dn} \|\hat{x}-x_0\|^2 \leq  \frac{\kappa_0}{n} + \frac{\sqrt{d\kappa_0} \log(dn)}{dn} \\
\displaystyle+ \frac{\gamma_1}{2dn}\sum_{l(i,j)\in \mathcal{E}}
(1+2 w_{ij}) \left\| D^{l(i,j)} x_0\right\| + \gamma_2 \left[ \frac{M}{n^{1/4}} + \frac{1}{dn}  \sum_{i=1}^n \|\mathcal{M}^i x_0 \|_* \right]
\end{array}
\end{align*}
holds with probability at least
\begin{align*}
1 - \frac{c}{d|\mathcal{E}|} -c_1 \exp \left\{ \log \log(dn) -  \frac{c_2}{M^2} \log(dn)\right\}-c_3 \exp\left(-c_4 \sqrt{d^3n} \right),
\end{align*}
where $c$, $c_1$, $c_2$, $c_3$ and $c_4$ are positive constants. Here $\kappa_0$ denotes the number of connected components and $B\in \mathbb{R}^{n\times |{\cal E}|}$ represents the incidence matrix for the graph ${\mathcal{G}}=([n],\mathcal{E})$.
\end{theorem}

In order to prove the above theorem, we need to characterize some properties of the $M$-concentrated random variables. Specifically, Lemma \ref{lemma: Mb}(i) describes the tail probabilities for the quadratic form of $M$-concentrated random variables, which is a generalization of the Hanson-Wright inequality for sub-Gaussian random variables in Lemma \ref{lemma: subgaussian}(i). Lemma \ref{lemma: Mb}(ii) states the concentration property of the linear combination of the coordinates of an $M$-concentrated random vector.

\begin{lemma}\label{lemma: Mb}
Let $z\in \mathbb{R}^p$ be an $M$-concentrated random vector. Then the following results hold.
\begin{itemize}[left=5pt, labelsep=3pt, itemsep=2pt]
\item[{\rm (i)}] There exist constants $c>0$ and $c'>0$ such that for any matrix $H\in \mathbb{R}^{p\times p}$ and any $t>0$,
\begin{align*}
\mathbb{P}\left( z^{\intercal} H z\geq t+ {\rm tr}(H) \right)\leq c \log(p)\exp \left\{ -  \frac{c'}{M^2}\min \left( \frac{t^2 }{\|H\|_F^2\log(p)},\frac{t}{ \|H\|_2}\right)\right\}.
\end{align*}
\item[{\rm (ii)}] Let $b\in \mathbb{R}^p$ be a deterministic vector. Then the random variable $b^{\intercal} z$ is $ (M\|b\|) $-concentrated.
\end{itemize}
\end{lemma}
\begin{proof}
(i) This result is taken from \cite[Theorem 1.4]{vu2015random}. (ii) Take an arbitrary convex, $1$-Lipschitz function $\phi:\mathbb{R}\rightarrow \mathbb{R}$. We construct a convex function $\psi:\mathbb{R}^p\rightarrow \mathbb{R}$: $\psi(x) = \frac{1}{\|b\|}\phi( b^{\intercal} x),\,  x\in \mathbb{R}^p$. The function $\psi$ is $1$-Lipschitz since for any $x,y\in \mathbb{R}^p$ it holds that
\begin{align*}
|\psi(x)-\psi(y)|=\frac{1}{\|b\|} \left| \phi(b^{\intercal} x)- \phi(b^{\intercal} y) \right|\leq \frac{1}{\|b\|} \|b^{\intercal} x - b^{\intercal} y\| \leq \|x-y\|.
\end{align*}
By the definition of $\psi$ and the fact that $z\in \mathbb{R}^p$ is $M$-concentrated, there exist $c>0$ and $c'>0$ such that
\begin{align*}
\mathbb{P}\left( \left| \phi(b^{\intercal} z)- \mathbb{E}[\phi(b^{\intercal} z)]
\right| \geq t \right) = \mathbb{P}\left( \left| \psi(z)- \mathbb{E}[ \psi(z)]
\right| \geq \frac{t}{\|b\| } \right)\leq c \exp\left( -c'\frac{t^2}{M^2 \|b\|^2}
\right).
\end{align*}
This implies that $b^{\intercal} z$ is $(M\|b\|)$-concentrated, which completes the proof.
\end{proof}

Now we are ready to give the proof of Theorem \ref{thm: Mconcentrated}. For brevity, only a concise outline of the proof is provided.

\begin{proof}[\textbf{Proof of Theorem \ref{thm: Mconcentrated}}]
By mirroring the analysis in the proof of Theorem \ref{thm:main_con}, we have the inequalities \eqref{eq: hat_bar} and \eqref{eq: G}. We need to bound the terms $\frac{1}{dn} \left|\epsilon^{\intercal} V_1 (\hat{y}-y_0)\right|$, $\frac{1}{dn}\left| \epsilon^{\intercal} V_2 V_2^{\intercal}\epsilon \right| $, and $\frac{\gamma_2}{dn} \left| \epsilon^{\intercal} V_2V_2^{\intercal}\hat{\theta}\right|$.

\textbf{Bound for $\frac{1}{dn} \left|\epsilon^{\intercal} V_1 (\hat{y}-y_0)\right|$:} According to the inequality \eqref{eq: epsilony}, we only need to bound $\|z\|_{\infty}$, where $z = (\epsilon^{\intercal} V_1 W^{\dagger})^{\intercal} \in \mathbb{R}^{d|\mathcal{E}|}$. For any $i\in [d|\mathcal{E}|]$, according to Lemma \ref{lemma: Mb}(ii), we can see that the $z_i$ is $(M \|V_1 W^{\dagger} e_i\|)$-concentrated with mean zero. Moreover, for each $i$, we have
\begin{align*}
\|V_1 W^{\dagger} e_i\|\leq \|V_1 W^{\dagger}\|_2 \leq \sqrt{\sigma_{\max}( (W^{\dagger})^{\intercal}  W^{\dagger})}  = \sigma_{\max}( W^{\dagger})=\frac{1}{\sigma_{\min}(B)}.
\end{align*}
By applying Boole's inequality and considering the function
$\psi(\cdot)$ in Definition \ref{def: M_concentrated} as the identity map, we can establish that there exist constants $c_0>0, \ c>0$ such that for any $t>0$,
\begin{align*}
&\mathbb{P}\left(\left\| \epsilon^{\intercal} V_1 W^{\dagger}\right\|_{\infty}\geq t\right) =  \mathbb{P}(\|z\|_{\infty} \geq t) \\
\leq &\sum_{i\in [d|\mathcal{E}|]}\mathbb{P}(|z_i|\geq t) \leq  c d|\mathcal{E}|\exp\left(-c_0^2\frac{t^2\sigma_{\min}^2(B)}{M^2}\right).
\end{align*}
Based on \eqref{eq: epsilony}, and by taking $t = \frac{M}{c_0\sigma_{\min}(B)} \sqrt{2\log (d|\mathcal{E}|)}$ in the above inequality, we have
\begin{align}
&\mathbb{P}\left( \frac{1}{dn} \left|\epsilon^{\intercal} V_1 (\hat{y}-y_0)\right| \geq  \frac{M}{dn c_0 \sigma_{\min}(B) }\sqrt{2d\log(d|\mathcal{E}|)} \sum_{l(i,j)\in \mathcal{E}}  \left\| W^{l(i,j)} (\hat{y}-y_0)\right\| \right) \label{eq: control_2norm_M} \\
& \leq \mathbb{P}\left(\left\| \epsilon^{\intercal} V_1 W^{\dagger}\right\|_{\infty}\geq \frac{M}{c_0\sigma_{\min}(B)} \sqrt{2\log (d|\mathcal{E}|)}  \right)
\leq \frac{c}{d|\mathcal{E}|}.\nonumber
\end{align}

\textbf{Bound for $\frac{1}{dn}\left| \epsilon^{\intercal} V_2 V_2^{\intercal}\epsilon \right| $.} According to Lemma \ref{lemma: Mb}(i), there exist constants $c_1>0$ and $c_2>0$ such that for any $t>0$,
\begin{align*}
\mathbb{P}\left(\frac{1}{dn}\epsilon^{\intercal} V_2 V_2 ^{\intercal} \epsilon \geq \frac{t+ d\kappa_0}{dn}\right) \leq c_1 \log(dn)\exp \left\{ -  \frac{c_2}{M^2}\min \left( \frac{t^2 }{d\kappa_0 \log(dn)},t \right)\right\},
\end{align*}
where the fact that ${\rm tr}(V_2 V_2^{\intercal})={\rm tr}(V_2^{\intercal} V_2) = d\kappa_0$, $\|V_2 V_2^{\intercal}\|_F^2 = {\rm tr}(V_2 V_2^{\intercal}) = d\kappa_0$, and $\|V_2 V_2^{\intercal}\|_2 = \|V_2^{\intercal}V_2\|_2 = 1$ is used. By taking $t = \log(dn) \sqrt{d\kappa_0}$, we have
\begin{align}
\mathbb{P}\left(\frac{1}{dn}\epsilon^{\intercal} V_2 V_2 ^{\intercal} \epsilon \geq \frac{\log(dn)\sqrt{d\kappa_0}+d\kappa_0}{dn}
\right)   \leq c_1 \exp \left\{ \log \log(dn) -  \frac{c_2}{M^2} \log(dn)\right\}. \label{eq: eve_M}
\end{align}

\textbf{Bound for $\frac{\gamma_2}{dn} \left| \epsilon^{\intercal} V_2V_2^{\intercal}\hat{\theta}\right|$.} From Lemma \ref{lemma: Mb}(ii), we have that the random variable $\hat{\theta}^{\intercal} V_2V_2^{\intercal} \epsilon$ is $(M\| V_2V_2^{\intercal}\hat{\theta}\|)$-concentrated with mean zero. This means that there exist $c_3>0,\ c_4>0$ such that
\begin{align*}
\mathbb{P}\left(\frac{1}{dn} \left| \epsilon^{\intercal} V_2V_2^{\intercal}\hat{\theta}\right|>t\right) \leq c_3 \exp\left(-c_4\frac{t^2 d^2 n^2}{ M^2 \| V_2V_2^{\intercal}\hat{\theta}\|^2} \right)\leq c_3 \exp\left(-c_4\frac{t^2 d^{3/2} n}{ M^2 } \right),
\end{align*}
where the last inequality holds since $\|\hat{\theta}\|^2\leq n \sqrt{d}$, which comes from the proof of Theorem \ref{thm:main_con}. Picking $t = M/n^{1/4}$ in the above inequality, we have
\begin{align}
\mathbb{P}\left(\frac{1}{dn} \left| \epsilon^{\intercal} V_2V_2^{\intercal}\hat{\theta}\right|>\frac{M}{n^{1/4}}\right) \leq c_3 \exp\left(-c_4 \sqrt{d^3n} \right). \label{eq: 2norm_epsilon_M}
\end{align}

Thus, when $\gamma_1 \geq\frac{2M}{c_0\sigma_{\min}(B)}\sqrt{2d\log(d|\mathcal{E}|)}$, it follows from the inequalities \eqref{eq: G}, \eqref{eq: control_2norm_M}, \eqref{eq: eve_M}, and \eqref{eq: 2norm_epsilon_M} that
\begin{align*}
\frac{1}{dn}\left| \langle V_1 (\hat{y}-\bar{y}) + V_2 (\hat{z}-\bar{z}),\epsilon\rangle\right| \leq & \ \frac{\gamma_1}{2dn }
\sum_{l(i,j)\in \mathcal{E}} \left\| W^{l(i,j)} (\hat{y}-y_0)\right\| \\
& \ + \frac{\kappa_0}{n} + \frac{\sqrt{d\kappa_0} \log(dn)}{dn} +\frac{\gamma_2M}{n^{1/4}},
\end{align*}
with probability at least 
\begin{align*}
p'_0:= 1 - \frac{c}{d|\mathcal{E}|} -c_1 \exp \left\{ \log \log(dn) -  \frac{c_2}{M^2} \log(dn)\right\}-c_3 \exp\left(-c_4 \sqrt{d^3n} \right).
\end{align*}
Therefore, we can further obtain that
\begin{align*}
\frac{1}{2dn} \|\hat{x}-x_0\|^2 \leq &\  \frac{\kappa_0}{n} + \frac{\sqrt{d\kappa_0} \log(dn)}{dn}  + \frac{\gamma_1}{2dn}\sum_{l(i,j)\in \mathcal{E}}
(1+2 w_{ij}) \left\| D^{l(i,j)} x_0\right\| \\
&\ + \gamma_2 \left[ \frac{M}{n^{1/4}} + \frac{1}{dn}  \sum_{i=1}^n \|\mathcal{M}^i x_0 \|_* \right]
\end{align*}
with probability at least $p'_0$. This completes the proof.
\end{proof}

Although we do not explicitly analyze  prediction consistency as $n$ tends to infinity, the result can be derived analogously to the sub-Gaussian case. To avoid redundancy and conserve space, we omit the detailed derivation.

\section{Implementation Details}
\label{sec:implementation}
In this section, we present the key implementation details of our proposed algorithm, including the implementable stopping criteria and the elementary computations necessary for its efficient execution.

\subsection{Implementable Stopping Criteria of ALM}
\label{sec:implementable_stopping}
To ensure the global and local convergence of Algorithm \ref{alg:alm} when applying Algorithm \ref{alg:ssncg} to solve the ALM subprproblems, we need the practical implementation of the stopping criteria \eqref{eq: stopA}, \eqref{eq: stopB1}, and \eqref{eq: stopB2}. The updating scheme \eqref{eq: updates_SSN} implies that
\begin{align*}
&\quad \Phi_k(x^{k+1},y^{k+1},z^{k+1}) - \inf \Phi_k = \phi_k(x^{k+1})-\inf \phi_k \\
&\leq \langle \nabla \phi_k(x^{k+1}), x^{k+1}- \bar{x}^{k+1}\rangle  - \frac{1}{2}\|x^{k+1}-\bar{x}^{k+1} \|^2\\
&\leq \|\nabla \phi_k(x^{k+1})\|  \|x^{k+1}- \bar{x}^{k+1}\|  - \frac{1}{2}\|x^{k+1}-\bar{x}^{k+1} \|^2 \leq \frac{1}{2} \|\nabla \phi_k(x^{k+1})\|^2,
\end{align*}
where $\phi_k(x):=\inf_{y,z} \Phi_k(x,y,z)$ for all $x\in \mathbb{R}^{dn}$, $\bar{x}^{k+1} = \arg\min \phi_k$, and the first inequality follows from the strong convexity of $\phi_k$. Moreover, it can be seen that $(\nabla \phi_k(x^{k+1}),0,0)\in \partial \Phi_k(x^{k+1},y^{k+1},z^{k+1})$. This implies that the stopping criteria \eqref{eq: stopA}, \eqref{eq: stopB1}, and \eqref{eq: stopB2} can be achieved by the following implementable surrogates, 
respectively:
\begin{align*}
    &\|\nabla \phi_k(x^{k+1})\| \leq \varepsilon_k/\sqrt{\sigma_k},\quad \varepsilon_k\geq 0,\quad \sum_{{\color{blue}k=0}}^{\infty} \varepsilon_k <\infty,\\
    &\|\nabla \phi_k(x^{k+1})\|  \leq \delta_k \sqrt{\sigma_k} \|(\mathcal{D}x^{k+1}-y^{k+1},x^{k+1}-z^{k+1} )\|,\quad \delta_k\geq 0, \quad \sum_{k=0}^{\infty} \delta_k <\infty,\\
    &\|\nabla \phi_k(x^{k+1})\|  \leq \delta'_k  \|(\mathcal{D}x^{k+1}-y^{k+1},x^{k+1}-z^{k+1} )\| ,\quad  0\leq \delta'_k\rightarrow 0,
\end{align*}
where $y^{k+1} = {\rm Prox}_{g/\sigma_k}({\cal D}x^{k+1} + v^k/\sigma_k ) $ and $z^{k+1} = {\rm Prox}_{h/\sigma_k}(x^{k+1} + w^k/\sigma_k ) $.

\subsection{Proximal Mappings and Generalized Jacobians}
\label{sec:proximal_mapping}
To make our proposed double-loop algorithm implementable in practice, we need to provide efficient procedures to compute the proximal mappings of $g(\cdot)$ and $h(\cdot)$, together with their generalized Jacobians.

For given $y\in \mathbb{R}^{d |\mathcal{E}|}$ and $\nu>0$, the proximal mapping ${\rm Prox}_{\nu g}(y)$ can be computed as:
\begin{align*}
\mathcal{P}_{l(i,j)}{\rm Prox}_{\nu g}(y) = {\rm Prox}_{\nu \gamma_1 w_{ij}\|\cdot\|_2}(\mathcal{P}_{l(i,j)}y),\quad l(i,j)\in \mathcal{E},
\end{align*}
and the generalized Jacobian $\partial {\rm Prox}_{\nu g}(y): \mathbb{R}^{d |\mathcal{E}|}\rightrightarrows \mathbb{R}^{d |\mathcal{E}|}$ takes the form as:
\begin{align*}
\partial {\rm Prox}_{\nu g}(y) = \left\{
Q \in \mathbb{R}^{d |\mathcal{E}|\times d |\mathcal{E}|}\middle\vert
\begin{aligned}
&Q = {\rm Diag}(\{Q_{l(i,j)}\}_{l(i,j)\in \mathcal{E}}),\\
&Q_{l(i,j)} \in \partial {\rm Prox}_{\nu \gamma_1 w_{ij}\|\cdot\|_2}(\mathcal{P}_{l(i,j)}y) , \ l(i,j)\in \mathcal{E}
\end{aligned}
\right\}.
\end{align*}
Here, for any $u\in\mathbb{R}^n$ and $\eta>0$, we have that
\begin{align*}
&{\rm Prox}_{\eta\|\cdot\|_2}(u) = \left\{\begin{aligned}
& u - \frac{\eta}{\|u\| } u  && \mbox{if } \|u\|>\eta \\
& 0 && \mbox{otherwise}
\end{aligned}
\right. ,\\
&\partial {\rm Prox}_{\eta\|\cdot\|_2}(u)
= \left\{
\begin{aligned}
& \left\{\left( 1-\frac{\eta}{\|u\|}\right)I_n + \frac{\eta uu^{\intercal}}{\|u\|^3} \right\} && \mbox{if }\|u\|>\eta\\
& \left\{ t \frac{ uu^{\intercal}}{\eta^2}  \middle\vert 0\leq t\leq 1 \right\} && \mbox{if }\|u\|=\eta\\
&\{0\} && \mbox{otherwise}
\end{aligned}
\right..
\end{align*}
Moreover, according to \cite[Lemma 2.1]{zhang2020efficient}, ${\rm Prox}_{\eta\|\cdot\|_2}$ is strongly semismooth with respect to $\partial {\rm Prox}_{\eta\|\cdot\|_2}$.

For given $z\in \mathbb{R}^{d n}$ and $\nu>0$, from the definition of the function $h(\cdot)$, it can be seen that
\begin{align*}
{\rm Prox}_{\nu h}(z) &= \underset{x\in \mathbb{R}^{d n}}{\arg\min}  \Big\{
\frac{1}{2}\|x-z\|^2 + \nu \gamma_2 \sum_{i=1}^n \|{\cal M}^{i} x\|_*\Big\}\\
&=\underset{x\in \mathbb{R}^{d n}}{\arg\min}  \Big\{
\sum_{i=1}^n \left(\frac{1}{2}\|{\cal M}^{i} x-{\cal M}^{i} z\|^2 + \nu \gamma_2  \|{\cal M}^{i} x\|_*\right)\Big\}.
\end{align*}
This implies that ${\rm Prox}_{\nu h}(z)\in \mathbb{R}^{d n}$ satisfies that 
\begin{align*}
{\cal M}^{i}  {\rm Prox}_{\nu h}(z) =   \underset{U\in \mathbb{R}^{d_1\times d_2}}{\arg\min} \ \Big\{ \frac{1}{2}\|U-{\cal M}^{i} z\|_F^2 + \nu \gamma_2  \|U\|_*\Big\}  = {\rm Prox}_{\nu\gamma_2\|\cdot\|_*}({\cal M}^{i} z), \ i \in [n].
\end{align*}
Therefore, we need to characterize the proximal mapping and the associated generalized Jacobian of the nuclear norm, which is stated in the following proposition.

\begin{proposition}\label{prop: nuclearnorm}
Assume $d_1\geq d_2$. Given $\gamma>0$ and $G\in \mathbb{R}^{d_1\times d_2}$ with the singular value decomposition $G = U[{\rm Diag}(\sigma);0] V^{\intercal}$, where $U\in \mathbb{R}^{d_1\times d_1}$, $V\in \mathbb{R}^{d_2\times d_2}$ are orthogonal matrices, $\sigma \in \mathbb{R}^{d_2}$ is the vector of the singular values of $G$, with $\sigma_1\geq \sigma_2\geq \cdots\geq \sigma_{d_2}\geq 0$. Let the scalar function $\varpi_{\gamma}(\cdot)$ be defined as $\varpi_{\gamma}(t) = (t-\gamma)_{+}$ for any $t\geq 0$. Then we have the following statements hold.
\begin{itemize}[left=5pt, labelsep=3pt, itemsep=2pt]	
\item[{\rm (a)}] The proximal mapping associated with $\gamma \|\cdot\|_*$ at $G$ can be computed as
\begin{align*}
{\rm Prox}_{\gamma \|\cdot\|_*}(G) = U[{\rm Diag}(\varpi_{\rho}(\sigma)); 0] V^{\intercal},
\end{align*}
where $\varpi_{\rho}(\sigma)\in \mathbb{R}^{d_2}$ is a vector such that its $i$-th component is given by $\varpi_{\rho}(\sigma_i)$.

\item[{\rm (b)}] The function ${\rm Prox}_{\gamma \|\cdot\|_*}$ is strongly semismooth everywhere in $\mathbb{R}^{d_1\times d_2}$ with respect to $\partial {\rm Prox}_{\gamma \|\cdot\|_*}$.

\item[{\rm (c)}] Decompose $U\in \mathbb{R}^{d_1\times d_1}$ into the form $U = [U_1\ U_2]$, where $U_1\in \mathbb{R}^{d_1\times d_2}$ and $U_2 \in \mathbb{R}^{d_1\times (d_1-d_2)}$. Define the index sets
\begin{align*}
&\alpha = \{1,\cdots,d_2\},\quad \gamma = \{d_2+1,\cdots,2d_2\},\quad \beta = \{2d_2+1,\cdots,d_1+d_2 \},\\
&\alpha_1 = \{i\in [d_2] \mid \sigma_i > \gamma\},\quad \alpha_2 = \{i\in [d_2] \mid \sigma_i = \gamma\},\quad
\alpha_3 = \{i\in [d_2] \mid \sigma_i < \gamma\}.
\end{align*}
Then we can construct $\mathcal{Q}^0\in \partial {\rm Prox}_{\gamma \|\cdot\|_*}(G)$, where the operator $\mathcal{Q}^0:\mathbb{R}^{d_1\times d_2}\rightarrow \mathbb{R}^{d_1\times d_2}$ is defined as: 
\begin{align*}
\mathcal{Q}^0[W] = \left[ U_1 \left( \Gamma_{\alpha \alpha} \circ \left( \frac{W_1+W_1^{\intercal}}{2}\right) + \Gamma_{\alpha \gamma} \circ \left( \frac{W_1-W_1^{\intercal}}{2}\right) \right)  + U_2 \left( \Gamma_{\alpha \beta}\circ W_2\right) \right]V^{\intercal},
\end{align*}
for any $W\in \mathbb{R}^{d_1\times d_2}$. Here $W_1 = U_1^{\intercal} W V$, $W_2 = U_2^{\intercal} W V$, and $\Gamma_{\alpha \alpha}^{0}\in \mathbb{S}^{d_2}$, $\Gamma_{\alpha \gamma}\in \mathbb{S}^{d_2}$, $\Gamma_{\beta \alpha }\in \mathbb{R}^{(d_1-d_2)\times d_2}$ are defined as
\setlength{\arraycolsep}{1pt}
\begin{align*}
\Gamma_{\alpha \alpha}  = \begin{pmatrix}
1_{\alpha_1 \alpha_1} & 1_{\alpha_1 \alpha_2} & \tau_{\alpha_1 \alpha_3} \\
1_{\alpha_2 \alpha_1} & 0 & 0\\
\tau_{\alpha_1 \alpha_3}^{\intercal} & 0 & 0
\end{pmatrix}, \quad \Gamma_{\alpha \gamma} = \begin{pmatrix}
\omega_{\alpha_1\alpha_1} & \omega_{\alpha_1\alpha_2} & \omega_{\alpha_1\alpha_3}\\
\omega_{\alpha_1\alpha_2}^{\intercal}  & 0 & 0\\
\omega_{\alpha_1\alpha_3}^{\intercal}  & 0 & 0
\end{pmatrix},\quad
\Gamma_{\beta \alpha} = \begin{pmatrix}
\mu_{\beta \alpha_1} & 0
\end{pmatrix}
\end{align*}
with $\tau_{\alpha_1 \alpha_3}\in \mathbb{R}^{|\alpha_1|\times |\alpha_3|}$, $\omega_{\alpha_1\alpha}\in \mathbb{R}^{|\alpha_1|\times d_2}$, $\mu_{\beta \alpha_1 }\in \mathbb{R}^{(d_1-d_2) \times |\alpha_1| }$ taking the forms of
\begin{align*}
&\tau_{ij} = \frac{\sigma_i - \gamma}{\sigma_i -\sigma_j}, \quad \mbox{for } i \in \alpha_1,\ j\in \alpha_3,\\
&\omega_{ij} = \frac{\sigma_i-\gamma + (\sigma_j-\gamma)_{+}}{\sigma_i+\sigma_j},\quad \mbox{for } i \in \alpha_1,\ j\in \alpha,\\
& \mu_{ij} = \frac{\sigma_j-\gamma}{\sigma_j},\quad \mbox{for } i\in \beta,\ j\in \alpha_1, \quad \mbox{where }\beta = \{1,\cdots,d_1-d_2\}.
\end{align*}
\end{itemize}
\end{proposition}

\begin{proof}
(a) follows from the results in \cite{cai2010singular,ma2011fixed}, and (b) comes from \cite[Theorem 2.1]{jiang2014partial}. As for (c), according to \cite[Lemma 2.3.6 and Proposition 2.3.7]{zhe2009study} and the fact that for any $U\in \mathbb{R}^{d_1\times d_2}$,
\begin{align*}
\partial_B {\rm Prox}_{\gamma \|\cdot\|_*}(U)\subseteq \partial {\rm Prox}_{\gamma \|\cdot\|_*}(U) = {\rm conv}(\partial_B {\rm Prox}_{\gamma \|\cdot\|_*}(U)),
\end{align*}
we can have the desired conclusion.
\end{proof}

A direct result from Proposition \ref{prop: nuclearnorm} is as follows, which is used in the proof of our main theorems in Section \ref{sec: finite_prediction}.

\begin{proposition}\label{prop: nuc2}
For any $X,Y \in \mathbb{R}^{d_1\times d_2}, \ d_1 \geq d_2 $ and $\gamma > 0$, it holds that
\begin{itemize}
\item[{\rm (a)}] $\|X - {\rm Prox}_{\gamma \|\cdot\|_*}(X) \|_F \leq \gamma \sqrt{d_2} $;
\item[{\rm (b)}] if $ {\rm Prox}_{\gamma \|\cdot\|_*}(X) = {\rm Prox}_{\gamma \|\cdot\|_*}(Y)  $, then $\|X - Y\|_F \leq \gamma \sqrt{d_2 - r}$, where $r = {\rm rank}({\rm Prox}_{\gamma \|\cdot\|_*}(X))$.
\end{itemize}
\end{proposition}
\begin{proof}
The proof is based on  Proposition \ref{prop: nuclearnorm}(a), by which we characterize the singular values of involved matrices.
Consider the singular value decomposition ${\rm Prox}_{\gamma \|\cdot\|_*}(X)  = U \Sigma V^{\intercal}$, where $U\in \mathbb{R}^{d_1\times d_1}$, $V\in \mathbb{R}^{d_2\times d_2}$ are orthogonal matrices, and $\Sigma\in \mathbb{R}^{d_1\times d_2}$ is a rectangular diagonal matrix with singular values
\begin{align*}
\sigma_1\geq \sigma_2\geq \cdots\geq \sigma_{r} > \underbrace{0 =\cdots= 0}_{d_2-r}
\end{align*}
on the diagonal, where $r$ is the rank of ${\rm Prox}_{\gamma \|\cdot\|_*}(X) $. By Proposition~\ref{prop: nuclearnorm}(a), we have
$
X  = U \Sigma' V^{\intercal}
$,
$\Sigma'\in \mathbb{R}^{d_1\times d_2}$ is a rectangular diagonal matrix with singular values
$
\sigma_1+\gamma \geq \cdots\geq \sigma_{r} +\gamma > \sigma_{r+1}\geq \cdots\geq \sigma_{d_2}
$
on the diagonal, which satisfy $\gamma\geq \sigma_{r+1}\geq \cdots\geq \sigma_{d_2}\geq 0$. Then we have
$$
\|X - {\rm Prox}_{\gamma \|\cdot\|_*}(X) \|_F = \|U (\Sigma - \Sigma') V^{\intercal}\|_F = \sqrt{r \gamma^2 + \sigma_{r+1}^2 + \cdots + \sigma_{d_2}^2}
\leq \gamma \sqrt{d_2}.
$$

Secondly, if ${\rm Prox}_{\gamma \|\cdot\|_*}(X) ={\rm Prox}_{\gamma \|\cdot\|_*}(Y)  = U \Sigma V^{\intercal}$, we shall further have $
Y  = U \widetilde{\Sigma} V^{\intercal}
$,
$\widetilde{\Sigma}\in \mathbb{R}^{d_1\times d_2}$ is a rectangular diagonal matrix with singular values
$
\sigma_1+\gamma \geq \cdots\geq \sigma_{r} +\gamma > \tilde{\sigma}_{r+1}\geq \cdots\geq \tilde{\sigma}_{d_2}
$
on the diagonal,  which satisfy $\gamma\geq \tilde{\sigma}_{r+1}\geq \cdots\geq \tilde{\sigma}_{d_2}\geq 0$.
Then we have
$$
\|X - Y \|_F \leq \|U (\Sigma' - \widetilde{\Sigma}) V^{\intercal}\|_F = \sqrt{(\sigma_{r+1} - \tilde{\sigma}_{r+1})^2  + \cdots + (\sigma_{d_2} - \tilde{\sigma}_{d_2})^2}
\leq \gamma \sqrt{d_2-r}.
$$
The proof is completed.
\end{proof}

\bibliographystyle{siamplain}
\bibliography{references}
\end{document}